\documentclass[11pt,reqno]{article}

\usepackage[margin=1in]{geometry}
\usepackage{amsmath,amsthm,amssymb}

\usepackage[colorlinks=true, pdfstartview=FitV, linkcolor=blue,citecolor=blue, urlcolor=blue]{hyperref}

\usepackage[abbrev,lite,nobysame]{amsrefs}
\usepackage{times}
\usepackage[usenames,dvipsnames]{color}

\usepackage{esint}

\usepackage{marginnote}

\usepackage{mathtools}
\usepackage{enumitem}
\mathtoolsset{showonlyrefs=true}

\usepackage[compact]{titlesec}

\usepackage[title]{appendix}

\usepackage{comment}

\newcommand{\ep}{\epsilon}

\newcommand{\leqc}{\lesssim}

\newcommand{\geqc}{\gtrsim}

\newcommand{\grad}{\nabla}
\newcommand{\MalD}{\cD}
\newcommand{\Wbf}{\mathbf{W}}
\newcommand{\fM}{\mathfrak{M}}

\newcommand{\norm}[1]{\left|\left| #1 \right|\right|}
\newcommand{\abs}[1]{\left| #1 \right|}
\newcommand{\set}[1]{\left\{ #1 \right\}}

\newcommand{\R}{\mathbb{R}}
\newcommand{\Real}{\R}

\newcommand{\Z}{\mathbb{Z}}
\newcommand{\T}{\mathbb{T}}
\newcommand{\K}{\mathbb{K}}
\renewcommand{\S}{\mathbb{S}}

\newcommand{\tensor}{\otimes}

\newcommand{\SL}{\mathrm{SL}}
\renewcommand{\sl}{\mathrm{sl}}

\newcommand{\cL}{\mathcal{L}}
\newcommand{\cC}{\mathcal{C}}
\newcommand{\cH}{\mathcal{H}}
\newcommand{\cM}{\mathcal{M}}
\newcommand{\cW}{\mathcal{W}}
\newcommand{\cY}{\mathcal{Y}}
\newcommand{\cD}{\mathcal{D}}
\newcommand{\cK}{\mathcal K}

\newcommand{\Hbf}{{\bf H}}

\newcommand{\dee}{\mathrm{d}}
\newcommand{\ds}{\dee s}
\newcommand{\dt}{\dee t}

\newcommand{\dx}{\dee x}

\newcommand{\dy}{\dee y}

\newcommand{\dz}{\dee z}

\newcommand{\dr}{\dee r}

\DeclareMathOperator{\supp}{\mathrm{supp}}

\DeclareMathOperator{\Id}{\mathrm{Id}}

\DeclareMathOperator{\Span}{\mathrm{span}}
\DeclareMathOperator{\Range}{\mathrm{Ran}}

\usepackage{bbm}
\newcommand{\1}{\mathbbm{1}}
\newcommand{\F}{\mathcal{F}}
\renewcommand{\P}{\mathbf{P}}

\newcommand{\E}{\mathbf{E}}
\newcommand{\EE}{\mathbf E}
\newcommand{\PP}{\mathbf P}
\newcommand{\Leb}{\operatorname{Leb}}

\newtheorem{theorem}{Theorem}[section]
\newtheorem{proposition}[theorem]{Proposition}
\newtheorem{corollary}[theorem]{Corollary}
\newtheorem{lemma}[theorem]{Lemma}
\newtheorem*{lemma*}{Lemma}
\newtheorem{claim}[theorem]{Claim}

\newtheorem{assumption}{Assumption}
\newtheorem{system}{System}

\theoremstyle{definition}
\newtheorem{definition}[theorem]{Definition}
\newtheorem{remark}[theorem]{Remark}
\newtheorem{example}[theorem]{Example}

\numberwithin{equation}{section}
    
\begin{document}

\title{Lagrangian chaos and scalar advection in stochastic fluid mechanics}
\author{Jacob Bedrossian\thanks{\footnotesize Department of Mathematics, University of Maryland, College Park, MD 20742, USA \href{mailto:jacob@math.umd.edu}{\texttt{jacob@math.umd.edu}}. J.B. was supported by NSF CAREER grant DMS-1552826 and NSF RNMS \#1107444 (Ki-Net)} \and Alex Blumenthal\thanks{\footnotesize Department of Mathematics, University of Maryland, College Park, MD 20742, USA \href{mailto:alex123@math.umd.edu}{\texttt{alex123@math.umd.edu}}. This material is based upon work supported by the National Science Foundation under Award No. DMS-1604805.} \and Sam Punshon-Smith\thanks{\footnotesize Division of Applied Mathematics,  Brown University, Providence, RI 02906, USA \href{mailto:punshs@brown.edu}{\texttt{punshs@brown.edu}}. This material is based upon work supported by the National Science Foundation under Award No. DMS-1803481.}}

\maketitle

\begin{abstract}
We study the Lagrangian flow associated to velocity fields arising from various models of fluid mechanics subject to white-in-time, $H^s$-in-space stochastic forcing in a periodic box. 
We prove that in many circumstances, these flows are chaotic, that is, the top Lyapunov exponent is strictly positive. 
Our main results are for the Navier-Stokes equations on $\T^2$ and the hyper-viscous regularized Navier-Stokes equations on $\T^3$ (at arbitrary Reynolds number and hyper-viscosity parameters), subject to forcing which is non-degenerate at high frequencies. 
As an application, we study statistically stationary solutions to the passive scalar advection-diffusion equation driven by these velocities and subjected to random sources. 
The chaotic Lagrangian dynamics are used to prove a version of anomalous dissipation in the limit of vanishing diffusivity, which in turn, implies that the scalar satisfies Yaglom's law of scalar turbulence -- the analogue of the Kolmogorov 4/5 law.  Key features of our study are the use of tools from ergodic theory and random dynamical systems, namely the Multiplicative Ergodic Theorem and a version of Furstenberg's Criterion, combined with hypoellipticity via Malliavin calculus and approximate control arguments. 
\end{abstract}

\setcounter{tocdepth}{1}
{\small\tableofcontents}

%% Introduction
\section{Introduction and outline}\label{sec:Intro}
%!TEX root = master.tex

\newcommand{\eps}{\varepsilon}

In this paper, we study the stochastic flow of diffeomorphisms $\phi^t : \T^d \to \T^d, t \geq 0$ defined by the random ODE
\begin{align}
\frac{\dee}{\dt} \phi^t (x) & = u_t(\phi^t (x)), \quad \quad \phi^0 (x) = x \label{eq:xtintro} \, .
\end{align}
Here, the random velocity field $u_t : \T^d \to \R^d$ at time $t > 0$ evolves according to one of several stochastically-forced fluid mechanics models, for example, the 2D Navier-Stokes at fixed (but arbitrary) inverse Reynolds number $\nu > 0$ on $\mathbb T^2$: 
\begin{align}
\partial_t u_t + u_t \cdot \grad u_t = -\grad p_t + \nu \Delta u_t + Q \dot W_t, \quad\quad \grad \cdot u_t = 0, 
\end{align}
where $p_t$ denotes the pressure at time $t$ and $Q \dot W_t$ is a white-in-time, colored-in-space Gaussian process described more precisely below (Section \ref{subsubsec:noiseProcess}).

It is expected \cite{CrisantiEtAl1991,FalkovickEtAl01} that when $u_t$ evolves according to either the Stokes equations (i.e., zero Reynolds number) 
or Navier-Stokes at arbitrary Reynolds number, the corresponding Lagrangian flows will generically be chaotic in terms of sensitivity with respect to initial conditions. 
This phenomenon is sometimes referred to as \emph{Lagrangian chaos}.
The primary objective of the present paper is to verify this by proving that the dynamical system defined via \eqref{eq:xtintro} 
possesses a strictly positive Lyapunov exponent: that is, there exists a constant $\lambda > 0$, depending on the parameters
of the relevant Stokes or Navier-Stokes equation, such that for every $x \in \T^d$ and any initial vector field in 
the support of $\mu$, the stationary measure of the stochastic fluid equation, we have that
\begin{align}
\lim_{t \rightarrow \infty} \frac{1}{t}\log \abs{D_x \phi^t} = \lambda > 0 \qquad \text{holds with probability 1}.
\end{align}
Here, $D_x \phi^t$ refers to the Jacobian matrix of $\phi^t : \T^d \to \T^d$ taken at $x$. This implies that almost everywhere in $\T^d$ and with probability 1, nearby particles are separated at an exponentially fast rate by the Lagrangian flow $\phi^t$. 

We further apply our Lagrangian chaos results to the `scalar turbulence' problem in the Batchelor regime (see e.g. \cite{Batchelor59,ShraimanSiggia00,FalkovickEtAl01} and the references therein for physics literature).
In particular, we prove that statistically stationary solutions of the passive scalar advection-diffusion equation (with random velocity fields given by the stochastic fluid models) obey the fundamental scaling law predicted by Yaglom in 1949 \cite{Yaglom49} in the vanishing diffusivity limit. Yaglom's law is the passive scalar analogue of the Kolmogorov 4/5 law -- or perhaps more accurately, the closely related 4/3 law; see \cite{Frisch1995} and the references therein.  
To our knowledge, this is the first rigorous proof of any scaling laws of this type for velocities arising from the Stokes or Navier-Stokes equations. 
See Section \ref{sec:Results} below for rigorous statements. 

\subsection{Setup and assumptions}\label{subsection:setupAssumptions}

\subsubsection{Probabilistic framework}\label{subsubsec:noiseProcess} 
Let $\T^d = [0,2\pi]^d$ denote the period box. Following the convention used in \cite{E2001-lg}, we define the following real Fourier basis for functions on $\T^d$ by
 \[\label{eq:Fourier-Basis}
 e_k(x) = \begin{cases}
 \sin(k\cdot x), \quad& k \in \Z^d_+\\
 \cos(k\cdot x),\quad& k\in \Z^d_-,
 \end{cases}
 \]
where $\Z_+^d = \{(k_1,k_2,\ldots k_d)\in \Z^d : k_d >0\}\cup\{(k_1,k_2,\ldots k_d)\in \Z^d \,:\, k_1>0, k_d=0\}$ and $\Z_-^d = - \Z_+^d$. 
We set $\Z^d_0 := \Z^d \setminus \set{0,\ldots, 0}$ and define $\{\gamma_k\}_{k\in \Z^d_0}$ a collection of full rank $d\times (d-1)$ matrices satisfying $\gamma^\top_k k = 0$, $\gamma_k^\top\gamma_k=\Id$, and $\gamma_{-k} = - \gamma_k$. Note that in dimension $d=2$, $\gamma_k$ is just a vector in $\R^2$ and is therefore given by $\gamma_k = \pm k^{\perp}/|k|$. In dimension $3$, the matrix $\gamma_k$ defines a pair of orthogonal vectors $\gamma_k^{1},\gamma_k^2$ that span the space perpendicular to $k$. 

Define
\[
\Wbf = \set{u \in L^2(\T^d, \R^d) : \int u \, \dx = 0, \grad \cdot u = 0}
\]
to be the Hilbert space of square integrable, mean-zero, divergence-free vector fields on $\T^d$ and let $W_t$ be a cylindrical Wiener process on $\Hbf$ defined by
\[
	W_t = \sum_{k\in\Z^d_0} e_k\gamma_k W^k_t,
\]
where $\{W^k_t\}_{k\in \Z^d_0}$ are a family of independent $(d-1)$-dimensional Wiener processes on a common canonical filtered probability space $(\Omega,\mathcal{F}, (\mathcal{F}_t),\P)$. Note that $W_t$ is divergence free by the fact that $\gamma^\top_k k = 0$.

Let $Q$ be a Hilbert-Schmidt operator on $\Wbf$ with singular values $\{q_k\}_{k\in \Z^d_0}$ satisfying the coloring assumption
\begin{equation}\label{eq:coloring-qk}
	q_k \lesssim |k|^{-\alpha}
\end{equation}
for an arbitrary, fixed $\alpha > \frac{5d}{2}$. Additionally, fix an arbitrary $\sigma>0$ satisfying 
\begin{align}\label{eqn:sigmaConstraint}
\frac{d}{2} + 2 < \alpha - 2(d-1) < \sigma < \alpha-\frac{d}{2}
\end{align}
and define the Hilbert space
\begin{align}
\Hbf = \set{u \in H^\sigma(\T^d, \R^d) : \int u \, \dx = 0, \grad \cdot u = 0},
\end{align}
where $H^\sigma(\T^d,\R^d)$ denotes the space of Sobolev regular vector-fields on $\T^d$ (see Section \ref{subsec:notation} for a precise meaning when $\sigma$ is not an integer).
For the entirety of this paper, we will consider a stochastic forcing $Q\dot{W}_t$, which takes the form for each $t>0$ and $x\in\T^d$
\[
 Q\dot{W}_t(x) = \sum_{k\in \Z^d_0}q_ke_{k}(x)\gamma_k\dot{W}^k_t.
 \]

\begin{remark}
The coloring assumption \eqref{eq:coloring-qk} and the upper bound on $\sigma$ in \eqref{eqn:sigmaConstraint} ensures that $\{|k|^{\sigma} q_k\}$ is square summable over $\Z^d_0$ and therefore $QW_t$ belongs to $\Hbf$ almost surely.
See Remark \ref{rem:sigmaStrongFeller} for a discussion of the lower bound on $\sigma$ specified in \eqref{eqn:sigmaConstraint}.
\end{remark}

We will also consider the following non-degeneracy condition on the low modes of the forcing. Define $\mathcal{K}$ to be the set of $k \in \Z^d_0$ such that $q_k \neq 0$. 
\begin{assumption}[Low mode non-degeneracy] \label{a:lowms}
Assume $k \in \mathcal{K}$ if $\abs{k}_{\infty} = 1$.
\end{assumption}
Above, for $k = (k_i)_{i = 1}^d \in \Z^d$ we write $|k|_{\infty} = \max_{i} |k_i|$. For several of the finite-dimensional models discussed in this paper, Assumption \ref{a:lowms} is actually stronger than needed, i.e., the results we obtain hold with forcing on fewer modes. Sharper sufficient conditions will be specified as we go along. 

For the infinite-dimensional models, we will in addition invoke the following nondegeneracy condition on all sufficiently high modes past some arbitrary finite cutoff.
\begin{assumption}[High mode non-degeneracy] \label{a:Highs}
There exists an $L > 0$ and an $\alpha \in (\frac{5d}{2},\infty)$ such that
\begin{align}
q_k \geqc \abs{k}^{-\alpha} \quad \textup{for} \quad \abs{k}_\infty \geq L. 
\end{align}
\end{assumption}
See Remark \ref{rmk:WhyHypo} for more discussion on Assumption \ref{a:Highs}.

\subsubsection{Fluid mechanics models}

Below, we write $\Hbf_{\mathcal K} \subset \Hbf$ for the subspace spanned by the Fourier modes $k \in \mathcal K$.

\begin{system} \label{sys:2DStokes} 
We refer to the \emph{Stokes system} in $\mathbb T^d$ ($d = 2,3$) as the following stochastic PDE for initial $u_0 \in \Hbf_{\mathcal K}$: 
\begin{equation} \label{eq:stokes-2d}
\begin{cases}
\,\partial_t u_t =- \grad p_t + \Delta u_t + Q \dot W_t \\ 
\,\grad \cdot u_t = 0 
\end{cases},
\end{equation}
where $Q$ satisfies Assumption \ref{a:lowms} and $\mathcal K$ is finite.
\end{system}

The assumption that $\mathcal K$ be finite is both natural (since only a few modes are required by Assumption \ref{a:lowms}), and expedient, since
System \ref{sys:2DStokes} is effectively a finite-dimensional 
Ornstein-Uhlenbeck process. However, the methods of this paper applied to 
Systems \ref{sys:NSE}, \ref{sys:3DNSE} easily extend to cover 
 System \ref{sys:2DStokes} when $\mathcal{K}$ is infinite and $Q$ satisfies
 Assumption \ref{a:Highs}. For more details, see Remark \ref{rmk:2Dstokes}.

\begin{system} \label{sys:Galerkin}
We refer to the \emph{Galerkin-Navier-Stokes system} in $\mathbb T^d$ ($d = 2,3$) as the following stochastic ODE for $u_0 \in \Hbf_N$:
\begin{equation}
\begin{cases}
\,\partial_t u_t + \Pi_N\left(u_t\cdot \grad u_t + \grad p_t \right) = \nu\Delta u_t + \Pi_N Q \dot W_t \\ 
\,\grad \cdot u_t = 0
\end{cases}
\end{equation}
where $Q$ satisfies Assumption \ref{a:lowms}; $N \geq 3$ is an integer; $\Pi_N$ denotes the projection to 
Fourier modes with $| \cdot|_{\infty}$ norm $\leq N$; $\Hbf_N$ denotes the span of the first $N$ Fourier modes; and $\nu > 0$ is fixed and arbitrary.
\end{system}

\begin{system} \label{sys:NSE}
We refer to the \emph{2D Navier-Stokes system} as the following stochastic PDE for $u_0 \in \Hbf$ on $\T^2$: 
\begin{equation}
\begin{cases}
\,\partial_t u_t + u_t \cdot \grad u_t =- \grad p_t + \nu \Delta u_t + Q \dot W_t \\ 
\,\grad \cdot u_t = 0
% u_0 = u \in \Hbf \, , 
\end{cases},
\end{equation}
where $Q$ satisfies Assumptions \ref{a:lowms} and \ref{a:Highs}. Here $\nu > 0$ is arbitrary and fixed.   
\end{system}

\begin{system} \label{sys:3DNSE}
We refer to the \emph{3D hyper-viscous Navier-Stokes system} as the following stochastic PDE for $u_0 \in \Hbf$ on $\T^3$: 
\begin{equation}
\begin{cases}
\,\partial_t u_t + u_t \cdot \grad u_t =- \grad p_t + \nu \Delta u_t - \eta \Delta^{2} u_t + Q \dot W_t \\ 
\,\grad \cdot u_t = 0
\end{cases},
\end{equation}
where $Q$ satisfies Assumptions \ref{a:lowms} and \ref{a:Highs}. Here $\nu,\eta > 0$ are arbitrary and fixed. 
\end{system}
We emphasize that for System \ref{sys:Galerkin} there is no relationship between the viscosity $\nu$ 
 the Galerkin cutoff $N$, or the parameters $\{ q_k \}$ determining the noise process $Q \dot W_t$. Similarly, for 
 System \ref{sys:NSE}, the parameter $\nu > 0$ is independent of the parameters $L, \alpha, \{ q_k\}$ 
in Assumption \ref{a:Highs} specifying the noise process $Q \dot W_t$ (and similarly for the arbitrary parameters $\eta,\nu > 0$ for System \ref{sys:3DNSE}).

\subsubsection{Well-posedness and stationary measures for Systems \ref{sys:2DStokes} -- \ref{sys:3DNSE}}
Recall the following well-posedness theorem for the systems we consider. For 2D Navier-Stokes as in System \ref{sys:NSE}, see, e.g., \cite{DPZ96,KS}; the hyper-viscous case follows similarly. 
For uniqueness of the stationary measure for 2D Navier-Stokes, see, e.g., \cite{HM06}, although under Assumption \ref{a:Highs} uniqueness follows from other methods (see Remark \ref{rmk:WhyHypo} below).
We are unaware of a work specifically proving uniqueness of the stationary measure for System \ref{sys:3DNSE}, however, under Assumption \ref{a:Highs} our work proves that this is the case (see also the work of \cite{RomitoXu11}). 
For the finite-dimensional Systems \ref{sys:2DStokes} and \ref{sys:Galerkin}, well-posedness follows from classical SDE theory (see e.g. \cite{Oksendal03,DaPrato14}). Uniqueness of the stationary measure for System \ref{sys:2DStokes} is likewise classical (it being effectively a finite-dimensional Ornstein-Uhlenbeck process), while uniqueness of the stationary measure for System \ref{sys:Galerkin} follows from \cite{E2001-lg,Romito2004-rc}.
For a more precise well-posedness statement, see Section \ref{subsec:wellPosedApp}.

\begin{proposition}[See e.g. \cite{KS}] \label{prop:WP}
For each of Systems \ref{sys:2DStokes}--\ref{sys:3DNSE} and all sufficiently regular initial data $u$,
there exists a global-in-time, $\PP$-a.s. unique, $\mathcal{F}_t$-adapted mild solution $(u_t)$ satisfying $u_0 = u$. 
Moreover, $(u_t)$ defines a Feller Markov process in the usual way.
In each case, the corresponding Markov semigroup has a unique (and hence ergodic; see Definition \ref{defn:PtmuInvariant}) stationary probability measure on $\Hbf$ which we denote $\mu$ (in all cases, as a slight abuse of notation).
\end{proposition}

With the $(u_t)$ process on $\Hbf$ as in Proposition \ref{prop:WP}, we write $\phi^t$ for the stochastic
flow of diffeomorphisms solving \eqref{eq:xtintro}.  This gives rise to an $\mathcal{F}_t$-adapted, Feller Markov process $(u_t, x_t)$ on 
$\Hbf \times \T^d$ defined by $x_t = \phi^t (x_0)$, where $x_0 = x$ for fixed initial $x \in \T^d$. 
We refer to $(u_t, x_t)$ as the \emph{Lagrangian flow process} or \emph{Lagrangian process}. 
A simple check verifies that $\mu \times \Leb$ is a stationary measure for the Lagrangian process, where $\Leb$ stands
for Lebesgue measure on $\T^d$. 
Note that ergodicity of $\mu$ does \emph{not} imply ergodicity of $\mu \times \Leb$. 
Indeed, consider the example $\mathcal{K} = \set{(1,0)}$ with the 2D Stokes equations \eqref{eq:stokes-2d}: in that case, one can directly check that $\mu \times \Leb$ is not ergodic. 
One of the purposes of Assumption \ref{a:lowms} is to rule out such degeneracies.     

\begin{remark} \label{rmk:WhyHypo}
Our methods currently require some regularity properties that we do not know how to verify without the strong 
Feller property of the Markov semigroup associated to the $(u_t, x_t)$ process (see definition \ref{defn:strongFeller}). 
In particular, the asymptotically strong Feller property \cite{HM06,HM11} is not enough for our purposes. 
It is for this reason that when treating Systems \ref{sys:NSE} and \ref{sys:3DNSE}, we must assume nondegeneracy of the forcing in the high modes as in Assumption \ref{a:Highs}. 
As in \cite{FM95, EH01}, a straightforward modification of the methods in 
this paper can be made to prove the strong Feller property when, in Assumption \ref{a:Highs}, 
the power laws in the lower and upper bound on $|q_k|$ differ by a small constant $< 1$. 
\end{remark}

\begin{remark}
Note that the forcing on the $(u_t,x_t)$ process is necessarily degenerate, even if we had completely non-degenerate noise acting on the velocity. This is the main technical challenge in proving the strong Feller property.
\end{remark}
    
\subsection{Statement and discussion of results} \label{sec:Results}

With the preliminaries now taken care of, we are situated to state
our main results on Lagrangian chaos. 
See Section \ref{sec:Outline} for a detailed outline of the proof. 

Below, $d = 2$ or $3$, and
 the vector field $u_t : \T^d \to \R^d, t > 0$ evolves according
to one of Systems \ref{sys:2DStokes} -- \ref{sys:3DNSE}, while
the Lagrangian flow $\phi^t : \T^d \to \T^d, t >0$ is as in \eqref{eq:xtintro}.
Throughout, $\hat \Hbf$ denotes the relevant vector field space for the 
system in question, e.g., $\hat \Hbf = \Hbf_{\mathcal K}$ when working
with System \ref{sys:2DStokes}. As in Proposition \ref{prop:WP}, $\mu$
denotes the stationary measure for the $(u_t)$ process on $\hat \Hbf$ for
each of Systems \ref{sys:2DStokes}, \ref{sys:Galerkin}, \ref{sys:NSE} or \ref{sys:3DNSE}.

\begin{theorem}[Positive Lyapunov exponent] \label{thm:Lyap}
Let $(u_t)$ be governed by any of Systems \ref{sys:2DStokes}--\ref{sys:3DNSE}. Then, there exists a deterministic constant $\lambda^+ > 0$ such that for \emph{every} initial vector field $u_0 \in \supp \mu$ and 
$x \in \T^d$, the following limit exists with probability one: 
\[
\lambda^+ = \lim_{t \to \infty} \frac{1}{t} \log | D_x \phi^t | > 0. 
\]
\end{theorem}

\noindent Indeed, 
as the following Corollary states, 
with probability 1 the Lagrangian flow map $\phi^t$ expands all vectors at the constant exponential rate $\lambda^+ > 0$ with probability 1.

\begin{corollary}[Norm growth of the flow map]\label{cor:expandAllDirections}
Let $\lambda^+ > 0$ be as in Theorem \ref{thm:Lyap}. 
For any $\eta > 0, \eta \ll \lambda^+$, 
$(u_0, x) \in \supp \mu \times \T^d$, and any unit vector $v \in \R^d$, there is a (random) constant $\delta = \delta(u_0, x, v, \eta)$ such that $\delta > 0$ almost-surely and for all $t>0$,
\[
| D_x \phi^t v| \geq \delta e^{t (\lambda^+ - \eta)} \qquad \text{with probability 1.}
\]
\end{corollary}

\begin{remark} \label{remark:deterministicHard}
Theorem \ref{thm:Lyap} and Corollary \ref{cor:expandAllDirections} (and the results on scalar advection below) make \emph{fundamental} use of the probabilistic framework. 
Such results seem hopelessly out of reach for 
deterministic models of fluid flows commonly observed in nature and many other systems of interest. For a general discussion of the difficulties involved, see, e.g., \cite{young2013mathematical, pesin2010open}.

A reasonable model for understanding the difficulties involved
 is the Chirikov Standard map \cite{chirikov1979universal},
a one-parameter family of deterministic, discrete-time,
volume-preserving mappings $\T^2 \to \T^2$ exhibiting the same stretching and folding expected to 
underly the mixing mechanism of the Lagrangian flow \cite{CrisantiEtAl1991}. Although anticipated to be true, it is a decades-old open problem to rigorously verify, for any parameter value, that the standard map is chaotic in the sense of a positive Lyapunov exponent on a positive-volume subset of phase space. Partly explaining the difficulties involved is the fact that very different asymptotic dynamical regimes coexist in phase space: for a topologically `large' subset of parameters, the Standard map has (1) an abundance of elliptic islands throughout phase space (inhibiting chaos) \cite{duarte1994plenty}, and (2) a positive Lyapunov exponent on a set of Hausdorff dimension 2 \cite{gorodetski2012stochastic}.
The situation is vastly different in the presence of even a small amount of noise: 
see \cite{blumenthal2017lyapunov} for positive results confirming chaos for the Standard map
subjected to small-amplitude noise.

In this paper, we will apply a principle known as Furstenberg's criterion
 from random dynamical systems theory: this says, 
roughly speaking, that $\lambda^+ > 0$ as in Theorem \ref{thm:Lyap} if the probabilistic law of the gradient $D_x \phi^t$
is sufficiently nondegenerate. See Section \ref{sec:Outline} and Section \ref{sec:RDS} for more discussion. 
\end{remark} 

\begin{remark}\label{rmk:supportMuIntro}
For Systems \ref{sys:2DStokes} -- \ref{sys:NSE}, Theorem \ref{thm:Lyap} and Corollary \ref{cor:expandAllDirections} hold for all initial $u_0 \in \hat \Hbf$. For the finite-dimensional System \ref{sys:2DStokes} and \ref{sys:Galerkin}, it follows from hypoellipticity, see \cite{E2001-lg,Romito2004-rc} $\supp \mu = \Hbf_{\mathcal K}$. For 2D Navier-Stokes as in System \ref{sys:NSE}, that $\supp \mu = \Hbf$ follows from \cite{agrachev2005navier}. It is likely that the same is true for 3D hyper-viscous Navier-Stokes as in System \ref{sys:3DNSE}, but as far as the authors are aware the appropriate controllability theorems do not appear in the literature.
\end{remark}

\begin{remark} 
The techniques we use currently require well-posed SPDEs, hence the hyper-viscous regularization in System \ref{sys:3DNSE}. 
We have included this case to emphasize that our infinite dimensional methods are not restricted to two dimensional flow -- the treatment of the 3D case (System \ref{sys:3DNSE}) is only slightly harder than 2D (System \ref{sys:NSE}).   
In fact, the methods could extend to many settings in which one has an infinite dimensional model coupled to finitely-many degrees of freedom on a Riemannian manifold.  
\end{remark}

\begin{remark} \label{rmk:StokesDegn} 
For 2D Stokes as in System \ref{sys:2DStokes}, we can prove all our results (above and below) using only the weaker noise condition (see Remark \ref{rmk:2Dstokes}) 
$\set{(1,0),(0,1),(-1,0),(0,-1)} \subset \mathcal{K}$. 
If these are the only modes, the velocity field is given by the very simple formula
\begin{align}
u(t,x) = Z_1(t) \begin{pmatrix} \sin y \\ 0 \end{pmatrix} + Z_2(t) \begin{pmatrix} \cos y \\ 0 \end{pmatrix} + Z_3(t) \begin{pmatrix} 0 \\ \sin x \end{pmatrix} + Z_4(t) \begin{pmatrix} 0 \\ \cos x \end{pmatrix},
\end{align} 
where $Z_j, 1 \leq j \leq 4$ are independent Ornstein-Uhlenbeck processes (they do not need to be i.i.d., though in that case the flow is statistically homogeneous in space).
\end{remark}

We note that Theorem \ref{thm:Lyap} and Corollary \ref{cor:expandAllDirections} for the finite-dimensional models in Systems \ref{sys:2DStokes} and \ref{sys:Galerkin}
follow from adaptations of previously known criteria \cite{carverhill1987furstenberg, baxendale1989lyapunov} (see also \cite{furstenberg1963noncommuting} and other
citations given in Section \ref{subsubsec:FurstenbergOutline}) 
for positive exponents for random dynamical 
systems generated by SDE combined with by-now standard hypoellipticity arguments for Galerkin truncations of Navier-Stokes \cite{E2001-lg,Romito2004-rc}. 
Nevertheless, we include them for the following reasons: these results are physically interesting and absent from the literature
(to the best of our knowledge); they emphasize that Assumption \ref{a:Highs} is not fundamental for Lagrangian chaos; 
all the ingredients needed for their proof are already required for our results on the infinite-dimensional model in System \ref{sys:NSE}; 
and, although simpler to work with, they are instructive for the proof in the infinite-dimensional case.

On the contrary, our results for the infinite-dimensional model in Systems \ref{sys:NSE}--\ref{sys:3DNSE} 
do not follow from previously existing results, and require a considerable amount of additional work. See Section \ref{sec:Outline} for an outline. 

\subsubsection{Scalar advection}
Consider first the problem of scalar advection without diffusivity
\begin{equation} \label{eq:SclNokap}
\partial_t f_t + u_t \cdot \grad f_t  = 0, 
\end{equation}
with $(u_t)$ given by one of System \ref{sys:2DStokes}--\ref{sys:3DNSE}. 
Here the initial datum $f_0 : \T^d \to \R$ is in $H^1$ with $\int f_0\, \dx = 0$.
By the same methods as in Proposition \ref{prop:WP}, the coupled system of $(u_t,f_t)$ has a $\PP$-a.s. unique, $\mathcal{F}_t$-adapted mild solution that defines a Feller Markov process on  $\Hbf \times H^1$. 
At times we will call $(u_t, f_t)$ the \emph{scalar process}. Using Theorem \ref{thm:Lyap} and some additional work, for the $(u_t, f_t)$ process
we prove the following exponential growth of gradients with probability 1: 

\begin{theorem}[Exponential gradient growth without diffusivity] \label{thm:ExpGrwth}
Consider \eqref{eq:SclNokap} with $(u_t)$ given by any of Systems \ref{sys:2DStokes}--\ref{sys:3DNSE}. Then, there exists a constant $\lambda > 0$, depending on the system, with the following property. For any $\eta > 0, \eta \ll \lambda$; any fixed initial $f_0 \in H^1 \setminus \{ 0 \}$ with $\int f_0 dx = 0$; and for every fixed initial $u_0 \in \supp \mu$,
there exists an almost-surely strictly positive random constant $\delta = \delta(u_0, f_0, \eta) > 0$ such that for all $t \geq 0$ and $p \in [1,\infty]$, 
\begin{align}
\norm{\grad f_t}_{L^p} \geq \delta e^{(\lambda- \eta)t} \qquad \text{with probability 1}.
\end{align}
When $d = 2$, $\lambda := \lambda^+$ as in Theorem \ref{thm:Lyap}.
\end{theorem} 
Recently the question of mixing of scalars, i.e. decay rates in $H^{-1}$ or mixing defined by Bressan in \cite{Bressan03}, has generated a lot of interest: see, e.g., \cite{LDT11,Seis2013,AlbertiEtAl14,IyerEtAl14} and the references therein.
This refinement will be addressed in future work. 

\subsubsection{Scalar turbulence in the Batchelor regime} \label{sec:Turbs}
Next, we are interested in studying vanishing diffusivity limits of the stationary measures associated to the following problem: 
\begin{equation}
\partial_t g_t + u_t \cdot \grad g_t = \kappa \Delta g_t + \widetilde{Q} \dot{\widetilde{W}}_t,  \label{eq:Sclkap} 
\end{equation}
with $u_t$ given by one of System \ref{sys:2DStokes}--\ref{sys:3DNSE}. Here, the initial datum is $g_0 \in H^1$ and has zero mean. 
The (mean-zero in space) random source $\widetilde{Q} \dot{\widetilde{W}}_t$ is of the form
\begin{equation}
\widetilde{Q} \dot{\widetilde{W}}_t = \sum_{k\in \Z^d_0} \widetilde{q}_k e_k(x)\dot{\widetilde{W}}_k(t),
\end{equation}
where $\{\widetilde{W}_k\}$ are an additional family of independent one-dimensional canonical Wiener processes also taken on the same filtered probability space $(\Omega,\mathcal{F},(\F_t),\P)$ and assumed independent of $\set{W_k}$. Define 
\begin{equation}
\bar{\eps} := \frac{1}{2}\sum_{k \in \Z^d_0} \abs{\tilde{q}_{k}}^2 \in (0,\infty). 
\end{equation}
For simplicity we additionally require at least $\sum_{k \in \Z^d_0} \abs{k}^2\abs{\tilde{q}_k}^2 < \infty$ (though it is likely this condition could be dropped). 
Note that the random source can be very smooth and degenerate, e.g. compactly supported in frequency. 
Under these conditions, as in Proposition \ref{prop:WP}, there is a $\PP$-a.s. unique, global-in-time, $\mathcal{F}_t$-adapted solution $(u_t,g_t)$  which defines a Feller Markov process on $\Hbf \times H^1$. 
Moreover, the Krylov-Bogoliubov procedure proves the existence of stationary measures $\set{\bar{\mu}^\kappa}_{\kappa > 0}$ supported on $\Hbf \times H^1$ (note that all such measures satisfy $\bar{\mu}^\kappa(A \times H^1) = \mu(A)$; see Section \ref{sec:turbulence} for more detail). 
By It\^o's lemma, one verifies that statistically stationary solutions $g^\kappa$ to \eqref{eq:Sclkap} satisfy the balance relation  
\begin{equation}
\kappa\EE \norm{\grad g^\kappa}_{L^2}^2 = \bar{\eps}. \label{eq:SclBal}
\end{equation}
As above, we are only considering $g$ which satisfy $\int g \,\dx = 0$ (which is conserved due to the mean-zero assumption on $\widetilde{Q}$).

The problem \eqref{eq:Sclkap} is an idealized model for `scalar turbulence' in the Batchelor regime (see e.g. \cite{Batchelor59,BalkovskyFouxon99,ShraimanSiggia00,CrisantiEtAl1991,FalkovickEtAl01}), which corresponds to the case when the velocity $u$ is much smoother (in space) than the scalar.
Passive scalar turbulence has been the subject of much research in the physics community both because of its intrinsic importance to physical applications and its potential to provide a place to develop analytic methods for understanding other turbulent systems \cite{ShraimanSiggia00}.
In Batchelor's original paper \cite{Batchelor59}, he considered a random straining flow as an idealized model for the small scale behavior of a passive scalar. Batchelor used this model to predict the power spectrum of the scalar, now known as \emph{Batchelor's law}. Later, the \emph{Kraichnan model} was introduced in \cite{Kraichnan68},  wherein the velocity field is taken to be a white-in-time Gaussian field with a prescribed correlation function in space. Hence, the random ODE \eqref{eq:xtintro} is replaced by an SDE with multiplicative noise and the scalar equation \eqref{eq:SclNokap} is replaced with a stochastic transport equation in Stratonovich form. There is an extensive literature on this model in physics; see e.g. \cite{ShraimanSiggia00, CrisantiEtAl1991, crisanti2012products} and the references therein. For the Kraichnan model, Theorem \ref{thm:Lyap} is proved in \cite{baxendale1993kinematic} using random dynamical systems theory developed in \cite{baxendale1989lyapunov}. 

The questions one is often interested in answering about systems such as \eqref{eq:Sclkap} are (A) can we develop analytical theories for predicting statistical properties of small scales in the limit $\kappa \rightarrow 0$? and (B) to what extent are these statistics universal, that is, which properties are independent of detailed information of the system? 
The predictions for (A) often come in the form of quantities such as \emph{structure functions}, for example
\begin{gather}
\EE (\delta_\ell g^\kappa)^p  \sim C_p \abs{\ell}^{\zeta_p}, \quad\quad \ell_D \lesssim \abs{\ell} \lesssim \ell_I \, , \\
\text{where } \quad \delta_\ell g(x) := g(x+\ell) - g(x) \, , 
\end{gather}
(where the meaning of $\sim$ is left informal for now) for a range of scales $\ell_D,\ell_I$ (for \emph{dissipative} and \emph{integral} respectively) assumed to satisfy $\lim_{\kappa \rightarrow 0} \ell_D(\kappa) = 0$ and $\ell_I$ much smaller than the length-scales of the large scale forcing in the system (but independent of $\kappa$).
For (B), the corresponding question is then to answer for which $p$ are the quantities $\zeta_p,C_p$ and/or $\ell_D$ are universal.
The first predictions of this general type were due to Kolmogorov \cite{K41a,K41b,K41c} in 1941, who studied the 3D Navier-Stokes equations as $\nu \rightarrow 0$. Some of his original predictions are now known to be inaccurate (though still good approximations for many statistics of interest); see e.g. \cite{K62,AnselmetEtAl1984,Frisch1995} and the references therein. 
One of his predictions, the 4/5 law, is very well matched by experiments (indeed, it is considered one of the few `exact' laws of turbulence \cite{Frisch1995}) and is universal\footnote{Both the constant and the exponent are universal; it is not clear whether $\ell_D$ is universal.}. In 1949, Yaglom \cite{Yaglom49} made the analogous prediction\footnote{Of course, this is more like the 4/3 law than the 4/5 law, but the distinction for Navier-Stokes is due to the vector-valued nature.}
\begin{align}
\EE \left( \abs{\delta_\ell g^\kappa}^2 \delta_{\ell}u \cdot \frac{\ell}{\abs{\ell}} \right) \sim -\frac{4}{3} \bar{\eps} \abs{\ell} \, .
\end{align}
This is the law we confirm for \eqref{eq:Sclkap} (in a spherically averaged sense); see Theorem \ref{thm:turb} below for the rigorous meaning of $\sim$ in this statement. 

Yaglom's law, like the Kolmogorov $4/5$ for 3D Navier-Stokes, is an expected consequence of the statistical stationarity and ``anomalous dissipation'', that is, when the dissipation rate of a quantity is non-vanishing (or at least vanishing at an anomalously slow rate) in the limit of vanishing dissipative effects (see \cite{Yaglom49,Frisch1995,BCZPSW18}). 
In \cite{BCZPSW18}, it is proved that the Kolmogorov 4/5 law follows for statistically stationary solutions of the 3D Navier-Stokes using that $\lim_{\nu \rightarrow 0} \nu \EE\norm{u^\nu}_{L^2}^2 = 0$. This property is referred to therein as ``weak anomalous dissipation''\footnote{We remark that this property is equivalent to the assertion that the Taylor microscale goes to zero as Reynolds number goes to infinity; see \cite{BCZPSW18} for details.}, and is a natural form of anomalous dissipation for statistically stationary solutions (see \cite{BCZPSW18} for more discussion).  

In this work, we use Theorem \ref{thm:ExpGrwth} to prove the analogous statement here (\eqref{eq:WAD} below) by adapting arguments from \cite{BCZGH}; see Section \ref{sec:turbulence} for details.
Then Yaglom's law, as stated in \eqref{eq:Yaglom}, follows from a straightforward variation of the argument in \cite{BCZPSW18}. 
Inequality \eqref{eq:WAD} cannot hold if solutions to \eqref{eq:Sclkap} remain concentrated in low frequencies in the limit $\kappa \rightarrow 0$; indeed in this case it is easy to check that $\kappa \E \norm{g^\kappa}_{L^2}^2 \gtrsim 1$ (see also Remark \ref{rmk:nrmExplode} below). 
For \eqref{eq:WAD} to hold, the fluid needs to transfer `most' of the $g$ to successively smaller scales where it is more efficiently dissipated by the $\kappa \Delta g^\kappa$ term, resulting in a much-enhanced dissipation rate. It is Theorem \ref{thm:ExpGrwth} that ultimately implies the Lagrangian flow-map creates small scales everywhere in the domain with probability 1. See also the earlier work using norm growth in the inviscid passive scalar problem to obtain `enhanced dissipation' effects for $\kappa >0$ models \cite{CKRZ,Zlatos10} and the recent related work \cite{CZDE18}.   

The idea that Lagrangian chaos and scalar turbulence scaling laws should be intimately related has long been expected by the physics community; see, e.g., \cite{AntonsenOtt1991,AntonsenEtAl96,YuanEtAl00,ShraimanSiggia00} and the references therein for more information. 
    
\begin{theorem}[Scalar turbulence in the Batchelor regime] \label{thm:turb}
Let $\set{u,g^\kappa}_{\kappa > 0}$ be a sequence of statistically stationary solutions to \eqref{eq:Sclkap} with $(u_t)$ given by any of Systems \ref{sys:2DStokes}--\ref{sys:3DNSE}.  Then,
\begin{itemize} 
\item[(i)] the Weak Anomalous Dissipation property holds: 
\begin{equation}
\lim_{\kappa \rightarrow 0} \kappa \EE \norm{g^\kappa}_{L^2}^2 = 0; \label{eq:WAD}
\end{equation}
\item[(ii)] Yaglom's law holds over a suitable inertial range: that is, $\forall \kappa > 0$ small, there exists an $\ell_D(\kappa) > 0$ with $\lim_{\kappa \rightarrow 0} \ell_D(\kappa) = 0$ such that 
\begin{equation}
\lim_{\ell_I \rightarrow 0} \limsup_{\kappa \rightarrow 0}  \sup_{\ell \in [\ell_D, \ell_I]}  \abs{ \frac{1}{\ell}\EE \fint_{\T^d} \fint_{\S^{d-1}} |\delta_{\ell n} g^\kappa|^2 \delta_{\ell n} u\cdot n \, \dee S(n) \dx + \frac{4}{3}\bar{\eps} } = 0. \label{eq:Yaglom}
\end{equation}
\end{itemize} 
\end{theorem} 

\begin{remark}
Note that by time stationarity, \eqref{eq:Yaglom} is the same as asserting the expected value of arbitrary length time averages follow Yaglom's law. Further, as in \cite{BCZPSW18}, if one assumes $Q$ and $\widetilde{Q}$ are spatially homogeneous, then there exists spatially homogeneous statistically stationary solutions to the system $(u_t,g_t)$ and one can remove the $x$ average from \eqref{eq:Yaglom}, that is, \eqref{eq:Yaglom} holds a.e. in $x$. 
\end{remark}

\begin{remark} \label{rmk:nrmExplode}
Note that by the balance \eqref{eq:SclBal}, the weak anomalous dissipation property \eqref{eq:WAD}, and Sobolev interpolation, there holds $\lim_{\kappa \rightarrow 0}\kappa \EE \norm{g^\kappa}_{H^\gamma}^2 = 0$ for all $\gamma \in (0,1)$ and $\lim_{\kappa \rightarrow 0} \kappa \EE \norm{g^\kappa}_{H^\gamma}^2 = +\infty$ for all $\gamma > 1$. 
\end{remark}

\section{Outline of the proofs} \label{sec:Outline}

Let us now give a somewhat detailed outline for the proofs of the main results of this paper, starting with Theorem \ref{thm:Lyap}. 

The basic structure of the proof can be summarized in two main points:
\begin{itemize} 
\item[(1)] The Multiplicative Ergodic Theorem and a variant of Furstenberg's criterion shows that, given suitable ergodic properties of the dynamics, the Lyapunov exponent is strictly positive unless there is a certain \emph{almost surely} invariant structure in the motion of $x_t = \phi^t(x_0)$ and the gradient $D_{x_0} \phi^t$; 
\item[(2)] hypoellipticity and approximate controllability arguments show that (A) the dynamics satisfy suitable ergodic properties and that (B) a rich range of motions of $x_t$ and $D_{x_0} \phi^t$ are realized. This will rule out the invariant structure and allow us to deduce a positive Lyapunov exponent as in Theorem \ref{thm:Lyap}.
\end{itemize} 

\noindent As we will see below, both are significantly harder in the infinite dimensional case (Systems \ref{sys:NSE}--\ref{sys:3DNSE}). 

\subsection{The RDS framework and the Multiplicative Ergodic Theorem}\label{subsubsec:RDSframeworkOutline}

Theorem \ref{thm:Lyap} makes two assertions: 
 (i) that the limit defining the Lyapunov exponent $\lambda^+$ exists and is 
 constant almost surely, and (ii) that this exponent satisfies $\lambda^+ > 0$.
Let us first outline how to prove assertion (i) using tools from random dynamical systems theory.

To start, we must formulate the Lagrangian process $(u_t, x_t)$ as a stochastic flow or \emph{random dynamical system} (RDS) on ${\hat{\Hbf}} \times \T^d$ (here, $\hat \Hbf$ is as in the beginning of Section \ref{sec:Results}). 
That is, given a random noise path $\omega \in \Omega$
and a fixed initial $(u_0, x_0) \in {\hat{\Hbf}} \times \T^d$, the 
assignment $(u_0, x_0) \mapsto (u_t, x_t)$ is realized as 
$(u_t, x_t) = \Theta_\omega^t(u_0, x_0)$,
where $\Theta_\omega^t : {\hat{\Hbf}} \times \T^d \to {\hat{\Hbf}} \times \T^d$ is a 
continuous mapping depending measurably on the noise parameter $\omega$ (see Section \ref{subsubsec:basicSetupRDS} for details).
In our setting, $\Theta^t_\omega$ is of the form $\Theta^t_\omega(u, x) = (\mathcal U_\omega^t (u), \phi^t_{\omega, u} (x))$, where
$\mathcal U_\omega^t : {\hat{\Hbf}} \to {\hat{\Hbf}}$ is the time-$t$ mapping associated to the equation governing $(u_t)$ (any of Systems \ref{sys:2DStokes}--\ref{sys:3DNSE}), i.e., 
the map sending $u_0 \mapsto u_t$,
and $\phi^t_{\omega, u} = \phi^t : \T^d \to \T^d$ is the time-$t$ Lagrangian flow map associated to the noise parameter $\omega$ and the initial vector field $u \in {\hat{\Hbf}}$
as in \eqref{eq:xtintro}, i.e., the diffeomorphism on $\T^d$ sending $x_0 \mapsto x_t$.
In the context of RDS, the matrix-valued mapping $\Omega \times {\hat{\Hbf}} \times \T^d \to M_{d \times d}(\R)$ 
sending $(\omega, u, x) \mapsto  D_x \phi^t_{\omega, u}$ for fixed $t > 0$ is an object known as a
\emph{linear cocycle} over the RDS $\Theta^t_\omega$. 

For more background on random dynamics and a precise enumeration of the assumptions involved, see
Sections \ref{subsec:RDSelements} -- \ref{subsec:LinearCocycles}, where the relevant theory and assumptions
are spelled out for an abstract RDS $\mathcal T$ acting on a metric space $Z$ and a linear cocycle $\mathcal A$
over $\mathcal T$. Throughout Section \ref{sec:RDS} we intend to apply this with $\mathcal T$ replaced by the Lagrangian
flow $\Theta$ acting on $Z = \hat \Hbf \times \T^d$ with $\mathcal A$ replaced by the gradient cocycle $D_x \phi^t$.
It is straightforward to verify the assumptions made in Sections \ref{subsec:RDSelements} -- \ref{subsec:LinearCocycles} for
$\Theta$ and $D_x \phi^t$; this is carried out in the Appendix (Section \ref{subsec:wellPosedApp}).

A fundamental result pertaining to linear cocycles is the Multiplicative Ergodic Theorem, stated in
full in Section \ref{subsubsec:METRDS} as Theorem \ref{thm:MET}. For the purposes of this discussion, we state
below the following consequence, often referred to as the Furstenberg-Kesten Theorem \cite{furstenberg1960products}.

\begin{proposition}\label{prop:lyapOutline}
The limit
\[
\lambda^+(\omega, u, x) := \lim_{t \to \infty} \frac{1}{t} \log |  D_x \phi^t_{\omega, u}|
\]
exists for $\P$-a.e. $\omega$ and $\mu \times \Leb$-a.e. $(u, x) \in {\hat{\Hbf}} \times \T^d$, where
$\mu$ is the stationary measure for the $(u_t)$ process as in Proposition \ref{prop:WP}.

Moreover, if $\mu \times \Leb$ is an ergodic stationary measure (Definition \ref{defn:PtmuInvariant}) for the Lagrangian process $(u_t, x_t)$,
then the limiting value $\lambda^+$ does not depend on $(\omega, u, x)$.
\end{proposition}

Ergodicity of $\mu \times \Leb$ as a stationary measure for the Lagrangian process $(u_t, x_t)$
is a necessary ingredient for Theorem \ref{thm:Lyap}. See Section \ref{subsubsec:ErgodicProperties}
below for a discussion of the ergodic properties of the $(u_t, x_t)$ process.

\begin{remark}\label{rmk:everyAEOutline}
	Note that in Theorem \ref{thm:Lyap}, the Lyapunov exponent $\lambda^+$ is
	asserted to exist with probability 1 at \emph{every} initial $(u, x) \in \supp \mu \times \T^d$, as opposed to $\mu \times \Leb$-
	almost every $(u, x)$ as in 
	Proposition \ref{prop:lyapOutline}.
	The strong Feller property (Definition \ref{defn:strongFeller})
	 for the $(u_t, x_t)$ process allows us to pass between these formulations: see Lemma 
	\ref{lem:ultraFeller}(b) in Section \ref{subsec:furstInfiniteDimensions}.
\end{remark}

\subsection{Determining positive Lyapunov exponents: Furstenberg's criterion}\label{subsubsec:FurstenbergOutline}

An entirely separate matter is to verify that $\lambda^+$ as in Proposition \ref{prop:lyapOutline} is
strictly positive. This problem is notoriously difficult (see Remark \ref{remark:deterministicHard} above). Aiding us, however, is the  fact that the cocycle $(\omega, u, x) \mapsto D_x \phi^t_{\omega, u}$
is subjected to some noise. For such cocycles, a powerful tool known as Furstenberg's criterion
implies $\lambda^+ > 0$ under suitable nondegeneracy conditions described in detail below.
The criterion was originally obtained in \cite{furstenberg1960products} for IID products of matrices,
and extended in scope by various authors in the ensuing years: see, e.g., 
\cite{avila2010extremal, guivarc1985frontiere, gol1989lyapunov, baxendale1989lyapunov, ledrappier1986positivity}, and also
the citations of Chapter 1 of \cite{bougerol2012products} for a more complete bibliography.

Ignoring for now the requisite quantifiers and other details, the
relevant version of Furstenberg's criterion can be stated as follows.
Proposition \ref{prop:furstenbergOutline} below is a version of the criterion given in \cite{ledrappier1986positivity}, and will 
be stated in full as Theorem \ref{thm:furstenberg} in Section \ref{subsec:FurstenbergFD}.
Below, $P^{d-1} = P(\R^d)$ denotes the manifold of one-dimensional subspaces of $\R^d$.
\begin{proposition}[Informal Furstenberg criterion]\label{prop:furstenbergOutline}
Assume $\mu \times \Leb$ is an ergodic stationary measure for the Lagrangian process $(u_t, x_t)$.
If $\lambda^+ = 0$, then to each $(\mu \times \Leb)$-generic
$(u, x)$, there is associated a deterministic (i.e., $\omega$-independent) probability measure
$\nu_{u, x}$ on $P^{d-1}$ with the property that 
\begin{align}\label{eq:furstenberg111}
 (D_x \phi^t_{\omega, u})_* \nu_{u, x} = \nu_{\Theta^t_\omega(u, x)}
 \end{align}
for all $t > 0$ and $\P \times \mu \times \Leb$-almost all $(\omega, u, x) \in \Omega \times \Hbf \times \T^d$.
 \end{proposition}
\noindent To prove $\lambda^+ > 0$, then, it suffices to obtain a contradiction from the conclusions of Proposition \ref{prop:furstenbergOutline}.

Conceptually, the measures $\nu_{u, x}$ should be thought of as deterministic ``configurations''
of vectors on $\R^d$, and the relation \eqref{eq:furstenberg111} says that this $(u,x)$-dependent family $(\nu_{u, x})$
 of deterministic ``configurations'' is left invariant by the Jacobian matrices $D_x \phi^t_{\omega, u}$ with probability 1.
As such, the relation \eqref{eq:furstenberg111} has the connotation of a \emph{degeneracy} in the probabilistic law of the matrices
$D_x \phi^t_{\omega, u}$ with $\omega$ distributed as $\P$.

\subsection{Ruling out Furstenberg's criterion: finite-dimensional models}\label{subsubsec:ruleOutFurstFDOutline}

Given a pair of probability measures $\nu, \nu'$ on $P^{d-1}$, the set of matrices $M \in SL_d(\R)$
for which $M_* \nu = \nu'$ has empty interior (Lemma \ref{lem:emptyInterior}). Roughly speaking, 
we can rule out \eqref{eq:furstenberg111} in Furstenberg's criterion if we can show
that for a ``large enough'' set of pairs $(u, x), (u', x') \in \hat \Hbf \times \T^d$, the probabilistic law of 
$A_t := D_{x_0} \phi^t_{\omega, u_0}$ conditioned on the event $(u_0, x_0) = (u, x), (u_t, x_t) = (u', x')$
is sufficiently nondegenerate. 

For the finite-dimensional models in Systems \ref{sys:2DStokes} and \ref{sys:Galerkin}, we can compute
this conditional law explicitly. 
The matrix-valued process $A_t := D_x \phi^t_{\omega, u}$ is a component of the Markov process $(u_t, x_t, A_t)$
 generated by the $(u_t)$ together with \eqref{eq:xtintro} and
 \begin{align}\label{eq:matrixDefOutline}
 \partial_t A_t = \grad u_t(x_t) \, A_t 
 \end{align}
on the finite-dimensional manifold $\mathcal M := \hat \Hbf \times \T^d \times SL_d(\R)$.

Under suitable nondegeneracy conditions on the SDE governing $(u_t, x_t, A_t)$, for instance,
H\"ormander's condition as described in \ref{subsubsec:ErgodicProperties} below, the law
$Q_t((u, x, \Id), \cdot)$ of $(u_t, x_t, A_t)$ conditioned on $(u_0, x_0, A_0) = (u, x, \Id)$
admits an everywhere-positive smooth density $\rho = \rho_{(u, x)} : \hat \Hbf \times \T^d \times SL_d(\R) \to (0,\infty)$
for all initial $(u, x) \in \hat \Hbf \times \T^d$.
It then follows that for 
any pair $(u, x), (u', x') \in \hat \Hbf \times \T^d$ and any $t > 0$, the probabilistic law of $A_t$ 
conditioned on $(u_0, x_0) = (u, x), (u_t, x_t) = (u',x')$ 
admits a smooth, everywhere-positive density $\hat \rho = \hat \rho_{(u, x), (u', x')}$, given for $M \in SL_d(\R)$
by 
	\[
	\hat \rho(M) = \rho(u', x', M)\bigg/\int_{SL_d(\R)} \rho(u', x', M') \,\dee \Leb_{SL_d(\R)} (M') \, ,
	\]

\noindent We conclude that \eqref{eq:furstenberg111} is impossible, hence $\lambda^+ > 0$, when H\"{o}rmander's
condition for the matrix process $(u_t, x_t, A_t)$ is satisfied. See 
 Proposition \ref{prop:hormanderForAll} in Section \ref{subsubsec:ErgodicProperties} below for a precise statement
 of H\"{o}rmander's condition, and see condition (C) in Section \ref{subsubsec:nondegenCondLaws} for a more
 detailed version of this argument.

We note that the technique of using H\"{o}rmander's condition for the matrix process $(u_t, x_t, A_t)$ to rule out Furstenberg's criterion is well-known; see, e.g., \cite{carverhill1987furstenberg, baxendale1989lyapunov}.

\subsection{Furstenberg's criterion: infinite-dimensional models}\label{subsubsec:furstenbergInfDimOutline}

For the infinite-dimensional models, Systems \ref{sys:NSE}--\ref{sys:3DNSE}, we are not aware of any means by which one can prove a positive density for the conditional law of $A_t = D_x \phi^t_{\omega, u}$ as was possible for the finite-dimensional models. 

Instead, we are able to prove a certain ``approximate controllability'' statement, described below. 
To articulate this we define the \emph{projective process} $(u_t, x_t, v_t)$ on 
 $ \Hbf \times \T^d \times P^{d-1}$, where $(v_t)$ is defined for initial $v_0$ by
 setting $v_t$ to be the projective representative of $D_{x_0} \phi^t_{\omega, u_0} v_0$. 
 Equivalently, $(v_t)$ is generated by $(u_t)$, \eqref{eq:xtintro} and
 \begin{align}\label{eq:projectiveSDEOutline}
 \partial_t v_t =\Pi_{v_t} \nabla u(x_t) v_t \, .
 \end{align}
 Here, $\Pi_{v_t}$ denotes the projection onto the orthogonal complement of (a unit vector representative of) $v_t$.
 
\begin{proposition}\label{prop:approxControlOutline}
Consider the Markov processes $(u_t, x_t, v_t)$ and $(u_t,x_t,A_t)$ 
generated by either of Systems \ref{sys:NSE} or \ref{sys:3DNSE}, together with \eqref{eq:xtintro}, \eqref{eq:matrixDefOutline}, and \eqref{eq:projectiveSDEOutline}. Then, for any $x, x' \in \T^d$ and $t > 0$, we have the following.
	\begin{itemize}
		\item[(a)] For any $\epsilon, M > 0$, we have that \[
		\P((u_t, x_t) \in B_\epsilon(0) \times B_\epsilon(x') \, , | A_t| > M \, | u_0 = 0, x_0 = x, A_0 = \Id) > 0 \, .\]
		\item[(b)] For any $\epsilon > 0$, $v \in P^{d-1}$ and open $V \subset P^{d-1}$, we have
		\[
		\P((u_t, x_t) \in B_\epsilon(0) \times B_\epsilon(x') \, , v_t \in V | u_0 = 0, x_0 = x, v_0 = v) > 0 \, .
		\]
	\end{itemize}
\end{proposition}
Condition (a) says, roughly, that gradient norms can be made arbitrarily large while
``approximately conditioning'' on the time $0$ and time $t$ values of the Lagrangian process, while
condition (b) says that we can rotate vectors arbitrarily in projective space. We see that this is weaker than obtaining information on the conditional law, but is clearly closely related. Our proof of Proposition \ref{prop:approxControlOutline} for Systems \ref{sys:NSE} and \ref{sys:3DNSE} is very physically intuitive; see Section \ref{subsubsec:ErgodicProperties} for more discussion. 

Furstenberg's criterion as in Proposition \ref{prop:furstenbergOutline} cannot be applied directly to the ``softer'' nondegeneracy condition in Proposition \ref{prop:approxControlOutline}. 
Possible issues include (1) that the family of measures $\set{\nu_{u, x}}_{(u, x) \in \Hbf \times \T^d}$
in Proposition \ref{prop:furstenbergOutline} might, a priori, be discontinuous in space, and (2)
that the individual measures $\nu_{u, x}$ could be quite pathological, e.g., singular continuous w.r.t. Lebesgue on $P^{d-1}$.
To address this, we obtain the following classification of all possible demeanors of the 
measure family $\nu_{u, x}$.

\begin{proposition}\label{prop:refineFurstOutline}
Assume that $\mu \times \Leb$ is an ergodic stationary measure for the Lagrangian process $(u_t, x_t)$, and moreover,
assume that the Lagrangian process $(u_t, x_t)$ satisfies the strong Feller property (Definition \ref{defn:strongFeller}). 
If $\lambda^+ = 0$, then one of the following alternatives holds.
\begin{itemize}
	\item[(a)] There is a continuously-varying family $\{ \langle \cdot, \cdot \rangle_{u, x} \}_{(u, x) \in \Hbf \times \T^d}$ of 
	inner products on $\R^d$ such that 
		\[ \langle D_x \phi^t_{\omega, u} v, D_x \phi^t_{\omega, u} w \rangle_{\Theta^t_\omega(u, x)} = \langle v, w \rangle_{u, x} \qquad \text{with probability 1.}
		\]
	for all $v, w \in \R^d, t > 0$ and $(u, x) \in \Hbf \times \T^d$.
	\item[(b)] There are $p \geq 1$ families $\{ E^i_{(u, x)}\}_{(u, x) \in \Hbf \times \T^d}, 1 \leq i \leq p$ of proper linear
	 subspaces of $\R^d$ such that
	(i) $(u, x) \mapsto E^i_{(u, x)}$ is \emph{locally continuous up to relabeling} (see Theorem \ref{thm:classification} (b) for details), and (ii) for all $(u, x) \in \Hbf \times \T^d$ and $1 \leq i \leq p$,
	\[
		D_x \phi^t_{\omega, u} ( E^i_{u, x}) = E^{\pi(i)}_{\Theta^t_\omega(u, x)} \qquad \text{with probability 1.}
	\]
	Here, $\pi = \pi_{\omega, u, x}$ is a permutation of $\{ 1, \cdots, p\}$.
\end{itemize}
\end{proposition}
\noindent Note that the Strong Feller property of the Lagrangian process is explicitly required; see Remark \ref{rmk:whyUseStrongFeller}
below for more discussion. We discuss proving the strong Feller property in Section \ref{subsubsec:ErgodicProperties} below.
Roughly speaking, Proposition \ref{prop:refineFurstOutline} 
follows from the strong Feller property as well as 
certain rigid geometric properties of $SL_d(\R)$ (Lemma \ref{lem:subgroups}) imposed by the condition of leaving 
a projective measure invariant (in the sense of Furstenberg's criterion as in Proposition \ref{prop:furstenbergOutline}). 

\medskip

Proposition \ref{prop:refineFurstOutline} is the analogue of Theorem 6.8 in Baxendale's paper \cite{baxendale1989lyapunov},
a similar classification-type theorem for the derivative cocycle of an SDE on a finite-dimensional manifold. 
The analogue we obtain (stated as Theorem \ref{thm:classification}
and proved in Section \ref{subsubsec:classifyRDS}) is considerably more general and applies to
linear cocycles over continuous-time RDS on possibly infinite-dimensional 
Polish spaces.
Our more general setting entails numerous complications not addressed in \cite{baxendale1989lyapunov};
see Remark \ref{rmk:whyNovelRDS} for a more thorough discussion of these.

Alternatives (a) and (b) in Proposition \ref{prop:refineFurstOutline} can
now be ruled out by straightforward continuity arguments and approximate controllability 
 as in Proposition \ref{prop:approxControlOutline}; see Section \ref{subsubsec:approxControlRuleOut} for more details.
Once this has been carried out, the proof of Theorem \ref{thm:Lyap} for Systems \ref{sys:NSE} and \ref{sys:3DNSE} is complete.

 \begin{remark}\label{rmk:whyUseStrongFeller}
 	As far as the authors are aware, the strong Feller property of the Lagrangian 
	process $(u_t, x_t)$ is required for Proposition \ref{prop:refineFurstOutline}. Specifically, 
	the strong Feller property is used to verify that the ``configurations'' appearing in 
	alternatives (a), (b) of Proposition \ref{prop:refineFurstOutline} are continuously-varying in an appropriate sense.
	We emphasize that this continuity is critical to the argument for ruling out (a), (b) using the approximate
	controllability condition in Proposition \ref{prop:approxControlOutline}. 
	
	In particular, this is precisely the step we are not able to execute for 2D Navier-Stokes
	with ``truly hypoelliptic'' forcing (that is, forcing only a handful of low modes as in Assumption \ref{a:lowms}
	and forgoing forcing all sufficiently high modes as in Assumption \ref{a:Highs}).
	In this regime, the strong Feller property is likely to be false for Systems \ref{sys:NSE}--\ref{sys:3DNSE} \cite{HM06}. 
 \end{remark}
 
 \subsection{Expansion in all directions: proof of Corollary \ref{cor:expandAllDirections}}\label{subsubsec:expandAllDirOutline}
 
For both the finite and infinite dimensional systems considered in this paper,
 Corollary \ref{cor:expandAllDirections} does not follow immediately from Theorem \ref{thm:Lyap}.
Indeed, a priori it is possible that given $(u, x) \in \hat \Hbf \times \T^d$, there are some $v \in \R^d$ for which 
$\limsup_{t \to \infty} \frac{1}{t} \log |D_x \phi^t_{\omega, u} v| < \lambda^+$ holds with probability 1.

We can rule this out using the ergodic theory of the projective process
$(u_t, x_t, v_t)$ as in \eqref{eq:projectiveSDEOutline}.
There is a well-known correspondence between the stationary probability measures $\nu$ 
on $\hat \Hbf \times \T^d \times P^{d-1}$ and the asymptotic exponential 
growth rates $\lim_{t \to \infty} \frac{1}{t} \log |D_x \phi^t_{\omega, u} v|$
realized ``with probability 1'' as $v$ varies in $\R^d \setminus \{0\}$.
The correspondence is given by the so-called Random Multiplicative Ergodic Theorem (Theorem III.1.2 in \cite{kifer2012ergodic}). We will not state the full result here,
except to note the following relevant consequence.

\begin{proposition}\label{prop:expandAllDirectionsOutline}
	Assume that there is a \emph{unique} stationary measure $\nu$ for the projective
	process $(u_t, x_t, v_t)$.
	Then, for $(\mu \times \Leb)$-almost every $(u, x) \in \hat \Hbf \times \T^d$ and 
	\emph{every} $v \in \R^d \setminus \{ 0 \}$, we have that
	\[
	\lim_{t \to \infty} \frac{1}{t} \log |D_x \phi^t_{\omega, u} v| = \lambda^+ \qquad \text{with probability 1.}
	\]
\end{proposition}
Proposition \ref{prop:expandAllDirectionsOutline} is formulated in a more general way as Proposition \ref{prop:refineMET} in Section \ref{subsubsec:projectiveProcessRDS}, 
to which we refer the reader for more details. 
The expansion estimate appearing in Corollary \ref{cor:expandAllDirections} now follows from a straightforward argument.

Added to our growing list of ingredients is uniqueness of the stationary measure $\nu$ for the projective process,
to which we refer the reader to Section \ref{subsubsec:ErgodicProperties} for more information.
 
\subsection{Gradient growth: proof of Theorem \ref{thm:ExpGrwth}}
Given an initial $u_0 = u \in \hat \Hbf$, an initial scalar $f_0 = f \in H^1, \int f dx = 0$, and a noise parameter $\omega \in \Omega$, the corresponding solution $(f_t)$ 
for the passive advection equation \eqref{eq:SclNokap}
is given by
\begin{align}
f_t(x) = f \circ (\phi^t_{\omega, u})^{-1}(x) \, .
\end{align}

By incompressibility, we have (recall $-\top$ is standard shorthand for the inverse transpose)
\begin{equation}
\| \grad f_t \|_{L^1} = \int \abs{\grad f_t(x)} \dx = \int \abs{ \left(D_x \phi^t_{\omega,u}\right)^{-\top}\grad f_0(x)} \dx. 
\end{equation}

The object $\left(D_x \phi^t_{\omega,u} \right)^{-\top} $ 
defines a cocycle over the RDS $\Theta^t_\omega$ on $\hat \Hbf \times \T^d$
 in the same manner as $D_x \phi^t_{\omega, u}$. To complete the proof of 
 Theorem \ref{thm:ExpGrwth}, it suffices to obtain the following analogue
 of Corollary \ref{cor:expandAllDirections} for this new cocycle.
 \begin{proposition}\label{prop:expandAllDirectionsNegTrans}
	There is a constant $\lambda > 0$ with the following property. For any $\eta > 0, \eta \ll \lambda$,
  $\mu \times \Leb$-almost every $(u, x) \in \hat \Hbf \times \T^d$, and every unit vector $v \in \R^d$, there is a 
  (random) constant $\hat \delta = \hat \delta_\omega(u, x, v, \eta)$ (i.e., depending on the  noise parameter $\omega \in \Omega)$
  such that with probability 1, $\hat \delta > 0$  and 
  \[
  | (D_x \phi^t_{\omega, u})^{- \top} v| \geq \hat \delta e^{t (\lambda - \eta)} \, .
  \]
  When $d = 2$, we have $\lambda = \lambda^+$.
 \end{proposition}
  Setting $v = \grad f_0(x) / | \grad f_0(x)| $ and integrating over $\{  x \in \T^d : \grad f_0 \neq 0\}$, we obtain Theorem \ref{thm:ExpGrwth}
for $p = 1$. The estimate for the remaining $L^p$ spaces follows from  $\| \grad f_t\|_{L^1} \lesssim \| \grad f_t\|_{L^p}$ for all $p \in [1,\infty]$.

\medskip

To prove Proposition \ref{prop:expandAllDirectionsNegTrans}, we prove Theorem \ref{thm:Lyap} and Corollary \ref{cor:expandAllDirections} with the $(-\top)$-cocycle $(D_x \phi^t)^{- \top}$ 
replacing the usual $D_x \phi^t$.
Let us summarize briefly how this will be done. For Theorem \ref{thm:Lyap} we have the following.

\begin{proposition}\label{prop:METforNegTrans} \
\begin{itemize}
	\item[(a)] For $\mu \times \Leb$-almost every $(u, x) \in \hat \Hbf \times \T^d$, the growth rate
		\[
		\check \lambda^+ (\omega, u, x) = \lim_{t \to \infty} \frac{1}{t} \log |(D_x \phi^t_{\omega, u})^{- \top}|
		\]
		exists with probability 1. Moreover, if $\mu \times \Leb$ is the unique (hence ergodic) 
		stationary measure for the $(u_t, x_t)$ process, then $\check \lambda^+$ is independent of $\omega, u, x$.
	\item[(b)] Let $\lambda^+$ be as in Proposition \ref{prop:lyapOutline}. Then, $\lambda^+ > 0$ iff $\check \lambda^+ > 0$.
	Indeed, $\lambda^+ = \check \lambda^+$ if $d = 2$. 
\end{itemize}	
\end{proposition}
Item (a) is merely a repetition of Proposition \ref{prop:lyapOutline} for the $(-\top)$-cocycle and is a consequence of
the Multiplicative Ergodic Theorem; see Theorem \ref{thm:MET} for details. As in Theorem \ref{thm:Lyap}, 
passing between ``almost every'' and ``every'' is done using the Strong Feller property; see Remark \ref{rmk:everyAEOutline}.
 Item (b) is a consequence of a general
relationship between the Lyapunov exponents of $D_x \phi^t$ and $(D_x \phi^t)^{- \top}$; see Section \ref{subsubsec:inverseTransCocycle} for details.
In particular, note that the relation $\lambda^+ = \check \lambda^+$ is exclusive to $d =2$; the authors are unaware
of any reason to expect it to hold in dimension $d = 3$.

\medskip

Having shown (Theorem \ref{thm:Lyap}) that $\lambda^+ > 0$, we conclude $\check \lambda^+ > 0$. 
To prove the analogue of Corollary \ref{cor:expandAllDirections} for the $(- \top)$-cocycle
will require, as in Proposition \ref{prop:expandAllDirectionsOutline}, for us to study the so-called
\emph{$(- \top)$-projective process} $(u_t, x_t, \check v_t)$ on $\hat \Hbf \times \T^d \times P^{d-1}$, 
defined for initial $\check v_0 \in P^{d-1}$ by setting $\check v_t$ to be the projective representative
of $(D_x \phi^t_{\omega, u})^{- \top} \check v_0$.
Equivalently, the $(\check v_t)$ process is governed by $(u_t)$, \eqref{eq:xtintro}, and
\begin{align}
\partial_t \check{v}_t & = -\Pi_{\check{v}_t} (\grad u_t (x_t))^{\top} \check{v}_t \,.
\end{align}
Repeating Proposition \ref{prop:expandAllDirectionsOutline} verbatim with $D_x \phi^t$ replaced
by $(D_x \phi^t)^{- \top}$, we see that Proposition \ref{prop:expandAllDirectionsNegTrans} follows
immediately from the existence of a unique (hence ergodic) stationary measure $\check \nu$
for the $(-\top)$-projective process $(u_t, x_t, \check v_t)$. 

\subsection{Hypoellipticity}\label{subsubsec:ErgodicProperties}

The previous discussion of the proofs of Theorems \ref{thm:Lyap} and \ref{thm:ExpGrwth} 
requires a number of ingredients pertaining to the properties of the various stochastic
processes (Lagrangian, projective, $(- \top)$-projective, and matrix) mentioned so far. 
Specifically, we need the following:
\begin{itemize}
	\item[(a)] Uniqueness of the stationary measure for the (i) Lagrangian, (ii) projective and (iii) $(- \top)$-projective processes;
	\item[(b)] For the infinite-dimensional Systems \ref{sys:NSE}--\ref{sys:3DNSE}, the Strong Feller property (Definition \ref{defn:strongFeller}) for the Lagrangian process $(u_t, x_t)$; and
	\item[(c)] For the matrix process $(u_t, x_t, A_t)$ and projective process $(u_t, x_t, v_t)$, either:
		\begin{itemize}
			\item[(i)]  H\"ormander's condition for the SDE defining $(u_t, x_t, A_t)$ for the finite-dimensional Systems \ref{sys:2DStokes}--\ref{sys:Galerkin};
			\item[(ii)] or approximate controllability condition in Proposition \ref{prop:approxControlOutline} for the infinite-dimensional Systems \ref{sys:NSE} -- \ref{sys:3DNSE}.
		\end{itemize}
\end{itemize}
Let us recall briefly where each of these is used. 
First, ingredient (a)(i) was used to deduce the almost-sure constancy of the exponential growth rates $\lambda^+, \check \lambda^+$
as in Proposition \ref{prop:lyapOutline} and Proposition \ref{prop:METforNegTrans}(a), respectively. Meanwhile, (a)(ii) was
used to deduce almost sure growth for $(D_x\phi^t)v$ in Corollary \ref{cor:expandAllDirections} (see Proposition \ref{prop:expandAllDirectionsOutline}); analogously,
 (a)(iii) was used to deduce growth of the $(D_x\phi^t)^{-\top}v$ in Proposition \ref{prop:expandAllDirectionsNegTrans}. On the other hand, (b) is used
to justify the refinement of Furstenberg's criterion (Proposition \ref{prop:refineFurstOutline}) used for Systems \ref{sys:NSE}--\ref{sys:3DNSE}. For the finite-dimensional Systems \ref{sys:2DStokes}, \ref{sys:Galerkin}, ingredient (c)(i) was used to rule out Furstenberg's criterion 
(Proposition \ref{prop:furstenbergOutline}); see the discussion in Section \ref{subsubsec:ruleOutFurstFDOutline}.
Lastly, ingredient (c)(ii) was used to rule out the refinement of Furstenberg's criterion in Proposition \ref{prop:refineFurstOutline}
for Systems \ref{sys:NSE} -- \ref{sys:3DNSE}.

All of items (a)-(c) require us to understand how the noise in the low modes of $u_t$ spread to the degrees of freedom associated with the Lagrangian flow. % such as in the matrix process $(u_t, x_t, A_t)$. 
Note the additional degrees of freedom  $(x_t,v_t,\check{v}_t,A_t)$ solve a series of random ODEs (collected below in equation \eqref{eq:defineSDE}). Since these unknowns are not directly forced by any noise, the corresponding SDE's are degenerate and we need to depend on hypoellipticity to show (a)-(c).

\medskip

\subsubsection{Finite dimensions: Systems \ref{sys:2DStokes} and \ref{sys:Galerkin}} 

Let us discuss how the ingredients for the finite-dimensional Systems \ref{sys:2DStokes}, \ref{sys:Galerkin} are obtained.
 For these models, all relevant stochastic processes as above are given by an SDE on a finite-dimensional manifold. Provided that one can show the algebra formed by taking successive Lie brackets of vector fields associated to the drift and the noise directions $e_k\gamma_k^i$ span the tangent space at every point, a condition known as {\em H\"{o}rmander's condition} (see Definition \ref{def:Hor} for a precise definition and Remark \ref{rem:hormanderPhysical} for a conceptual discussion), we may apply H\"{o}rmander's Theorem (see \cite{Hormander67,Hormander85} and the discussions in \cite{DaPrato14,HairerNotes11}) to deduce that the Markov transition kernels for the Lagrangian, projective, $(- \top)$-projective and matrix processes have a smooth positive density. Assumption \ref{a:lowms} ultimately ensures that H\"{o}rmander's condition is satisfied. Specifically we prove the following Proposition in Section \ref{sec:galerkin}:
\begin{proposition}\label{prop:hormanderForAll}
	Assume $(u_t)$ is governed by either of the finite-dimensional Systems \ref{sys:2DStokes} or \ref{sys:Galerkin}. 
	For each of (i) the Lagrangian process $(u_t, x_t)$, (ii) the projective process $(u_t, x_t, v_t)$, (iii)
	the matrix process $(u_t, x_t, A_t)$, and (iv) the $(-\top)$-projective process $(u_t, x_t, \check v_t)$,
	the SDE governing the relevant process satisfies H\"ormander's condition.
\end{proposition}
\noindent
 By standard arguments (see e.g. \cite{DPZ96}), uniqueness of the stationary measures then follows for the Lagrangian, projective and $(- \top)$-projective processes
 \cite{DaPrato14}, thereby fulfilling ingredients (a)(i) -- (iii) above as well as (b). Likewise (c)(i) is immediately satisfied for the matrix process.

\begin{remark}\label{rem:hormanderPhysical} 
Physically, one may view H\"{o}rmander's condition as an infinitesimal controllability statement.  When it is satisfied for the $(u_t, x_t, v_t, A_t)$ process, one can infinitesimally move each component of this process independently of the others using special choices of noise paths. Hence, all possible infinitesimal deformations of the flow map are realized with non-zero probability. 
\end{remark}

\medskip

\subsubsection{Infinite dimensions: Systems \ref{sys:NSE}--\ref{sys:3DNSE}}

In infinite dimensions, H\"ormander's condition is not applicable and so we must work harder to verify ingredients (a)(i) -- (iii). There have been a number of works proving uniqueness of the stationary measure for the Navier-Stokes equations under degenerate noise. A standard approach is to apply the Doob-Khasminskii Theorem \cite{Doob1948-rb, khasminskii1960ergodic}, the fact that
distinct ergodic stationary measures for strong Feller processes (Definition \ref{defn:strongFeller}) have disjoint supports,
% by proving the strong Feller property 
 and then to check that there exists a point which belongs to the support of every invariant measure (a.k.a. weak irreducibility). Following this strategy, in Section \ref{sec:Feller} we prove the strong Feller property for the Lagrangian, projective and $(-\top)$-projective processes. 

\begin{proposition}[Strong Feller] \label{prop:Feller}
For Systems \ref{sys:NSE}--\ref{sys:3DNSE}, the Markov semigroups associated with the Lagrangian process $(u_t,x_t)$ and the projective processes $(u_t,x_t,v_t)$, $(u_t,x_t,\check{v}_t)$ are all strong Feller in $\mathbf{H} \times \T^d \times P^{d-1}$.
\end{proposition} 
\begin{remark} \label{rem:sigmaStrongFeller}
This proposition is where we need the lower bound $\sigma > \alpha - 2(d-1)$ as in \eqref{eqn:sigmaConstraint}.
\end{remark} 

\begin{remark}
If the noise if suitably non-degenerate then the strong Feller property for the Navier-Stokes equations can be proved by the Bismut-Elworthy-Li formula (see for instance \cite{FM95} and \cite{Cerrai1999-ce}). However if the noise if too degenerate,
 it is not known whether the strong Feller property even holds. Indeed, to get around this difficulty, Hairer and Mattingly \cite{HM06,HM11} introduced a weaker notion, {\it the asymptotic strong Feller property}, which when combined with weak irreducibility, gives a generalization of the Doob-Khasminskii Theorem, still giving uniqueness of the stationary measure. While the asymptotic strong Feller property is clearly good enough obtain ingredients (a)(i) -- (iii), it {\em does not} appear to be enough to prove the refinement of Furstenberg's criterion (Proposition \ref{prop:refineFurstOutline}), which requires that $(u_t, x_t)$ be strong Feller (ingredient (b)). It is precisely this strong Feller requirement for Furstenberg's criterion that dictates our non-degeneracy Assumption \ref{a:Highs}. 
\end{remark}

To conclude uniqueness of the stationary measures as in (a)(i) -- (iii), it suffices to prove the following weak irreducibility properties, proved in Section \ref{sec:Irr} below. 
\begin{proposition} \label{prop:UniErg} 
For Systems \ref{sys:NSE}--\ref{sys:3DNSE} we have the following.
\begin{itemize}
	\item[(1)] The support of any stationary measure for the Lagrangian process $(u_t, x_t)$ on $\Hbf \times \T^d$ must contain
	the set $\{ 0 \} \times \T^d$.
	\item[(2)] The support of any stationary measure for the projective processes $(u_t, x_t, v_t)$, $(u_t, x_t, \check{v}_t)$ on $\Hbf \times \T^d \times P^{d-1}$ must
	contain $\{ 0 \} \times \T^d \times P^{d-1}$.
\end{itemize}
\end{proposition}

Uniqueness of the stationary measures now follow.
\begin{corollary} 
The processes $(u_t), (u_t,x_t), (u_t,x_t,v_t),$ and $(u_t,x_t,\check{v}_t)$ all have unique stationary measures. 
\end{corollary}
Additionally, it remains to address ingredient (c)(ii), the approximate controllability condition in Proposition \ref{prop:approxControlOutline}.
Once Propositions \ref{prop:UniErg} and \ref{prop:approxControlOutline} are completed, the proof of Theorem \ref{thm:Lyap} for System \ref{sys:NSE} is complete.

\subsubsection{Strong Feller}\label{subsubsec:feller}
Our proof of Proposition \ref{prop:Feller} is inspired by the methods of Eckmann and Hairer \cite{EH01}. In \cite{EH01}, the authors prove strong Feller for the complex Ginzburg-Landau equations with forcing that satisfies Assumption \ref{a:Highs}, using a cut-off technique and a high-low frequency splitting. This cut-off approach has since been extended to Markov selections of the 3D Navier-Stokes equations in \cite{RomitoXu11}. Similar results to \cite{RomitoXu11} were proved in \cite{Albeverio2012-or} using the infinite dimensional Kolmogorov equation. Our proof of strong Feller is closer to \cite{EH01} and \cite{RomitoXu11}, but differs in our choice of the cut-off process, the use of non-adapted controls, estimates on Skorohod integrals, and an interpolation inequality introduced in \cite{HM11} used to circumvent some technicalities with applying Norris's Lemma in $L^2([0,1])$.

Similarly to \cite{RomitoXu11,EH01,FM95}, it does not seem possible to obtain an estimate on the derivative of the Markov semigroup of the projective process $(u_t,x_t,v_t)$. The strategy is to show that such an estimate is available for a ``cut-off'' or ``regularized'' process.  In our setting, we will find it convenient to augment the projective process $(u_t,x_t,v_t)$ by a Brownian motion $(z_t)$ on $\R^{2d}$ (likewise for the $(-\top)$ projective process). The augmented process $w_t = (u_t,x_t,v_t,z_t)$ solves an abstract evolution equation
\begin{equation}
 \partial_t w_t = F(w_t) + Aw_t + Q\dot{W}_t
 \end{equation} 
on $\Hbf\times \cM$ where $\cM$ is a smooth finite dimensional manifold. Let $\widehat{P}_t$ be the Markov semi-group associate to $w_t$, then our goal is to find a regularized process $w^\rho_t$ such that $\P\left((w_t)_{t\in[0,T]} \neq (w^\rho_t)_{t\in[0,T]}\right)$ is vanishingly small as $\rho \rightarrow \infty$ but for which one can obtain a derivative estimate on the associated semigroup $\widehat{P}^\rho_t$. 
\begin{remark}
It is important to note that our choice of cut-off process $w^\rho_t$ is different from that used in \cite{EH01} and \cite{RomitoXu11} and uses the augmentation by $z_t$ to introduce new sources of noise while avoiding technical difficulties with multiplicative white noise (see Section \ref{sec:Feller} for more details on the cut-off process).
\end{remark} 
Our main effort is then to prove that the cut-off semi-group $\widehat{P}^\rho_t$ satisfies the following gradient estimate (Proposition \ref{prop:GradPhi})
\begin{equation}\label{eq:grad-est-trunc_intro}
\|D \widehat{P}_t ^\rho \phi(w)\|_{\Hbf\times T_v\cM} \leqc_\rho t^{-a_\ast} \left(1 +  \norm{w}_{\Hbf}^{b_\ast}\right) \norm{\phi}_{L^\infty}
\end{equation}
for all bounded measurable $\phi$ on $\Hbf\times\cM$ and sufficiently small $t$, and $a_*$ and $b_*$ are certain constants.  We show in the proof of Proposition \ref{prop:Feller} in Section \ref{sec:Feller} this estimate on $\widehat{P}^\rho_t$ implies that $\widehat{P}_t$ is strong Feller, albeit without an estimate on the derivative.

The fundamental tool for proving \eqref{eq:grad-est-trunc_intro} is Malliavin calculus. This involves taking derivatives of the solution with respect to the noise. Well-posedness of the cutoff process implies that for each $\rho$ and initial data $w\in \Hbf\times\cM$, the solution $w_t^\rho$ at time $t>0$ is a continuous function of the noise path $W\in C(\R_+,\Wbf)$. Specifically, we have $W|_{[0,t]}\mapsto w_t^\rho\left[W|_{[0,t]}\right]$
is a continuous mapping from $C([0,t],\Wbf)$ to $\Hbf\times \cM$ for each $t>0$. In fact, it is straightforward to show that $W|_{[0,t]}\mapsto w_t^\rho\left[W|_{[0,t]}\right]$ is actually differentiable over the Banach space $C([0,t];\Wbf)$ (see for instance \cite{HM11} Proposition 4.1). Indeed, for any process $g = (g_t)$ (not necessarily adapted to $\F_t$) that belongs almost surely to $L^2(\R_+,\Wbf)$, the Malliavin derivative $\MalD_g w_t^\rho$ of $w_t^\rho$ in the direction of $g$, defined by
\begin{equation}\label{eq:MalD-def}
	\MalD_g w_t^\rho = \frac{\dee}{\dee h} w_t^\rho\!\left[W + hG \right]\big|_{h=0},\quad G = \int_0^\cdot g_s\ds,
\end{equation}
exists almost surely for each $t>0$. We will often refer to $g$ as a \emph{control}. A key feature of the Malliavin derivative is the celebrated Malliavin integration by parts formula, which states that for each $\phi\in C^1(\Hbf\times\cM)$ and a suitably regular $g$ (see Proposition \ref{lem:MalIBP} for the precise conditions) one has
\begin{equation}\label{eq:MalIBP-intro}
	\E\left(D \phi(w_t^\rho) \MalD_g w_t^\rho\right) =  \E  \, \MalD_g \phi(w^\rho_t)= 
	 \E\left(\phi(w_t^\rho)\int_0^t \langle g_s , \delta W_s\rangle \right),
\end{equation}
where the stochastic integral $\int_0^t \langle g_s , \delta W_s\rangle$ above denotes the Skorohod integral (see, e.g., Definition 1.3.1 in \cite{Nualart06} or Section 11.3 in \cite{DaPrato14}). If $g$ is adapted to the filtration $\mathcal{F}_t$ then the Skorohod integral coincides with the usual It\^{o} integral. The formula \eqref{eq:MalIBP-intro} can be used to obtain smoothing estimates on the semi-group $\widehat{P}_t^\rho$. Indeed, if for every $h\in \Hbf\times T_v \cM$ one could find a ``nice enough'' control $g$ such that $\MalD_g w_t^\rho = Dw_t^\rho h$, where $Dw_t^\rho h$ denotes the direction derivative of $w_t^\rho$ in the direction $h$ with respect to the initial data, then an estimate on $D\widehat{P}^\rho_t$ follows  from \eqref{eq:MalIBP-intro} as long as one can bound the Skorohod integral term (see \eqref{eq:MallIntro} below for more details). However, in our setting we are unable to find such a control $g$ due to subtleties in infinite dimensions. Instead we opt to find a control $g$ such that for each fixed $0<T<1$, we have 
\begin{equation}\label{eq:Mal+remainder}
	\MalD_g w_T^\rho = Dw_T^\rho h + r_T
\end{equation}
where $r_T$ is a remainder which will be small when $T$ is small, and consequently the Skorohod integral $E\left|\int_0^t \langle g_s , \delta W_s\rangle\right|$ will be singular as $T$ approaches $0$ (see Lemma \ref{lem:RTbd+SkrdBd} for the exact estimates). The (non-adapted) control $g$ is chosen with an elaboration of the high-low splitting used in \cite{EH01}. At high frequencies it is chosen such that the contribution to the Malliavin integration by parts formula reduces to the Bismut-Elworthy-Li formula, while at lower frequencies, the control is set by inverting a finite-dimensional approximation of the Malliavin matrix (the partial Malliavin matrix) while attempting to minimize the amount by which the low frequency control perturbs the higher frequencies. The invertibility of the partial Malliavin matrix can be deduced from the fact that the projective process associated to finite dimensional approximations of the Navier-Stokes equations satisfy H\"{o}rmander's condition (shown in Section \ref{sec:galerkin}).

The fact that we can have a remainder in \eqref{eq:Mal+remainder} and can still prove a smoothing estimate depends heavily on the precise dependence of the bounds on $r_T$ and the Skorohod integral.  The key idea, inspired by \cite{EH01} and \cite{Cerrai1999-ce} involves using the semi-group property and the integration by parts formula \eqref{eq:MalIBP-intro} to write
\begin{equation}
\begin{aligned}
	D \widehat{P}_{2T}^\rho\phi(w)h &=
	\E \big(D\widehat{P}_{T}^\rho\phi(w_T) Dw_T h \big) \\  
	& = \E\left(\widehat{P}_T^\rho\phi(w_T^\rho)\int_0^T \langle g_t, \delta W(t)\rangle_{\Wbf}\right) - \E \left(D \widehat{P}_{T}^\rho\phi(w_T^\rho)r_T\right). 
	 \end{aligned}\label{eq:MallIntro}
\end{equation}
Using the estimates on $r_T$ and the Skorohod integral one can close estimates on $D \widehat{P}_{t}^\rho\phi$ for sufficiently short times. The details of this argument can be found in the proof of Proposition \ref{prop:GradPhi}.

\subsubsection{Weak irreducibility and approximate control}

Let us first discuss Proposition \ref{prop:UniErg}. For simplicity, let us here only discuss the 2D case, System \ref{sys:NSE}. Weak irreducibility for $(u_t)$ is a consequence of the energy/enstrophy dissipation (see Section \ref{sec:Irr} and, e.g. \cite{E2001-lg}), which shows that $0$ is in the support of all stationary measures for the $(u_t)$ processes. 
Using a stability argument and the positivity of the Weiner measures, the main content of the irreducibility in Proposition \ref{prop:UniErg} is the study of the control problem 
\begin{align}
\partial_t u_t + B(u_t,u_t) = - A u_t + Q g(t) \, , \label{eq:ctrl}
\end{align} 
where $g \in C^\infty(\R_+, \Wbf)$ is a smooth control, $A = - \Delta$ and $B(u, u) =  (I + \grad (-\Delta)^{-1}\grad \cdot)(u\cdot \grad u)$.  Here, $x_t, v_t, \check v_t$ and $A_t$ are implicitly
controlled through $(u_t)$. 
First, we prove that for all $(x,v), (x',v') \in \T^d \times P^{d-1}$, there exist smooth controls $g$ such that
\begin{align}
(u_0,x_0,v_0) = (0,x,v), \quad (u_1,x_1,v_1) = (0,x',v').
\end{align}
(and analogously for the $(u_t,x_t,\check{v}_t)$ process). 
We note that it suffices to control near $u_t \approx 0$ precisely because $0$ is in the support of the stationary measure $\mu$. 
To solve this control problem we use that the following flows are \emph{exact} solutions (for arbitrary $a,b$) of the 
steady Euler equation $B(u, u) = 0$ as well as eigenfunctions of $A$: %$(u_t)$ equation for suitable choices of $Qg$:
\begin{align}
u(y_1, y_2) =  \begin{pmatrix} \cos(y_2 - b) \\ 0 \end{pmatrix}, \quad  \begin{pmatrix} 0 \\ \cos(y_1 - a) \end{pmatrix}, \quad \begin{pmatrix} \sin(y_2 - b) \\ -\sin(y_1-a) \end{pmatrix}. \label{eq:flows}
\end{align}
The first two are shear flows whereas the last flow is a cellular flow with separatices aligned along the diagonals. The first two flows are used to move the particle $x_t$ whereas the latter flow is used to move $v_t$ without moving the particle. 
Once these flows can be formed, it is not difficult to verify the necessary controllability of System \eqref{eq:ctrl}; see Lemma \ref{lem:Ctrluxv} for details. Note that Assumption \ref{a:lowms} is slightly stronger than what is necessary to form the flows \eqref{eq:flows}, which is why, for example, Remark \ref{rmk:StokesDegn} holds (see Lemma \ref{lem:Ctrluxv} and Remark \ref{rmk:2Dstokes}). 
Similarly, for the case of Systems \ref{sys:NSE} and \ref{sys:3DNSE}, one can prove Theorem \ref{thm:Lyap} (and all of the other main results) using only Assumption \ref{a:Highs}; see Remark \ref{rmk:HighOnly}. 

The non-degeneracy of the $(u_t,x_t,v_t)$ and $(u_t,x_t,\check{v}_t)$  processes needed to prove Proposition \ref{prop:UniErg} and Condition (b) in  Proposition \ref{prop:approxControlOutline} then follow from the controllability and suitable stability estimates (see Section \ref{sec:Irr} for details). 
In order to satisfy Condition (a) in Proposition \ref{prop:approxControlOutline} we also need to demonstrate arbitrarily large growth of $A_t$ in the $(u_t,x_t,A_t)$ process (under similar constraints as for the projective control statements). This is done by applying the cellular flow as above, but shifted so that the hyperbolic fixed point causes exponential growth of $A_t$ without moving the particle $x_t$; see Proposition \ref{prop:Actrl} for details. 

\subsection{Proof of Yaglom's Law \eqref{eq:Yaglom} as in Theorem \ref{thm:turb} (ii)}

Next, we summarize the proof of Theorem \ref{thm:turb} (see Section \ref{sec:turbulence} for details). 
First, we prove the estimate \eqref{eq:WAD}. This result follows from a straightforward adaptation of the compactness method of \cite{BCZGH}, originally applied to passive scalars with deterministic, constant-in-time velocity fields.
The first step is to renormalize $f^\kappa_t = \sqrt{\kappa} g_t$ to obtain 
\begin{align}
\partial_t f^\kappa_t + u \cdot \grad f^\kappa_t = \kappa \Delta f^\kappa_t + \sqrt{\kappa} \widetilde{Q} \dot{\widetilde{W}}_t. \label{eq:frenorms}
\end{align}
The balance \eqref{eq:SclBal} then becomes, for statistically stationary solutions,  
\begin{align}
\EE \norm{\grad f^\kappa}_{L^2}^2 = \bar{\eps}. \label{eq:BalRenorm}
\end{align}
Denote by $\set{\bar{\mu}^\kappa}_{\kappa > 0}$ a sequence of stationary measures to \eqref{eq:frenorms} supported on $\hat \Hbf \times H^1$. The bound \eqref{eq:BalRenorm} is sufficient to obtain tightness of $\set{\bar{\mu}^\kappa}_{\kappa > 0}$ to pass to the limit and deduce the existence of a stationary measure $\bar{\mu}^0$ of the problem \eqref{eq:frenorms} with $\kappa = 0$ supported on $\hat \Hbf \times H^1$. Theorem \ref{thm:ExpGrwth} is then applied to prove by contradiction that necessarily $\bar{\mu}^0 = \hat \Hbf \times \delta_0$ (where $\delta_0$ denotes the Dirac delta centered at zero). The limit \eqref{eq:WAD} then follows from additional moment bounds in $L^2$; see Section \ref{sec:turbulence} for more details.

In order to prove \eqref{eq:Yaglom} we in turn adapt the method  of \cite{BCZPSW18} .
One of the basic identities used in \cite{BCZPSW18} is a version of the classical K\'arm\'an-Horvath-Monin relation \cite{deKarman1938,MoninYaglom,Frisch1995} which is a refinement of the $L^2$ energy balance. 
Here, we apply a similar identity, now a refinement of the $L^2$ balance for $g_t$ (see Proposition \ref{prop:KHM} below).
This identity implies a differential equation (in weak form) for the quantity (see \eqref{eq:ODE}), 
\begin{align}
\bar{\mathfrak{D}}(\ell) = \E \fint_{\T^{d}}\fint_{\mathbb S^{d-1}} \delta_{\ell n} u \cdot n \abs{\delta_{\ell n} g}^2 \dee S(n) \dx \, .
\end{align}
Solving the ODE \eqref{eq:ODE} in terms of the source and dissipation, we apply \eqref{eq:WAD} to show that the effect of the diffusivity on the balance vanishes over an appropriate range of scales $[\ell_D(\kappa), \ell_I]$ satisfying $\lim_{\kappa \rightarrow 0} \ell_D(\kappa) = 0$. This then yields \eqref{eq:Yaglom}. 

\subsection{A guide to notation}\label{subsec:notation}

\begin{itemize}
    \item We use the notation $f \lesssim g$ if there exists a constant $C > 0$ such that $f \leq Cg$ where $C$ is independent of the parameters of interest. Sometimes we use the notation $f \approx_{a,b,c,...} g$ to emphasize the dependence of the implicit constant on the parameters, e.g. $C = C(a,b,c,...)$. We denote $f \approx g$ if $f \lesssim g$ and $g \lesssim f$. 
    \item Throughout, $\R^d$ is endowed with the standard Euclidean inner product $(\cdot, \cdot)$ and corresponding norm $|\cdot|$. We continue
	to write $|\cdot|$ for the corresponding matrix norm. We use $\abs{k}_{p}$ to denote the $\ell^p$ norms.   
    \item When the domain of the $L^p$ space is omitted it is always understood to be $\T^d$: $\norm{f}_{L^p} = \norm{f}_{L^p(\T^d)}$. We use the notations $\EE X = \int_{\Omega} X(\omega) \PP(d\omega)$ and $\norm{X}_{L^p(\Omega)} = \left(\EE \abs{X}^p \right)^{1/p}$. We use the notation $\norm{f}_{H^s} = \sum_{k \in \Z^d} \abs{k}^{2s} \abs{\hat{f}(k)}^2$ (denoting $\hat{f}(k) = \frac{1}{(2\pi)^{d/2}} \int_{\T^d} e^{-ik\cdot x}f(x) dx$ the usual complex Fourier transform).  
	\item If $M$ is a Riemannian manifold, we write $\Leb_M$ for the Lebesgue volume on $M$. For short, we write $\Leb$ for the normalized Lebesgue measure on $\T^d$. 
	\item For $d \geq 1$, we write $M_{d \times d}(\R)$ for the space of real $d \times d$ matrices, and $SL_d(\R)$ for the subgroup of matrices
	of determinant 1. 
	\item We write $P^{d-1} = P(\R^d)$ for the real projective space of $\R^d$, i.e., the manifold of equivalence classes
	of vectors in $\R^d \setminus \{0\}$ up to scaling. When it is clear from context, we will abuse notation and
	intentionally confuse an element $v \in P^{d-1}$ with a unit vector representative $v \in \R^d$, and vice versa. Likewise $\S^{d-1}$ denotes the unit sphere in $\R^d$.
	\item Given a matrix $B \in M_{d \times d}(\R)$ we
	use the same symbol $B : P^{d-1} \to P^{d-1}$ to denote the corresponding map on projective space. If $\nu$ is a probability measure on $P^{d-1}$, we write $B_* \nu := \nu \circ B^{-1}$ for the pushforward of $\nu$ by $B$. 
	\item For $\sigma \in  (\alpha-2(d-1), \alpha - \frac{d}{2})$ fixed, we write $\Hbf$ for the subspace of $H^\sigma$ divergence-free, mean-zero
	 vector fields on $\T^d, d = 2$ or 3.
	Given $N \geq 1$ as in System \ref{sys:Galerkin}, we write $\Hbf_N \subset \Hbf$ for the span of all Fourier modes $k$ with $|k|_{\infty} \leq N$. Given $\mathcal K \subset \Z^d$
	as in Assumption \ref{a:lowms}, we write $\Hbf_{\mathcal K} \subset \Hbf$ for the span of all Fourier modes in $\mathcal K$.
	\item Given the vector field process $(u_t)$ on $\hat \Hbf$ governed by Systems \ref{sys:2DStokes}, \ref{sys:Galerkin}, \ref{sys:NSE} or \ref{sys:3DNSE}, we write $(u_t, x_t)$
	for the \emph{Lagrangian process} on $\hat \Hbf \times \T^d$ as defined by $x_t = \phi^t(x_0), \phi^t$ as in \eqref{eq:xtintro}, 
	where $\hat \Hbf$ is the appropriate space of vector fields as above. We write $\Theta^t_\omega : \hat \Hbf \times \T^d \to \hat \Hbf \times \T^d, t \geq 0$
	for the corresponding RDS as defined in Section \ref{subsubsec:RDSframeworkOutline}. We write $(u_t, x_t, v_t)$ for the \emph{projective process} on
	$\hat \Hbf \times \T^d \times P^{d-1}$ as defined in Section \ref{subsubsec:expandAllDirOutline}, and $(u_t, x_t, A_t)$ for the \emph{matrix process}
	on $\hat \Hbf \times \T^d \times SL_d(\R)$ as defined in Section \ref{subsubsec:ruleOutFurstFDOutline}.
	These processes are governed by $(u_t)$ as in Systems \ref{sys:2DStokes} -- \ref{sys:3DNSE} and
	the random ODE
\begin{subequations}  \label{eq:defineSDE} 
\begin{align}
\partial_t x_t & = u_t(x_t), \label{eq:xctrl} \\ 
\partial_t v_t  & = \Pi_{v_t} \grad u_t (x_t) v_t \, ,  \label{eq:vctrl} \\
\partial_t\check{v}_t & = - \Pi_{\check{v}_t} \grad u_t(x_t)^{\top}  \check{v}_t,\\
\partial_t A_t & = \grad u_t (x_t) A_t,
\end{align}
\end{subequations}
where $\Pi_v = \Id - v \tensor v$ is the orthogonal projection from $\R^d$ onto the tangent space of $\S^{d-1}$ (viewing $v$ as a unit vector in $\R^d$).

\item We denote by $B(u,u) = \left(I + \grad (-\Delta)^{-1} \grad \cdot \right)\left(u \cdot \grad u\right)$ the Euler nonlinearity in both 2D and 3D. We similarly denote $A = -\nu \Delta + \eta \Delta^2$ in 3D and $A =  -\nu \Delta$ in 2D. 
\end{itemize}

%% RDS Stuff
\section{Random dynamical systems preliminaries}\label{sec:RDS}
%!TEX root = master.tex

\newcommand{\Bor}{\operatorname{Bor}}
\newcommand{\N}{\mathbb N}
\newcommand{\codim}{\operatorname{codim}}
\newcommand{\Ac}{\mathcal A}
\newcommand{\Sc}{\mathcal S}
\newcommand{\Cc}{\mathcal C}
\newcommand{\Pc}{\mathcal P}
\newcommand{\Fc}{\mathcal F}
\newcommand{\Tc}{\mathcal T}

In this section we will present necessary background from random dynamical systems theory. This section is mostly an exposition of 
material drawn from various sources in the dynamics literature.
General references include the books of Arnold \cite{arnold2013random}, 
Kifer \cite{kifer2012ergodic}, and Kuksin \& Shirikyan \cite{KS}. 

The plan for Section \ref{sec:RDS} is as follows. We begin in Section \ref{subsec:RDSelements}
with some essential ergodic-theoretical background: the definition and standard axioms we use
for random dynamical systems (RDS) and some elementary results.
Section \ref{subsec:LinearCocycles} introduces the notion of \emph{linear cocycle} over a given RDS
and formulates the Multiplicative Ergodic Theorem (MET), allowing us to define the Lyapunov exponent
$\lambda^+$ appearing in Theorem \ref{thm:Lyap}. 
In Section \ref{subsec:FurstenbergFD}  
we turn our attention to the problem of how to prove $\lambda^+ > 0$ using Furstenberg's criterion (Theorem \ref{thm:furstenberg}).

\subsection{Elements of ergodic theory of random dynamical systems}\label{subsec:RDSelements}

\subsubsection{Basic setup for random dynamics}\label{subsubsec:basicSetupRDS}

Let $(\Omega, \Fc, \P)$ be a probability space and let 
$(\theta^t)$ be a measure-preserving semiflow on $\Omega$, i.e.,
$\theta : [0,\infty) \times \Omega \to \Omega, (t, \omega) \mapsto \theta^t\omega$ is a measurable mapping satisfying (i) $\theta^0 \omega \equiv \omega$ for all $\omega \in \Omega$;
(ii) $\theta^t \circ \theta^s = \theta^{t + s}$ for all $s, t \geq 0$, and (iii) $\P \circ (\theta^t)^{-1} = \P$ for all $t \geq 0$.
At times (which we will specify), it will be useful to assume that $\Omega$ has some topological structure. If so, we will assume additionally that $\Omega$ is a Borel subset of a Polish space, and $\Fc$ is the set of Borel subsets of $\Omega$.

\smallskip

Let $(Z, d)$ be a separable and complete metric space. A \emph{random dynamical system}
or RDS on $Z$ is an assignment to each $\omega \in \Omega$ of a mapping
$\Tc_\omega : [0,\infty) \times Z \to Z$ satisfying the following basic properties.

\begin{itemize}
\item[(i)] (Measurability) The mapping $\Tc : [0,\infty) \times \Omega \times Z \to Z$, $(t, \omega, z) \mapsto \Tc_\omega^t z$, 
is measurable with respect to $\Bor([0,\infty)) \otimes \Fc \otimes \Bor(Z)$ and $\Bor(Z)$.
\item[(ii)] (Cocycle property) For all $\omega \in \Omega$, we have $\Tc^0_\omega = \Id_Z$ (the identity mapping 
	on $Z$), and for $s, t \geq 0$, we have $\Tc^{s + t}_\omega = \Tc^t_{\theta^s \omega} \circ \Tc^s_\omega$. 
\end{itemize}

\begin{itemize}
	\item[(iii)] (Continuity) For all $\omega \in \Omega$, the mapping $\Tc_\omega : [0,\infty) \times Z \to Z$
	belongs to $C_{u, b}([0,\infty) \times Z, Z)$.
\end{itemize}

Here, for metric spaces $V, W$, the space $C_{u, b}(V, W) \subset C(V, W)$ is defined as follows:
\begin{definition}
We define\footnote{We use the slightly non-standard topology $C_{u, b}(Z, Z)$ to accommodate for the situation when $Z$ is not locally compact. The regularity of $C_{u, b}$-topology is used in several places, especially in Section \ref{subsec:furstInfiniteDimensions}, and so will be assumed from this point on.} $C_{u,b}(V, W)$ to be the space of continuous maps $F : V \to W$ for which the following holds for each bounded $U \subset V$: 
 \begin{itemize}
 	\item[(a)] The restriction $F|_U$ is uniformly continuous; and
	\item[(b)] the image $F(U)$ is a bounded subset of $W$.
 \end{itemize}
We endow $C_{u, b}(V, W)$ with the topology of uniform convergence on bounded sets (abbreviated UCBS).
It is a simple exercise to check that if $(F_n)_n$ is a sequence in $C_{u, b}(V, W)$ converging to
some $F : V \to W$ in the UCBS mode, then $F \in C_{u, b}(V, W)$ holds.
Moreover, it is a simple exercise to check in this setting that $C_{u, b}(V, W)$ is metrizable.
\end{definition}

Note that automatically, condition (iii) implies that 
$\Tc_\omega^t \in C_{u, b}(Z, Z)$ for all $t \geq 0, \omega \in \Omega$. Indeed, by (iii), for any 
$\omega \in \Omega, T > 0$ and bounded $U \subset Z$, the family $\{ \Tc^t_\omega|_U : U \to Z\}_{t \in [0,T]}$ is
equicontinuous.

\begin{definition}
We refer to $\Tc$ satisfying (i) -- (iii) above as a {\em continuous RDS} on $Z$. 
\end{definition}

In addition to (i) -- (iii) above, we will almost always assume that the RDS $\Tc$
satisfies the following \emph{independent increments
assumption}.

\begin{itemize}
	\item[(H1)] For all $s, t > 0$, we have that $\Tc^t_\omega$ is independent of $\Tc^s_{\theta^t \omega}$. That is,
	the $\sigma$-subalgebra $\sigma(\Tc^t_\cdot) \subset \Fc$ generated by the $C_{u, b}(Z, Z)$-valued random variable
	 $\omega \mapsto \Tc^t_\omega$ 
	is independent of the $\sigma$-subalgebra $\sigma(\Tc^s_{\theta^t \cdot})$
	 generated by $\omega \mapsto T^s_{\theta^t \omega}$.
\end{itemize}

\begin{example}\label{exam:SDERDS}
	Let $n \geq 1$ and let $Y_0, Y_1, \cdots, Y_m$ be smooth, globally Lipschitz
	 vector fields on $\R^n$. Let $W^1_t, \cdots, W^m_t$ be independent standard Brownian motions.
	Then, the stochastic differential equation
	\[
	\dee X_t = Y_0 (X_t) \dt + \sum_{i =1 }^m Y_i(X_t) \dee W^i_t
	\]
	defines a random dynamical system on $Z = \R^n$, 	where
	 $\Omega = C_0([0,\infty), \R)^{\otimes m}$ is the $k$-fold product of
	Canonical Spaces equipped with the $k$-fold product Borel $\sigma$-algebra and Wiener measure $\P$, 
	and $\theta^t : \Omega \to \Omega$
	is the leftward shift by $t \geq 0$. The resulting RDS satisfies the measurability and continuity
	 conditions (i) -- (iii). The independent increments condition (H1) follows from the independence of the Brownian increments
	 $W^i_{s + t} - W^i_t$ and $W^j_t$ for all $s, t > 0$ and each $1 \leq i , j \leq k$.
	 See, e.g., \cite{kunita1996stochastic, arnold2013random} for more details.
\end{example}

\subsubsection{Markov chain formulation and stationary measures}

For fixed $z \in Z$, consider the stochastic process $(z_t)_{t \geq 0}$ given by $z_t = \Tc^t_\omega z_0, z_0 := z$.

\begin{lemma}\label{lem:MarkovProperty}
	Let $\Tc$ be a continuous RDS as in Section \ref{subsubsec:basicSetupRDS} satisfying the independent increments condition (H1). 
	Then, the process $(z_t)_{t \geq 0}$ as above is \emph{Markovian}.
\end{lemma}
For a proof of Lemma \ref{lem:MarkovProperty}, see, e.g., Kuksin-Shirikyan \cite{KS} where the Markov property is proved under a somewhat weaker hypothesis than (H1).

For $t > 0, z \in Z$ and 
	 $K \in \Bor(Z)$, we define the {\em Markov kernel}
	\[
	P_t(z, K):= \P(z_t \in K | z_0 = z).
	\]
The Markov kernel $P_t(z,K)$ has a natural action on any bounded measurable observable $h : Z \to \R$
\[
P_t h (z) := \int h(z') P_t(z, \dz') \, .
\]
The Markov property of $(z_t)$ implies the semigroup relation $P_{s + t} = P_t \circ P_s$. We refer to the operators $(P_t)_{t>0}$ as the {\em Markov semi-group} asssociated to $(z_t)$.

The proof of the following proposition is straightforward and omitted for brevity.

\begin{proposition}\label{prop:semigroupRegularity1} Assume the setting of Lemma \ref{lem:MarkovProperty}.
	\begin{itemize}
		\item[(a)] The semigroup $(P_t)$ has the {\em Feller property}, i.e., for any $t \geq 0$ and any 
		$h : Z \to \R$ be continuous and bounded, we have that $P_t h$ is defined and is a continuous 
		function $Z \to \R$.
		\item[(b)] The semigroup $(P_t)$ is a $C^0$-semigroup on $C_{u, b}(Z, \R)$. That is, for any
		fixed $h \in C_{u, b}(Z, \R)$, we have that (1) $P_t h \in C_{u, b}(Z, \R)$ for all $t > 0$, and (2)
		the mapping $t \mapsto P_t h, t \geq 0$ is continuous in the topology on $C_{u, b}(Z, \R)$. 
	\end{itemize}
\end{proposition}

We regard the (formal) dual $(P_t)^*$ of the operator $P_t$ as acting on the space of finite signed Borel measures
on $Z$. Given a finite signed Borel $\mu$ on $Z$, $(P_t)^* \mu$ is defined for Borel $A \subset Z$ by
\[
(P_t)^* \mu(A) = \int P_t(z, A) \dee \mu(z) \, .
\]
If $\mu$ is a (positive) Borel probability on $Z$ for which $(P_t)^* \mu = \mu$ for all $t \geq 0$, we call $\mu$ \emph{stationary}.

The following Lemma is a consequence of a standard Krylov-Bogoliubov argument.

\begin{lemma}\label{lem:stationaryMeasureExists}
Assume the setting of Lemma \ref{lem:MarkovProperty}.
Then, the Markov semigroup $(P_t)$ admits at least one stationary measure $\mu$ in either of the following circumstances:
	\begin{itemize}
		\item[(a)] The space $Z$ is compact; or
		\item[(b)] there exists a Borel probability $\mu_0$ for which the sequence $(P_t)^* \mu_0$ is tight.
	\end{itemize}
\end{lemma}

\subsubsection{Skew product formulation and invariant measures}\label{subsubsec:skewProd}

The material in Section \ref{subsubsec:skewProd} is mostly taken from Chapter I of \cite{kifer2012ergodic}.

The Markov chain formulation given above is useful in that it identifies `time-invariant' statistics on $Z$ 
for the RDS, namely, its stationary measures. On the other hand, the Markov kernel loses some structure of the RDS, in the 
sense that the same Markov kernel can arise from qualitatively different RDS. See, e.g., Example I.1.1 of \cite{kifer2012ergodic} for an extreme example of this. 

The following \emph{skew product} formulation, unlike the Markov chain, 
encodes the entire RDS.

\begin{definition}
	The {\em skew product} associated to the above random dynamics is the mapping
	$\tau : [0,\infty) \times \Omega \times Z \to \Omega \times Z$ given by $\tau(t, \omega, z) = \tau^t(\omega, z) = (\theta^t \omega, T^t_\omega z)$.
\end{definition}

We regard $\tau$ as a single ``deterministic'', measurable semiflow on the augmented space $\Omega \times Z$.
In particular, this provides us a connection between ``standard'' ergodic theory, i.e., the theory of invariant measures
for individual mappings of a measurable space, and our present setting of random dynamical systems. 
The following Lemma makes this connection explicit.

Recall that a probability measure $\eta$ on $\Omega \times Z$
is \emph{invariant} for the semiflow $\tau$ if $\eta \circ (\tau^t)^{-1} = \eta$ for all $t \geq 0$.

\begin{lemma}[Lemma I.2.3
in \cite{kifer2012ergodic}]\label{lem:invariantMeasCharacterization}
	Assume $\Tc$ is a continuous RDS as in Section \ref{subsubsec:basicSetupRDS} satisfying (H1)
	and generating the Markov semigroup $(P_t)$ as in Lemma \ref{lem:MarkovProperty}.
	Let $\mu$ be a Borel probability measure on $Z$. Then, the following are equivalent.
		\begin{itemize}
			\item[(a)] The measure $\P \times \mu$ is invariant for the skew product $(\tau^t)$.
			\item[(b)] The measure $\mu$ is stationary for the Markov semigroup $(P_t)$.
		\end{itemize}
\end{lemma}

A similar correspondence exists between the \emph{ergodic} stationary measures of the 
semigroup $(P_t)$ and the ergodic invariant measures of the skew product $(\tau^t)$.

Recall the following standard definition from ergodic theory (see, e.g., \cite{walters2000introduction}):
 a $(\tau^t)$-invariant measure $\eta$ is \emph{ergodic} if,
for any bounded measurable $h : \Omega \times Z \to \R$ for which
$h \circ \tau^t = h$ holds $\eta$-almost-surely for all $t \geq 0$, we have
that $h$ is constant $\eta$-almost surely.
For stationary measures $\mu$ of the Markov semigroup $(P_t)$, we use the following
definitions:
\begin{definition}[pg. 19 of \cite{kifer2012ergodic}]\label{defn:PtmuInvariant}
	Let $h : Z \to \R$ be bounded and Borel measurable. Given a stationary $\mu$, we say that $\phi$ is $(P_t, \mu)$-\emph{invariant} if $P_t \phi = \phi$ holds $\mu$-almost surely for all $t \geq 0$. 
We say that a set $K \subset Z$ is $(P_t, \mu)$ invariant if its characteristic function $\chi_K$ is $(P_t, \mu)$-invariant
in the above sense.

We call a stationary measure $\mu$ \emph{ergodic} if the only $(P_t, \mu)$-invariant functions are $\mu$-almost-surely constant.
\end{definition}

\begin{proposition}[Theorem I.2.1 in \cite{kifer2012ergodic}]\label{prop:characterizeErgodicityRDS}
Assume the setting of Lemma \ref{lem:invariantMeasCharacterization}. Let $\mu$ be a stationary
measure for $(P_t)$, noting that $\P \times \mu$ is an invariant measure for $(\tau^t)$ by
Lemma \ref{lem:invariantMeasCharacterization}. Then, the following are equivalent.

\begin{itemize}
	\item[(a)] The invariant measure $\P \times \mu$ is ergodic for the skew product $(\tau^t)$.
	\item[(b)] The stationary measure $\mu$ is ergodic for the Markov semigroup $(P_t)$.
\end{itemize}
\end{proposition}

\subsection{Linear cocycles over RDS and the Multiplicative Ergodic Theorem}\label{subsec:LinearCocycles}

We start by defining and motivating the concept of a linear cocycle over a random dynamical system in Section \ref{subsubsec:linearCocycleBasicRDS}.
Next, in Section \ref{subsubsec:METRDS} we state precisely the Multiplicative Ergodic Theorem (Theorem \ref{thm:MET}). The remainder of Section \ref{subsec:LinearCocycles}
is devoted to establishing useful Corollaries and refinements of Theorem \ref{thm:MET}.

\subsubsection{Basic setting: Linear cocycles over RDS}\label{subsubsec:linearCocycleBasicRDS}

Fix a positive integer $d$. Roughly speaking, a linear cocycle over a given ``base'' dynamical system is a 
composition of \emph{time-dependent} $d \times d$-matrices driven by the dynamics on the base. More precisely,
in our setting we have the following definition.

\begin{definition}
	Let $\Tc$ be a continuous RDS as in Section \ref{subsubsec:basicSetupRDS}, referred to below as the \emph{base RDS},
	and let $(\tau^t)$ be its associated skew product as in Section \ref{subsubsec:skewProd}. 
	A $d$-dimensional linear cocycle $\Ac$ over the base RDS $\Tc$ is a
	mapping $\Ac : \Omega \to C_{u, b}([0,\infty) \times Z, M_{d \times d}(\R))$
	with the following properties:
		\begin{itemize}
			\item[(i)] The evaluation mapping $\Omega \times [0,\infty) \times Z \to M_{d \times d}(\R)$
			sending $(\omega, t, z) \mapsto \Ac^t_{\omega, z}$ is $\Fc \otimes \Bor([0,\infty)) \otimes \Bor(Z)
			-$measurable, and for fixed $\omega \in \Omega, t \geq 0$, the mapping $Z \to M_{d \times d}(\R)$, 
			$z \mapsto \Ac^t_{\omega, z}$, is continuous.
			\item[(ii)] The mapping $\Ac$ satisfies the \emph{cocycle property}: for any $z \in Z, \omega \in \Omega$
			we have $\Ac^0_{\omega, z} = \Id_{\R^d}$, the $d \times d$ identity matrix, and for $s, t \geq 0$ we have
			\begin{align}\label{eq:cocycleProperty}
			\Ac^{s + t}_{\omega, z} = \Ac^s_{\tau^t(\omega, z)} \circ \Ac^t_{\omega, z} \, .
			\end{align}
		\end{itemize}
\end{definition}

To motivate this definition, consider the following example.

\begin{example}\label{rmk:derivativeCocycle}
		Let $Z$ be a Riemannian manifold and assume that for each $\omega \in \Omega$,
		 $\Tc_\omega^t : Z \to Z$ is a $C^1$ mapping on $Z$ (e.g., the RDS defined in Example \ref{exam:SDERDS}).
	   The cocycle $\Ac_{\omega, z}^t := D_z \Tc^t_\omega, z \in X, t \geq 0$, is often referred to
	   as the \emph{derivative cocycle} for $\Tc$.
	The cocycle property \eqref{eq:cocycleProperty} is a manifestation
	of the Chain Rule from standard calculus
	and the cocycle property (ii) in Section \ref{subsec:RDSelements} for the RDS $\Tc$. For more information,
	see, e.g., \cite{arnold2013random, kifer2012ergodic}.
\end{example}

\subsubsection{The Multiplicative Ergodic Theorem (MET)}\label{subsubsec:METRDS}

It is of natural interest, in the setting described above, to study the
\emph{asymptotic exponential growth rate}
\begin{align}\label{eq:lyapExpDefinition}
\lim_{t \to \infty} \frac{1}{t} \log |\Ac^t_{\omega, z} v| \, ,
\end{align}
at $z \in Z, v \in \R^d$. When it exists, the quantity in \eqref{eq:lyapExpDefinition} is the \emph{Lyapunov exponent}
at $z$ in the direction $v$. For systems such as those in Example \ref{rmk:derivativeCocycle}, the existence and \emph{positivity} of the limit \eqref{eq:lyapExpDefinition}
implies that the orbit of $x$ is \emph{sensitive with respect to initial conditions}, a possible symptom 
of an asymptotically chaotic regime for $\Tc$.

However, there is a priori no guarantee that the limits \eqref{eq:lyapExpDefinition}
even exist in the first place. As it turns out, the most successful approach to the problem of the existence of 
the limits \eqref{eq:lyapExpDefinition} is
through ergodic theory: the limits \eqref{eq:lyapExpDefinition} exist for 
all $v \in \R^d$, $\P$-almost all $\omega \in \Omega$,
 and for points $z \in Z$ \emph{generic with respect to 
stationary measures} for the RDS $\Tc$, modulo 
a condition ensuring $|\Ac^t_{\omega, z}|$ does not get too large too fast as $t \to \infty$ for `most' $(\omega, z) \in \Omega \times Z$. 
This is the content of the MET,
which we will now state precisely.

\bigskip

Let $\mu$ be a stationary measure
for the RDS $\Tc$ satisfying the independent increments condition (H1). Let $\Ac$
be a linear cocycle as above. Throughout, we will assume the following
\emph{integrability condition} for the cocycle $\Ac$.

\begin{itemize}
	\item[(H2)] The triple $(\Tc, \Ac, \mu)$ has the property that $\Ac^t_{\omega, z}$ is an invertible matrix
	for all $t \in [0,\infty), \omega \in \Omega, z \in Z$, and\footnote{Here, $\log^+(a) := \max\{ 0 , \log a\}$ for $a > 0.$}
	\begin{align}\label{eq:integrability1}
	\E \int \bigg( \sup_{0 \leq t \leq 1} \log^+ |\Ac^t_{\omega, z}| \bigg) \, \dee \mu(z) \, , \quad
	\E \int \bigg( \sup_{0 \leq t \leq 1}  \log^+ | (\Ac^t_{\omega, z})^{-1}| \bigg) \, \dee \mu( z) < \infty \, .
	\end{align}
\end{itemize}
These conditions are standard for the derivative cocycles of stochastic flows generated by SDE; see, e.g., \cite{kifer1988}.

\begin{theorem}[Multiplicative Ergodic Theorem; Theorem 3.4.1 in \cite{arnold2013random}]\label{thm:MET}
Let $\Tc$ be a continuous RDS as in Section \ref{subsubsec:basicSetupRDS} satisfying condition (H1). Let $\mu$ be an \emph{ergodic}
stationary measure associated to $\Tc$ and assume that $\Ac$ is a linear cocycle over $\Tc$
for which the integrability condition $(H2)$ holds.

Then, there exist $r$ 
distinct deterministic real numbers 
\[
\lambda_1 > \cdots > \lambda_r \, , 
\]
$r \in \{ 1 , \cdots, d\}$, a $(\tau^t)$-invariant\footnote{That is, $T^t \Gamma \subset \Gamma$ for all $t \geq 0$.}  set $\Gamma \subset \Omega \times Z$ of full $\P \times \mu$-measure, 
and for each $(\omega, z) \in \Gamma$, a flag of subspaces 
\[
\R^d =: F_1 \supset F_2(\omega, z)  \supset \cdots \supset F_r(\omega, z) \supset F_{r + 1} := \{ 0 \} \, , 
\]
with $\dim F_i \equiv m_i$ for constants $m_i \in \{ 1, \cdots, d\}, 1 \leq i \leq r$,
for which the following holds. For any $1 \leq i \leq r$ and $v \in 
F_i(\omega, z) \setminus F_{i + 1}(\omega, z)$, we have
\begin{align}\label{eq:MET}
\lim_{t \to \infty} \frac1t \log | \Ac^t_{\omega, z} v | = \lambda_i \, .
\end{align}
Moreover, the assignment $(\omega, z) \mapsto F_i(\omega, z)$ varies measurably.
\end{theorem}
\noindent Note that automatically, for any $(\omega, z) \in \Gamma$ and $t > 0$ we have that
\[
\Ac^t_{\omega, z} F_i(\omega, z) = F_i(\tau^t(\omega, z))
\]
for each $i = 1,\cdots, d$. This is a straightforward consequence of \eqref{eq:MET} and is left to the reader.

\medskip

The MET as above is originally due to Oseledets \cite{oseledets1968multiplicative}; since then
many proofs of the MET have been recorded, each providing a different perspective on this seminal
result. One perspective useful to us in this study is that given by the proof-technique 
of Ragunathan \cite{raghunathan1979proof} and Ruelle \cite{ruelle1979ergodic, ruelle1982characteristic}.
For future use, we record the following intermediate step in this proof. 

Below, for a $d \times d$-matrix $A$ and for $1 \leq i \leq d$, we write $\sigma_i(A)$ for the $i$-th singular value of $A$.

\begin{lemma}\label{lem:singularValuesLimit}
Let $\lambda_i$ and $(\omega, z) \mapsto F_i(\omega, z), 1 \leq i \leq r$ be as in Theorem \ref{thm:MET}.
\begin{itemize}
\item[(i)] For any $1 \leq i \leq d$, the limits
\[
\chi_i = \lim_{t \to \infty} \frac{1}{t} \log \sigma_i(\Ac^t_{\omega, z})
\]
exist and are constant for $\P\times \mu$-almost every $(\omega, z) \in \Omega \times Z$.
Moreover, the Lyapunov exponents $\lambda_i, 1 \leq i \leq r$ are precisely the distinct values among the $\chi_i, 1 \leq i \leq d$.

\item[(ii)] For $\P \times \mu$-almost every $(\omega, z) \in \Omega \times Z$, the limit
\[
\Lambda_{\omega, z} := \lim_{t \to \infty} \frac{1}{t} \log \big( ( \Ac^t_{\omega, z})^{\top} \Ac^t_{\omega, z} \big) 
\]
exists. The matrix $\Lambda_{\omega, z}$ is symmetric with distinct eigenvalues
$\lambda_i, 1 \leq i \leq r$ and corresponding eigenspaces $E_1(\omega, z), \cdots, E_r(\omega, z)$. Moreover, for each $1 \leq i \leq r$ we
have
\[
F_i(\omega, z) = \bigoplus_{j = i}^r E_j(\omega, z) \, .
\]
\end{itemize}
\end{lemma}
Lemma \ref{lem:singularValuesLimit}(i) is often proved using the Kingman Subadditive Ergodic 
Theorem \cite{kingman1973subadditive}. Item (ii) follows from item (i) and a linear
algebra argument; see \cite{raghunathan1979proof, ruelle1979ergodic}
for more details.

Note that from Lemma \ref{lem:singularValuesLimit}(i), we have
 that $\lambda_1 = \lambda^+$ and $\lambda_r = \lambda^-$, where
\begin{align}\label{eq:defineLambdaPlusMinus}
\lambda^+ = \lim_{t \to \infty} \frac{1}{t} \log | \Ac^t_{\omega, z}| \, , 
\quad \lambda^- = \lim_{t \to \infty} - \frac{1}{t} \log |(\Ac^t_{\omega, z})^{-1}| \, ,
\end{align}
since for any invertible matrix $A \in M_{d \times d}(\R)$ we have $\sigma_1(A) = |A|$ and $\sigma_d(A) = |A^{-1}|^{-1}$.
In particular, $r > 1$ (i.e., there exist at least two distinct Lyapunov exponents) if and only if $\lambda^+ > \lambda^-$.
Of course, the problem of verifying that $\lambda^+ > \lambda^-$ for concrete systems is often
extremely challenging: this is precisely the subject of Sections \ref{subsec:FurstenbergFD} and \ref{subsec:furstInfiniteDimensions}.

For the remainder of Section \ref{subsec:LinearCocycles} we will 
continue our discussion of linear cocycles and the MET by introducing
several auxiliary processes associated to a linear cocycle $\Ac$, namely, the
\emph{projective process} (Section \ref{subsubsec:projectiveProcessRDS}) and \emph{matrix processes} (Section \ref{subsubsec:matrixProcess}),
as well as the \emph{$(-\top)$-cocycle} $\check \Ac$ associated to $\Ac$ (Section \ref{subsubsec:inverseTransCocycle}).

\subsubsection{Projective RDS associated to the cocycle $\mathcal A$}\label{subsubsec:projectiveProcessRDS}

Let us write $P^{d-1}$ for the projective space associated to $\R^d$. 
The action of an invertible matrix $A \in M_{d \times d}(\R)$ on $\R^d$ descends to a well-defined action 
 $A : P^{d-1} \to P^{d-1}$.
 
With this understanding, we can think of the cocycle $\mathcal A$ as giving rise to an RDS
on the product $Z \times P^{d-1}$, i.e., that given for $\omega \in \Omega$ by 
\[
(t, z, v) \mapsto (\Tc^t_\omega z, \Ac^t_{\omega, z} v) \, , \quad (z, v) \in Z \times P^{d-1} \, , t \in [0,\infty) \, .
\]
We refer to the RDS on $Z \times P^{d-1}$ as the \emph{projective RDS} or \emph{projective process}.
As one can easily check, this is a continuous RDS in the sense of Section \ref{subsubsec:basicSetupRDS} with $Z \times P^{d-1}$
replacing $Z$. Correspondingly we will assume in what follows that the following independent increments
condition, analogous to (H1), is satisfied:

\begin{itemize}
	\item[(H3)] \label{item:H3}For all $s, t > 0$, we have that the $C_{u, b}(Z, Z) \times C_{u, b}(Z, M_{d \times d}(\R))$-valued 
	random variables $(T^t_\cdot, \Ac^t_{\cdot, \cdot})$ and $(T^s_{\theta^t \cdot}, \Ac^s_{\theta^t \cdot, \cdot})$
	on $(\Omega, \Fc, \P)$ are independent.
\end{itemize}

Assumption (H3) ensures (Lemma \ref{lem:MarkovProperty}) that associated to the RDS on
 $Z \times P^{d-1}$ is a Markov process $(z_t, v_t)_{t \geq 0}$ on $Z \times P^{d-1}$
with transition kernel 
\[
\widehat P_t((z, v), K) = \P((z_t, v_t) \in K | (z_0, v_0) = (z, v)) = \P \{ (T_\omega^t z, \Ac^t_{\omega, z} v) \in K\}
\]
defined for $(z, v) \in Z \times P^{d-1}, K \subset Z \times P^{d-1}$ Borel. 
In addition, we can consider the associated skew product semiflow
$\hat \tau^t : \Omega \times Z \times P^{d-1} \to \Omega \times Z \times P^{d-1}$, $t \in [0,\infty)$,
as in Section \ref{subsubsec:skewProd}.

\bigskip

We now turn our attention to the relationship between the ergodic theory of the projective process
and the MET.
It is not hard to see that any stationary measure $\nu$ for $(\widehat P_t)$
must project to some $(P_t)$-stationary measure $\mu$ on the $Z$-factor. Conversely, by Lemma \ref{lem:stationaryMeasureExists} we have the following.

\begin{lemma}\label{lem:existenceProjectiveMeasures}
Given a stationary measure $\mu$ for $(P_t)$, there exists at least one stationary measure
$\nu$ for the projective semigroup $(\widehat P_t)$ such that $\nu(A\times Z) = \mu(A)$.
\end{lemma}

If $\nu$ as above is the \emph{unique} stationary measure with marginal $\mu$, then we obtain the following refinement of the MET.

\begin{proposition}\label{prop:refineMET}
Assume that there is only one stationary measure $\nu$ for the projective RDS projecting
to $\mu$ on the $Z$-factor. Then, we have the following: for $\mu$-almost every $z \in Z$
and any $v \in \R^d \setminus \{ 0 \}$, we have
\[
\lim_{t \to \infty} \frac{1}{t} \log | \Ac^t_{\omega, z} v |  = \lambda_1
\]
with $\P$-probability 1.
\end{proposition}

Proposition \ref{prop:refineMET} is actually a corollary of the more general
 Random Multiplicative Ergodic Theorem, discovered independently by Kifer (Theorem III.1.2 in \cite{kifer2012ergodic}) and Carverhill \cite{carverhill1985flows},
  describing the situation when several stationary measures $\nu$ project to a single stationary $\mu$.
Since we do not use this more general formulation here, we omit it and refer the interested reader
to the references above for more information.

\subsubsection{Matrix RDS associated to the cocycle $\mathcal A$}\label{subsubsec:matrixProcess}

The cocycle $\mathcal A$ also gives rise to an RDS on the product space $Z \times M_{d \times d}(\R)$; 
for $\omega \in \Omega$, the time-$t$ mapping applied to $(z, A) \in Z \times M_{d \times d}(\R)$ 
is given by
\[
(z, B) \mapsto (\Tc^n_\omega z, \Ac^n_{\omega, z}  B) \, .
\]
Like before, this RDS on $Z \times M_{d \times d}(\R)$ falls into the framework given in Section \ref{subsubsec:basicSetupRDS}
 with $Z \times M_{d \times d}(\R)$ replacing $Z$.

Similarly, under the independent increments hypothesis
 (H3) we can associate to this RDS a Markov process $(z_t, A_t)$ on $Z \times M_{d \times d}(\R)$
with transition kernel $Q_t((z, A), K)$.
Note that if $A \in M_{d \times d}(\R)$ is invertible
and $K = K_1 \times K_2$ where $K_1 \subset Z, K_2 \subset M_{d \times d}(\R)$, then
\[
Q_t((z, A), K) = Q_t((z, \Id), K_1 \times (K_2 A^{-1}) ) \, ,
\]
where $\Id = \Id_{\R^d}$. Thus, frequently we are only interested in the Markov kernel $(Q_t)$ evaluated at $(z, \Id)$.

\subsubsection{The MET for the $(-\top)$-cocycle $\check \Ac$}\label{subsubsec:inverseTransCocycle}

In this paper we will also need to consider what we call the \emph{$(- \top)$-cocycle} $\check \Ac$, defined for
$z \in Z, \omega \in \Omega, t \geq 0$ by
\[
\check \Ac^t_{\omega, z} = (\Ac^t_{\omega, z})^{-\top} \, .
\]
Here, ``$(-\top)$'' refers to the inverse-transpose of a $(d \times d)$-matrix. As one can easily check, 
$\check \Ac$ is a linear cocycle over the RDS $\mathcal T$; when
(H2) and (H3) for the original cocycle $\Ac$ are assumed, the same hold for the $(-\top)$-cocycle $\check \Ac$. 
Therefore the MET (Theorem \ref{thm:MET}) and all the aforementioned material 
applies, yielding Lyapunov exponents $\check \lambda_1 > \cdots > \check \lambda_{\check r}$ 
and associated subspaces $\check F_2(\omega, z), \cdots, \check F_{\check r}(\omega, z)$. 

These objects can be directly represented in terms of the exponents and subspaces
of the original cocycle $\Ac$. 

\begin{proposition}[Theorem 5.1.1 in \cite{arnold2013random}]\label{prop:METforCheckAc}
	We have that $\check r = r$, and for each $1 \leq i \leq r$, we have
	\begin{gather}
		\check \lambda_i = - \lambda_{r - (i - 1)} \, , \quad \text{ and} \\
\label{eq:subspacesCorrespondenceRDS}		\check F_i(\omega, z) = \big( F_{r - (i-1) + 1}(\omega, z) \big)^{\perp} \quad \text{ for almost all } (\omega, z) \in \Omega \times Z \, .
	\end{gather}
\end{proposition}
\begin{proof}
This follows on applying Lemma \ref{lem:singularValuesLimit} to the cocycle $\check \Ac$ and noting that
\[
\log \big( ( \check \Ac^t_{\omega, z})^\top \check \Ac^t_{\omega, z}\big)  = - \log \big( ( \Ac^t_{\omega, z})^{\top} \Ac^t_{\omega, z} \big) 
\]
holds for all $(\omega, z) \in \Omega \times Z$ and $t \geq 0$. 
\end{proof}

Under assumption (H3), the cocycle $\check \Ac$ induces the \emph{$(-\top)$-projective process} $(z_t, \check v_t)$ on $Z \times P^{d-1}$ defined
for fixed initial $z_0 \in Z, \check v_0 \in P^{d-1}$ by setting $\check v_t$ to be the projective representative of $\check \Ac^t_{\omega, z_0} \check v_0$.
Then, all the material from Section \ref{subsubsec:projectiveProcessRDS} applies with $\check \Ac$ replacing $\Ac$ and $(z_t, \check v_t)$ replacing
$(z_t, v_t)$. 

In particular, the conclusions of Proposition \ref{prop:refineMET} hold with $\check \Ac$ replacing $\Ac$ when the stationary measure
for $(z_t, \check v_t)$ projecting to $\mu$ on the $Z$ factor is unique. 

\subsection{The MET in the random setting: Furstenberg's criterion}\label{subsec:FurstenbergFD}

Furstenberg's criterion was originally discovered by Furstenberg in his seminal 1968 paper,
\emph{Noncommuting Random Products} \cite{furstenberg1963noncommuting}.
 It has since been refined and extended over the
subsequent years by a variety of authors; see Section \ref{subsubsec:FurstenbergOutline} for some
citations.

In Section \ref{subsubsec:introduceFurstenberg} we will state Furstenberg's criterion precisely
in the setup of Sections \ref{subsec:RDSelements} and \ref{subsec:LinearCocycles}.
In Section \ref{subsubsec:nondegenCondLaws} we provide a condition
for checking Furstenberg's criterion which is most useful when $\Tc$ and $\Ac$ are generated by
finite-dimensional SDE.
In Section \ref{subsec:furstInfiniteDimensions} we will consider conditions for checking Furstenberg's criterion which
are amenable to the situation when the phase space for $\Tc$ is more general and, possibly, infinite-dimensional.

\medskip

{\it For the remainder of Section \ref{sec:RDS} we assume the setting of Sections \ref{subsec:RDSelements} and \ref{subsec:LinearCocycles}. 
Specifically, $\Tc$ is a continuous RDS
on the metric space $Z$ as in Section \ref{subsubsec:basicSetupRDS} satisfying (H1) and admitting an ergodic
stationary measure $\mu$, while the cocycle $\Ac$ over $\Tc$ satisfies the conditions of Section \ref{subsubsec:linearCocycleBasicRDS}
as well as the integrability condition (H2) and the independent increments condition (H3).}

\subsubsection{Furstenberg's criterion in the RDS setting}\label{subsubsec:introduceFurstenberg}

Furstenberg's criterion revolves around a central theme:
if $\lambda^+ = \lambda^-$ as above, then there is a \emph{deterministic}, i.e., 
$\omega$-independent, structure preserved by the cocycle $\Ac$
with probability one. 

Let us make this more precise. 

\begin{theorem}\label{thm:furstenberg}
	If $\lambda^+ = \lambda^-$, then for each $z \in Z$ there a Borel measure $\nu_z$ on $P^{d-1}$ such that
	(i) the assignment $z \mapsto \nu_z$ is measurable\footnote{To wit,
	 for any Borel $K \subset P^{d-1}$, the function $z \mapsto \nu_z(K)$ is Borel measurable. Equivalently,
	 $z \mapsto \nu_z$ is Borel measurable in the weak$^*$ topology on finite Borel measures on $P^{d-1}$.}
	 and (ii) for each $t \in [0,\infty)$ and $(\P \times \mu)$-almost all $(\omega, z) \in \Omega \times Z$ (perhaps depending on $t$), we have that $T_\omega^t z \in Z$ and
	\begin{align}\label{eq:invariantMeasuresFamily}
	(\Ac^t_{\omega, z})_* \nu_z = \nu_{\Tc^t_\omega z} \, .
	\end{align}
\end{theorem}
\noindent Theorem \ref{thm:furstenberg} as above is a consequence of Proposition 2 and Theorem 3 in \cite{ledrappier1986positivity}. Deducing
the version given above requires passing from the discrete-time setting of \cite{ledrappier1986positivity} to our present
continuous-time setting, and is the reason why the $(\P \times \mu)$-almost sure set may depend on $t$.
Further details are left to the reader.

\smallskip

Note that automatically, if $\lambda^+ = \lambda^-$, then the measure $\nu$ on $Z \times P^{d-1}$ defined by
\[
\dee \nu(z, v) = \dee \mu(z) \dee \nu_z(v) \, , \quad (z, v) \in Z \times P^{d-1} \, , 
\]
is a stationary measure for the Markov semigroup $(\widehat P^t)$ associated to the projective RDS on $Z \times P^{d-1}$.

We conclude that $\lambda^+ > \lambda^-$ if, from 
the conclusions of Theorem \ref{thm:furstenberg}, we can derive a contradiction.
Our goal in the remainder of Section \ref{sec:RDS} is to identify criteria for the cocycle $\mathcal A$
under which a contradiction can be derived. 

Before continuing, let us establish some useful vocabulary. Any measurable family $(\nu_z)$ 
of probability measures on $P^{d-1}$ will be referred to as a \emph{family of fiber measures}, 
while for $z \in Z$ the individual measure $\nu_z$ will be called the \emph{fiber measure at $z$}.
If the family of fiber measures $(\nu_z)$ satisfies \eqref{eq:invariantMeasuresFamily} for 
all $t \geq 0$ and $\P \times \mu$-almost every $(\omega, z) \in \Omega \times Z$ (the almost-sure 
set perhaps depending on $t$), we call $(\nu_z)$ 
an \emph{invariant fiber measure family}.

\subsubsection{Nondegeneracy of conditional laws}\label{subsubsec:nondegenCondLaws}

For simplicity, and because our primary application in this paper falls in this special case, 
let us restrict our attention to the case when $\mathcal A$ is an $SL_d(\R)$ cocycle. That is,
$\det \Ac^t_{\omega, z} \equiv 1$ for all $t \geq 0, z \in Z, \omega \in \Omega$. 

Our starting point is the following observation.
\begin{lemma}\label{lem:emptyInterior}
	Let $\nu, \nu'$ be Borel probability measures on $P^{d-1}$. Then, the set
	\[
	\{ A \in SL_d(\R) : A_* \nu = \nu' \} \subset SL_d(\R) \, ,
	\]
	has empty interior.
\end{lemma}
\noindent The proof is straightforward and is omitted.
\medskip

In relation to the condition \eqref{eq:invariantMeasuresFamily}, 
Lemma \ref{lem:emptyInterior} says that if for some $t_0 > 0$ we can somehow fix both
$z$ and the image $z' = \Tc^{t_0}_\omega z$, then the set of matrices mapping
the measure $\nu = \nu_z$ to $\nu' = \nu_{z'}$ is `small' in the topological sense.

We can make sense of this using regular conditional probabilities. 
Let us consider the measure $Q_{t_0}((z, \Id), \cdot)$ on $Z \times M_{d \times d}(\R)$
and disintegrate it according to the value $z_{t_0}$ attained by the $(z_t)$ process, 
conditioned on $z_0 = z$. To wit, fix $t_0 > 0$; 
for $P_{t_0}(z, \cdot)$-generic $z' \in Z$, we intend to define the \emph{regular conditional probability}
\[
Q^{t_0}_{z, z'}(K) := \P( \Ac^{t_0}_{\omega, z} \in K | \Tc^{t_0}_\omega z = z') \, , \quad K \in \Bor(SL_d(\R)) \, .
\]
This is justified rigorously below.
\begin{lemma}[\cite{chang1997conditioning}]\label{lem:conditioningRDS}
Assume $\Ac$ is an $SL_d(\R)$ cocycle and that $\Omega$ is a Borel subset of a Polish space equipped with the $\sigma$-algebra $\Fc$ of Borel subsets of $\Omega$. Fix $z \in Z$. Then, there is a mapping $Z \times \Bor(SL_d(\R)) \mapsto [0,1]$, 
$(z', K) \mapsto Q^{t_0}_{z, z'}(K)$, with the following properties.
\begin{itemize}
	\item[(1)] For each $K \in \Bor(SL_d(\R))$, the mapping $z' \mapsto Q^{t_0}_{z, z'}(K)$ is Borel measurable.
	\item[(2)] For $P_{t_0}(z, \cdot)$-almost all $z' \in Z$, the set function $Q^{t_0}_{z, z'}(\cdot) : \Bor(SL_d(\R)) \to [0,1]$
	is a Borel probability measure on $SL_d(\R)$.
	\item[(3)] For any bounded measurable function $h : Z \times SL_d(\R) \to \R$, we have that 
	\[
	\int  h(z', A') \,Q^{t_0}_{z, z'}(\dee A')\,  P_{t_0}(z, \dz') = \int h(z', A') \, Q_{t_0}((z, \Id), \dee (z', A')) \, .
	\]
\end{itemize}
\end{lemma}

\begin{definition}\label{defn:propertyCRDS}
Let $\mathcal A$ be an $SL_d(\R)$-cocycle and assume $(\Omega, \Fc)$ is as in Lemma \ref{lem:conditioningRDS}.
We say that $\mathcal A$ satisfies condition (C) 
if there is a $t_0 > 0$ and a set $S \subset Z$ of 
positive $\mu$-measure with the following property: for each $z \in S$, there is a $P_{t_0}(z, \cdot)$-positive
measure set $S_z \subset  Z$ such that $Q^{t_0}_{z, z'}(\cdot) $ is defined and is absolutely continuous 
with respect to Lebesgue measure on $SL_d(\R)$.

\end{definition}
Note that if (C) holds and $z \in S, z' \in S_{z}$, 
then the support of $Q^{t_0}_{z, z'}(\cdot)$ has nonempty interior in $SL_d(\R)$.
Therefore by Theorem \ref{thm:furstenberg} and Lemma \ref{lem:emptyInterior} we conclude the following.

\begin{corollary}
	If the $SL_d(\R)$-cocycle $\mathcal A$ satisfies condition (C), then $\lambda^+ > \lambda^-$. In particular, 
	and $\lambda_1 > 0$.
\end{corollary}
\begin{proof}
	By Theorem \ref{thm:furstenberg}, $\lambda^+ > \lambda^-$. Since $\Ac$ is an $SL_d(\R)$ cocycle,
	it follows from basic linear algebra that $1 = \det (\Ac^t_{\omega, z}) = \prod_{i = 1}^d \sigma_i(\Ac^t_{\omega, z})$
	for all $z \in Z, \omega \in \Omega, t \geq 0$. Thus from Lemma \ref{lem:singularValuesLimit} we have
	 that $\sum_{i = 1}^d \chi_i = 0$, $(\chi_i)$ as in Lemma \ref{lem:singularValuesLimit}(i). Since $\lambda_1 = \lambda^+ = \chi_1 , \lambda_r = \lambda^- = \chi_d$,
	 we conclude from $\lambda^+ >\lambda^-$ that $\chi_1 > 0$ and $\chi_d < 0$.
\end{proof}

Condition (C) holds for a large class of systems for which the process 
$(z_t, A_t)$ is governed by a finite-dimensional SDE on $Z \times SL_d(\R)$;
see Section \ref{subsubsec:ErgodicProperties}. We note that condition (C) is a 
straightforward adaptation of a condition given in \cite{carverhill1985flows}
for the Lyapunov exponent of a divergenceless SDE to have a positive Lyapunov exponent.

\section{Positive Lyapunov exponents for cocycles over infinite-dimensional RDS} \label{subsec:furstInfiniteDimensions}

For stochastic processes on infinite-dimensional spaces there is no corresponding analogue of Hormander's
Theorem. As a result it is frequently quite difficult in applications to verify the condition (C) (Definition \ref{defn:propertyCRDS}). 

Thankfully, condition (C) is far from necessary to rule out the criterion in Theorem \ref{thm:furstenberg}. In this section
we prove a sufficient condition, weaker than (C), which is better suited for infinite dimensional RDS. To the best of our knowledge, this result appears to be new. The proof is carried out in several steps:

First, in Section \ref{subsubsec:measurableToTopRDS} we will establish the \emph{continuous dependence} of an invariant fiber measure
family $(\nu_z)$ on the base point $z$ under the assumption that the Markov semigroup $P_t$ associated
to the RDS $\Tc$ has the strong Feller property (Definition \ref{defn:strongFeller} below). 
Leveraging this continuity result, in Section \ref{subsubsec:classifyRDS} we will take advantage of algebraic properties of $SL_d(\R)$ 
 to obtain a classification (Theorem \ref{thm:classification}) for the family $(\nu_z)$ under the assumption that $\lambda^+ = \lambda^-$ as
 in Furstenberg's criterion (Theorem \ref{thm:furstenberg}).
 Finally, in Section \ref{subsubsec:approxControlRuleOut} we 
will state a weakening (C') (Definition \ref{defn:condCprimeRDS}) ruling out each alternative in the classification we obtain.

\medskip

For the entirety of Section \ref{subsec:furstInfiniteDimensions}, we assume the setting given at the beginning of Section \ref{subsec:FurstenbergFD}.

\subsection{From measurable to topological}\label{subsubsec:measurableToTopRDS}

The goal of Section \ref{subsubsec:measurableToTopRDS} is to turn the measurable information
contained in Theorem \ref{thm:furstenberg}, namely, that the
invariant measure family $(\nu_z)$ satisfies \eqref{eq:invariantMeasuresFamily} 
for $(\P \times \mu)$-almost all $(\omega, z)$, into topological information concerning ``all'' $\omega$, in a suitable sense,
and all $z$ in a closed set.
This will be accomplished in two phases: First, the family $(\nu_z)$ will be replaced
with a \emph{$\mu$-almost sure version} $(\bar \nu_z)_{z \in \supp \mu}$ which is weak$^*$ continuous
as $z$ varies in $\supp \mu$ (Proposition \ref{prop:contVaryingMeasure}). 
Second, the $\P \times \mu$-almost sure relation \eqref{eq:invariantMeasuresFamily}
for the family $(\nu_z)$ will be turned into a corresponding relation among the family $(\bar \nu_z)$
for all $z \in \supp \mu$ and ``$\P$-almost-all'' replaced by ``all'', in a sense to be made precise (Lemma \ref{lem:almostSurelyToTopological}).

\smallskip

The material in Section \ref{subsubsec:measurableToTopRDS} is analogous to Proposition 6.3 and Lemma 6.5 of \cite{baxendale1989lyapunov}.
For a summary of the differences between the latter and our results in this setting, see Remark \ref{rmk:whyNovelRDS} below.

\bigskip

Going forward, we will require an additional regularity assumption on the Markov semigroup $(P_t)$
associated to the RDS $\Tc$, which we now spell out here.
\begin{definition}\label{defn:strongFeller}
	We say that the Markov semigroup $(P_t)$ has the \emph{strong Feller property}
	if for all bounded, measurable $h : Z \to \R$, and for all $t > 0$, the 
	function $P_t h : Z \to \R$ is bounded and continuous. 
\end{definition}

At times it will be helpful to use the following well-known result regarding strong Feller semigroups.

\begin{lemma}\label{lem:ultraFeller}
Assume $Z$ is a Polish space. 
\begin{itemize}
\item[(a)] If the Markov semigroup $(P_t)$ on $Z$ has the strong Feller property, then 
it is automatically \emph{ultra Feller}, i.e., for all $t > 0$ the mapping $z \mapsto P_t(z, \cdot)$ is 
continuous in the total variation distance\footnote{Given two finite signed measures $\eta_1, \eta_2$ on the same measurable space $(X, \mathfrak F)$, the total variation distance is defined by $\| \eta_1 - \eta_2\|_{tv} = \sup_{K \in \mathfrak F} | \eta_1(K) - \eta_2(K)|$.} $\| \cdot \|_{tv}$ on the space of finite signed measures
on $Z$.
\item[(b)] Let $\mu$ be a stationary measure for $(P_t)$ and let $K \subset Z$ be a Borel set of full $\mu$ measure.
Then, $P_t(z, K) = 1$ for all $t > 0$ and $z \in \supp \mu$.
\end{itemize}
\end{lemma}
\begin{proof}
	Item (a) is proved in \cite{seidler2001note}. For (b), one checks that for all $t \geq 0$, the set
	$\{ z \in Z : P_t(z, K) = 1\}$ is dense in $\supp \mu$. Item (b) now follows from continuity in total variation
	as in (a).
\end{proof}

With these preparations out of the way, 
we can now state precisely the first step in our program,
a continuity result for the invariant measure family $(\nu_z)_{z \in Z}$.

\begin{proposition}\label{prop:contVaryingMeasure}
	Assume $(P_t)$ is strong Feller, and let $(\nu_z)$ be
	an invariant fiber measure family on $Z$ as in Section \ref{subsubsec:introduceFurstenberg}.
	Then, there exists an invariant fiber measure family $(\bar \nu_z)$, defined for $z \in \supp \mu \subset Z$,
	with the following properties.
	\begin{itemize}
		\item[(a)] The family $(\bar \nu_z)$ is a \emph{$\mu$-almost sure version} of the original family $(\nu_z)$, i.e.,
		for $\mu$-almost every $z \in \supp \mu$, we have $\nu_z = \bar \nu_z$.
		\item[(b)] The family $(\bar \nu_z)$ is continuously varying in the weak$^*$ topology 
			on $P^{d-1}$. 
	\end{itemize}
\end{proposition}

That is, by Proposition \ref{prop:contVaryingMeasure} we can replace the possibly discontinuous
invariant measure family $(\nu_z)$ with a continuously-varying invariant measure family $(\bar \nu_z)$
defined at each $z \in \supp \mu$, at the expense of modifying $(\nu_z)$ on a set of $\mu$-measure zero.
So as not to interrupt the flow of ideas, Proposition \ref{prop:contVaryingMeasure} is proved at the end of Section \ref{subsubsec:measurableToTopRDS}.

\bigskip

Let us now describe the second step in our program, namely, turning
the $\P \times \mu$-almost sure relation \eqref{eq:invariantMeasuresFamily} 
into an analogous relation holding ``surely''-- roughly speaking, holding for
all $(\omega, z) \in \Omega \times \supp \mu$ and for all $t \geq 0$, in a sense we make precise below.

To begin, some notation: let us write 
\[
\Cc =  C_{u, b}(Z, Z) \times C_{u, b}(Z, M_{d \times d}(\R)) 
\]
equipped with the product topology. Elements of $\Cc$ are written $(T, A)$ where
$T : Z \to Z, z \mapsto T z \in Z$ and $A : Z \to M_{d \times d}(\R), z \mapsto A_z \in M_{d \times d}(\R)$.
Given $t \geq 0$, let us write $\Sc_t$ for the topological support of the $\Cc$-valued
random variable $(\Tc^t_\omega, \Ac^t_{\omega})$ where $\omega$ is distributed as $\P$.
We set $\Sc = \overline{\cup_{t \geq 0} \Sc_t}$ for the closure of the union of the $\Sc_t$ in $\Cc$.

\begin{lemma}\label{lem:almostSurelyToTopological}
	Assume the setting, notation and conclusions of Proposition \ref{prop:contVaryingMeasure}. 
	Then, for any $z \in \supp \mu$ and $(T, A) \in \Sc$, we have that $T z \in \supp \mu$, and
	\begin{align}\label{eq:surelyInvariantMeasureFamily}
	A_z \bar \nu_z = \bar \nu_{T z} \, .
	\end{align}
\end{lemma}

The relation \eqref{eq:surelyInvariantMeasureFamily} for all $(T, A) \in \Sc$
is analogous to the `measure-theoretical' relation \eqref{eq:invariantMeasuresFamily}; in contrast to the latter,
\eqref{eq:surelyInvariantMeasureFamily} holds identically for all $(T, A)$ in the \emph{closed} subset $\Sc\subset \Cc$.
For this reason we regard \eqref{eq:surelyInvariantMeasureFamily}
as a ``topological'' statement, as opposed to a measure-theoretic one.

\bigskip

We now turn to the proofs of Proposition \ref{prop:contVaryingMeasure} and Lemma \ref{lem:almostSurelyToTopological}.

\subsubsection*{Proof of Proposition \ref{prop:contVaryingMeasure}}
Fix a continuous function $g : P^{d-1} \to \R$. Define $G : Z \to \R$ by $G(z) = \int g(v) \dee \nu_z(v)$. 
We begin by making the following Claim.

\begin{claim}\label{cla:fullMeasCauchy}
	There is a full $\mu$-measure subset $\tilde Z \subset \supp \mu$ with the following property. Let $G : Z \to \R$ be as above. 
	Then, $G|_{\tilde Z}$ has the property that for any Cauchy sequence $\{ z^m\}_{m \geq 1} \subset \tilde Z$,
	we have that the sequence $\{ G(z^m)\}_{m \geq 1}$ is Cauchy.
\end{claim}

Assuming the Claim, let us define the family $(\bar \nu_z)$. To start, for $z \in \tilde Z$
we set $\bar \nu_z := \nu_z$. Note that this ensures $(\bar \nu_z)$ is a version of $(\nu_z)$ 
as in item (a) above. 

Next, for $z \in \supp \mu \setminus \tilde Z$, we define $\bar \nu_z$ as follows.
Since $\tilde Z$ is dense in $\supp \mu$, we can find a sequence $\{ z^m\}_{m \geq 1} \subset \tilde Z$
converging to $z$. We now define $\bar \nu_z$ to be any weak$^*$ limit of the $\bar \nu_{z^m}$
(at least one exists by Prokhorov's Theorem since $P^{d-1}$ is compact \cite{billingsley2013convergence}).

Indeed, the weak$^*$ limit $\lim_{m \to \infty} \bar \nu_{z^m}$ actually exists: to see this,
fix any $g : P^{d-1} \to \R$ continuous and observe that the sequence 
$\{ G(z^m) = \int g(v) \dee \bar \nu_{z^m}(v)\}_{m \geq 1}$ is Cauchy by the Claim; this implies weak$^*$ convergence. Moreover
this same argument implies that the definition of $\bar \nu_z, z \in \supp \mu \setminus \tilde Z$
is independent of the approximating sequence $\{ z^m\}_{m \geq 1} \subset \tilde Z$.

This completes the definition of the family $(\bar \nu_z)$. By construction, $(\bar \nu_z)$ is a $\mu$-almost-sure
version of $(\nu_z)$, and so item (a) in Proposition \ref{prop:contVaryingMeasure} is satisfied.
 To show continuity as in item (b),
fix a continuous $g : P^{d-1} \to \R$; we will check that $\bar G(z) := \int g(u) \dee \bar \nu_z(u)$ is a 
continuous real-valued function.
For this, fix $z \in \supp \mu$ and let $\{ z^m\}_{m \geq 1} \subset \supp \mu$ be a sequence converging to $z$.
For each $m$, fix $\check z^m \in \tilde Z$ such that $d(\check z^m, z^m) < 1/m$ and $|G(\check z^m) - \bar G(z)| < 1/m$.
Then, 
\[
|\bar G(z^m) - \bar G(z)| \leq |\bar G(z^m) - G(\check z^m)| + |G(\check z^m) - \bar G(z)| \leq \frac{1}{m} + |G(\check z^m) - \bar G(z)| \, .
\]
The Claim and our definition of $\bar \nu_z$ imply that the second RHS term goes to zero.
This completes the proof
of continuity as in item (b).
It remains to prove the Claim.

\begin{proof}[Proof of Claim \ref{cla:fullMeasCauchy}]
	It is straightforward to construct a full $\mu$-measure subset $\tilde Z \subset \supp \mu$
	with the property that for all $z \in \tilde Z$ and rational $t$, we have with probability 1
	that $\Tc^t_\omega z \in \tilde Z$
	and that \eqref{eq:invariantMeasuresFamily} holds.
	For such $z \in \tilde Z$, on integrating the left and right-hand sides of \eqref{eq:invariantMeasuresFamily} 
	with respect to $d\P(\omega)$, we obtain that
	\[
	\int (\widehat P_{t} g) (z, v) \, \dee \bar \nu_z(v) = P_{t} G(z) \, ,
	\]
	where $\widehat P_t$ denotes the Markov semigroup associated to the projective process as 
	defined in Section \ref{subsubsec:projectiveProcessRDS}.
	
	Now, fix a Cauchy sequence $\{ z^m\}_{m \geq 1} \subset \tilde Z$ converging to some $z \in Z$.
	Fix $\epsilon > 0$ and fix a neighborhood $U$ of $z$; without loss, $\{ z^m\}_{m \geq 1} \subset U$. 
	Since $\widehat P_t g \to g$ uniformly on bounded
	subsets of $Z \times P^{d-1}$ (Proposition \ref{prop:semigroupRegularity1}(b)), 
	we have that $\widehat P_t g \to g$ uniformly
	on $U \times P^{d-1}$. Fix $t = t_\epsilon$ for which $|\widehat P_s g - g| < \epsilon$ on all of $U \times P^{d-1}$
	for all $s \in [0,t_\epsilon]$.
	
	Fix a rational $t_* \in [0,t_\epsilon]$. Given $m, m' \geq 1$ we estimate
	\begin{align}
	|G(z^m) - G(z^{m'})| &= \bigg| \int g(v) \dee \nu_{z^m}(v) - \int g(v) \dee \nu_{z^{m'}}(v) \bigg| \\
	& \leq \int |g(u) - P_{t_*} g(v)| \dee \nu_{z^m}(v) + \bigg| \int P_{t_*} g(v) \dee \nu_{z^m}(v) - \int P_{t_*} g(v) \dee \nu_{z^{m'}}(v) \bigg| \\
	& \hspace{.5in}+ \int |g(v) - P_{t_*} g(v)| \dee \nu_{z^{m'}}(v) \\
	& \leq 2 \epsilon + | P_{t_*} G({z^m}) - P_{t_*} G(z^{m'})| \, .
	\end{align}
	
	Now, $P_{t_*} G$ is a continuous function by the strong Feller property, and so $\{ P_{t_*} G(z^m)\}_{m \geq 1}$ is 
	a Cauchy sequence. The Cauchy property for $\{ G(z^m)\}_{m \geq 1}$ now follows.
\end{proof}

\subsubsection*{Proof of Lemma \ref{lem:almostSurelyToTopological}}
We begin by verifying that $T z \in \supp \mu$ for any $z \in \supp \mu$ and $(T, A) \in \Sc$.
To start, observe that since $\supp \mu$ has full $\mu$-measure, we have from
stationarity that $P_t(z, \supp \mu) = 1$ for all $t > 0$ and for $\mu$-almost all $z \in Z$.
As one can easily check, for continuous RDS $\Tc$ as in Section \ref{subsubsec:basicSetupRDS} satisfying (H1), the mapping
 $z \mapsto P_t(z, \cdot)$ is weak$^*$ continuous (irrespective of the strong Feller property).
 Thus, by the Portmanteau Theorem and the density of $\mu$-almost sure sets in $\supp \mu$,
 we conclude that $P_t(z, \supp \mu) = 1$ for all $z \in \supp \mu$. 
 
So, for any fixed $z \in Z$, we have for all $t \geq 0$ that $\Tc_\omega^t z \in \supp \mu$
with probability 1. In particular, any $(T, A) \in \Sc_t$ is the limit (in the topology on $\Cc$)
of elements $(T^m, A^m) \in \Sc_t$ for which $T^m z \in \supp \mu$ for all $m$. 
Therefore $T z \in \supp \mu$ holds by the closedness of $\supp \mu$ for any $(T, A) \in \Sc_t$.
A similar argument implies $T z \in \supp \mu$ for any $(T, A) \in \Sc$.

\medskip

Let us now move on to verifying the relation \eqref{eq:surelyInvariantMeasureFamily}.
For $z \in \supp \mu$, we define
\[
G_z = \{ (T, A) \in \Cc : (A_z)_* \bar \nu_z = \bar \nu_{T z} \} \, .
\]
Note that by the argument in the previous two paragraphs, 
$\bar \nu_{T z}$ is defined for all $z \in \supp \mu$ and $(T, A) \in \Sc$.
To complete the proof of Lemma \ref{lem:almostSurelyToTopological} 
it will suffice to show that $G_z \supset \Sc$ for all $z \in \supp \mu$.

To start, one checks that $G_z$ is closed in $\Cc$ by the closedness of $\supp \mu$ and the fact
that $z \mapsto \bar \nu_z$ is weak$^*$ continuous.
Next, let $\tilde Z$ be as constructed in the proof of Claim \ref{cla:fullMeasCauchy}.
It follows that for $z \in \tilde Z$ and all rational $t$ that
\[
\P( (\Tc^t_\omega, \Ac^t_{\omega}) \in G_z ) = 1 \, .
\]
So, for all rational $t \geq 0$ we deduce that $G_z$ is dense in $\Sc_t$, hence $G_z \supset \Sc_t$
since $G_z, \Sc_t$ are closed in $\Cc$. Moreover, for irrational $t \geq 0$, each $(T, A) \in \Sc_t$
is a limit of elements $(T^n, A^n) \in \Sc_{t_{n}}$ in $\Cc$, 
where $\{ t_n\}$ is a sequence of rationals for which $t_n \to t$ as $n \to \infty$. Again by closedness of $G_z$
we deduce that $G_z \supset \Sc_t$ for all $t \geq 0$. We conclude $G_z \supset \Sc$
for all $z \in \tilde Z$.

To conclude for $z \in \supp \mu \setminus \tilde Z$: let $z^m \to z$ be a convergent sequence,
$z^m \in \tilde Z$, and fix $(T, A) \in \Sc$. That $(T, A) \in G_z$ now follows from the 
fact that $(T, A) \in G_{z^m}$ for all $m$ from above and from the continuity of $z \mapsto \bar \nu_z$. 
This completes the proof of Lemma \ref{lem:almostSurelyToTopological}.

\subsection{A refinement of Furstenberg's criterion}\label{subsubsec:classifyRDS}

The refinement of Furstenberg's criterion we present here is effectively a \emph{classification} of the 
the fiber measures $\nu_z, z \in \supp \mu$ comprising a family satisfying 
the `topological' relation \eqref{eq:surelyInvariantMeasureFamily}. For the sake of brevity, and because it serves
our purposes in this paper, we prove this classification 
when $d$, the dimension of the cocycle $\mathcal A$, is less than or equal to 3, although it is likely to hold in higher dimensions (see Remark \ref{rmk:why3D}).

This classification, Theorem \ref{thm:classification} below,
is the analogue in our setting of Theorem 6.8 of \cite{baxendale1989lyapunov}. 
Our situation is significantly more general and entails several subtleties unique to our setting;
see Remarks \ref{rmk:badFiberMeasures}, \ref{rmk:whyNovelRDS} for more discussion.

\bigskip

The germ of this idea comes from 
the geometry of $SL_d(\R)$ and the restrictions placed on the subgroup of matrices
preserving a single projective measure. To wit, we have the following (for any dimension $d \geq 1$):

\begin{lemma}\label{lem:subgroups}
Let $d \geq 1$. Let $\eta$ be a Borel measure on $P^{d-1}$ and define $H = H_\eta \subset SL_d(\R)$ to be the subgroup
of matrices $A \in SL_d(\R)$ for which $A_* \eta = \eta$. Then, $H$ is closed, and moreover we
have the following dichotomy:
\begin{itemize}
	\item[(a)] If $H$ is compact, then there is an inner product $\langle \cdot, \cdot \rangle'$ 
	on $\R^d$, with corresponding norm $\| \cdot \|'$, with respect to which every $A \in H$
	is an isometry.
	\item[(b)] If $H$ is noncompact, then there exist distinct, proper, nontrivial linear subspaces $E^1, \cdots, E^p \subset \R^d$, $p \geq 1$,
	with the following properties.
		\begin{itemize}
		\item[(i)] We have $\eta(\cup E^i) = 1$;
		\item[(ii)] For all $A \in H$, we have $A E^i = E^{\pi(i)}$ for all $1 \leq i \leq p$, where $\pi = \pi_A$ is a permutation
			on $\{ 1, \cdots, p\}$; and
		\item[(iii)] For each $1 \leq i \leq p$ there is an inner product $\langle \cdot , \cdot \rangle^i$ on $E^i$ such that
		for all $A \in H$, we have that $A|_{E^i}$ is conformal with respect to the inner products
			$\langle \cdot, \cdot \rangle^i, \langle \cdot, \cdot \rangle^{\pi(i)}$ respectively.
		\end{itemize}
\end{itemize}
\end{lemma}
Lemma \ref{lem:subgroups}(a) can be found in Proposition 6.7 (ii) in \cite{baxendale1989lyapunov},
while the argument for Lemma \ref{lem:subgroups}(b) is an extension of arguments appearing in
the proof of Theorem 8.6 in \cite{furstenberg1963noncommuting}.
Since Lemma \ref{lem:subgroups} is crucial to our approach and contains strictly more information than 
what the authors can find in the literature, we provide a proof sketch later on in Section \ref{subsubsec:classifyRDS}.

Building off Lemma \ref{lem:subgroups}, we give below
a corresponding classification of the linear cocycles $\mathcal A$
preserving the invariant measure family $(\bar \nu_z)$ as in \eqref{eq:surelyInvariantMeasureFamily}.

\begin{theorem}[Classification of invariant fiber measure families]\label{thm:classification}
	Assume $d \leq 3$, and
	assume the setting, notation and conclusions of Proposition \ref{prop:contVaryingMeasure}
	and Lemma \ref{lem:almostSurelyToTopological}. Let $(\bar \nu_z)_{z \in \supp \mu}$
	denote the invariant measure family so-obtained. Then, one of the following alternatives holds.
		\begin{itemize}
			\item[(a)] There is a continuously-varying assignment to each $z \in \supp \mu$ of an 
			inner product $\langle \cdot, \cdot \rangle_z$ on $\R^d$ with the property that for all
			$(T, A) \in \Sc$ and $z \in \supp \mu$, we have that $A_z : (\R^d, \langle \cdot, \cdot \rangle_z) \to (\R^d, \langle \cdot, \cdot \rangle_{T z}) $ is an isometry.
			\item[(b)] For some $p \geq 1$, the following holds. There are $p$ measurably-varying
			assignments to each $z \in \supp \mu$ of a proper, distinct, nontrivial
			 linear subspace $E^i_z \subsetneq \R^d, 1 \leq i \leq p$,
			  with 
			 the property that for each $z \in \supp \mu$
			 and $(T, A) \in \Sc$, we have $A_z E^i_z = E^{\pi(i)}_{T z}$
			 for all $1 \leq i \leq p$, where
			 $\pi = \pi_{(T, A)}$ is a permutation on $\{ 1, \cdots, p\}$. Moreover,
			 $\bar \nu_z( \cup_{i = 1}^p E^i_z) = 1$. 
			 
			 Finally, the collection $(E^i_z)$ is \emph{locally continuous up to re-labelling}: 
			 for every $z \in \supp \mu$ there is an open neighborhood 
			 $U \subset Z$ and a labelling of the subspaces $E^i_z, z \in U \cap \supp \mu$
			 with the property that $z \mapsto E^i_z, z \in U \cap \supp \mu$ is continuously varying.
		\end{itemize}
\end{theorem}
\noindent The proof of Theorem \ref{thm:classification} deviates significantly from that in Theorem 6.8 in  
 \cite{baxendale1989lyapunov}, particularly where it is proved that the objects in alternatives (a) and (b) above 
are continuously varying. See Remark \ref{rmk:badFiberMeasures} below for a discussion
of the subtleties involved.

\bigskip

For the remainder of Section \ref{subsubsec:classifyRDS} we will prove Lemma \ref{lem:subgroups} and Theorem \ref{thm:classification}.

\subsection*{Proof of Lemma \ref{lem:subgroups}}

We will prove Lemma \ref{lem:subgroups} for any value of the dimension $d$.
Let us first dispense with the relatively easier proof of part (a), i.e., the case when $H = H_\eta$ is a compact subgroup of $SL_d(\R)$.
If $H$ is compact, then it admits a right-invariant Haar probability measure $\gamma$ (Proposition 11.4 in \cite{folland2013real}). That is,
$\gamma$ is a Borel probability measure on $H$ with the property that for any $A \in H$ and Borel $K \subset H$,
we have $\gamma(KA) = \gamma(K)$. With $(\cdot, \cdot)$ the standard inner product on $\R^d$, we define
$\langle \cdot, \cdot \rangle'$ on $\R^d$ for $v, w \in \R^d$ by
\[
\langle v, w \rangle' := \int_H ( A' v, A' w) \, \dee \gamma(A') \, .
\]
Using right-invariance of $\gamma$, one easily checks that $\langle A v, A w \rangle' = \langle v, w \rangle'$ for all $v, w \in \R^d$ and $A \in H$.
This completes the proof of Lemma \ref{lem:subgroups} in case (a).

Before proceeding to case (b), let us state and prove the following useful Claim.

\begin{claim}\label{cla:splitSubspaces}
Let $k \geq 1$ and let $(M_n)$ be a sequence of determinant 1 matrices in $M_{k \times k}(\R)$ for which $|M_n| \to \infty$ as $n \to \infty$. 
	Then, on refining to a subsequence $(M_{n'})$, there exist proper linear subspaces 
	$V^1, V^2 \subset \R^k$ for which $\operatorname{dist}(M_{n'} v, V^2) \to 0$ as $n' \to \infty$ for all 
	$v \notin V^1$.
\end{claim}
\begin{proof}[Proof of Claim \ref{cla:splitSubspaces}]
Using the fact that $\det M_n \equiv 1$ for all $n$, we can, without loss, pass to a subsequence with the property that
for some fixed $1 \leq l < k$, we have
%$\sigma_l(M_n) > \sigma_{l + 1}(M_n)$ for all $n$, and that 
\begin{align}\label{eq:SVtoInfty}
\frac{\sigma_l(M_n)}{ \sigma_{l + 1}(M_n)} \to \infty \, .
\end{align}
Applying the Singular Value Decomposition to each $M_n$, let $V^1_n$ be the unique $(k - l)$-dimensional subspace for which $|M_n|_{V^1_n}| = \sigma_{l + 1}(M_n)$, and let
$V^2_n$ be the unique $l$-dimensional subspace for which $|M_n^{-1} |_{V^2_n}| = (\sigma_l(M_n))^{-1}$.
Passing to a further subsequence, we can assume that the subspaces $V^1_n, V^2_n$ converge to subspaces $V^1, V^2$, respectively.
It now follows from \eqref{eq:SVtoInfty} that for all $v \notin V^1$, $\lim_{n \to \infty} \operatorname{dist}(M_n v, V^2) = 0$, as desired.
\end{proof}

We now proceed to case (b), which we prove in a series of Lemmas. 
Assume $H = H_\eta$ is noncompact, and consider the set $\mathcal G$ of finite tuples of proper, nontrivial, distinct subspaces 
$(E^i)_{i = 1}^p$ of $\R^d$ for which $\eta(\cup_i E^i) = 1$.
Applying Claim \ref{cla:splitSubspaces} to a sequence $\{ M_n\} \subset H$ with $|M_n| \to \infty$, note that the pair $(V^i)_{i = 1}^2$ 
so-obtained is such a tuple.
 If $(E^i)_{i = 1}^p, (\check E^i)_{i = 1}^{\check p}$ are two such tuples, 
let us write $(E^i) \leq (\check E^i)$ if $\cup_i E_i \subset \cup_i \check E_i$. Note that $\leq$ is a partial order on $\mathcal G$
We say that two tuples $(E^i)_{i =1 }^p, (\check E^j)_{j = 1}^{\check p}$ in $\mathcal G$ are \emph{equivalent up to relabeling} if $p = \check p$
and there is some permutation $\pi$ on $\{ 1, \cdots, p\}$ for which $\check E^j = E^{\pi(j)}$ for all $1 \leq j \leq p$.

\begin{lemma}\label{lem:subgroupDeeper}
Let $\eta, H_\eta$ be as in the setting of Lemma \ref{lem:subgroups} and assume $H_\eta$ is noncompact (case (b)).
	Then, there is a unique tuple $(E^i)_{i = 1}^p$ (up to relabeling) of distinct, proper and nontrivial linear subspaces of $\R^d$
	minimal with respect to the partial order $\leq$ on $\mathcal G$.
		This tuple has the property that for each $A \in H_\eta$, there is a permutation $\pi = \pi_A$ of $\{ 1, \cdots, p\}$ for which
		$A E^i = E^{\pi(i)}$ for all $1 \leq i \leq p$.
\end{lemma}
Lemma \ref{lem:subgroupDeeper} is straightforward and left to the reader (see Theorem 8.6 in \cite{furstenberg1963noncommuting} for more detail).
The minimal tuple $(E^i)$ therefore satisfies conditions (i) -- (ii) in Lemma \ref{lem:subgroups}. Item (iii) is verified below.

\begin{lemma}\label{lem:item4}
	For each $1 \leq i \leq p$, there is an inner product $\langle \cdot, \cdot \rangle^i$ on $E^i$ with the property that for each $A \in H$,
	we have that $A : (E^i, \langle \cdot, \cdot \rangle^i) \to (E^{\pi(i)}, \langle \cdot, \cdot \rangle^{\pi(i)})$, $\pi = \pi_A$, is a conformal mapping.
\end{lemma}

\begin{proof}
For (ii), form the subgroup $\tilde H = \tilde H_\eta = \{ A \in H_\eta : A E^i = E^i \text{ for all } 1 \leq i \leq p\}$. As one can check, 
 $\tilde H \subset H$ is a closed, normal subgroup of finite index. The quotient group $H / \tilde H$ is naturally 
 isomorphic to a subgroup $\mathfrak S$ of the group of permutations on $p$ symbols. 
Let us assume for the moment that $\mathfrak S$ acts transitively\footnote{Let $S \subset \{ 1, \cdots, p \}$ and assume $\mathfrak S S = S$. We say that $\mathfrak S$ acts \emph{transitively} on $S $ if for all $i, j \in S$ there is some $\pi \in \mathfrak S$ for which $\pi(i) = j$.} on $\{ 1, \cdots, p\}$; we will remove this restriction at the end of the proof.

Fix an arbitrary $i \in \{ 1, \cdots, p\}$ and form 
\[
\check H^{(i)} = \{ (\det (A|_{E^i}))^{- \frac{1}{\dim E^i}} A|_{E^i} : A \in \tilde H\} \, .
\]
Note that linear operators in $\check H^{(i)}$ preserve the measure $\eta|_{E^i}$. 
Since any $A \in \tilde H$ maps $E^i$ into itself, we can think of $\check H^{(i)}$ as a subgroup of $SL_{\dim E^i}(\R)$ 
on identifying $E^i$ with $\R^{\dim E^i}$. We claim that $\check H^{(i)}$ is compact. If not, then by Claim \ref{cla:splitSubspaces}
there are proper linear subspaces $\check V^1, \check V^2 \subset E^i$ for which $\eta(\check V^1 \cup \check V^2) = \eta(E^i)$. This contradicts
minimality of $(E^i)_{i =1 }^p$ as in item (i). Thus $\check H^{(i)}$ is compact; it now follows from Lemma \ref{lem:subgroups}(a) 
that there exists an inner product $\langle \cdot, \cdot \rangle^i$ on $E^i$ with respect to which $\check H^{(i)}$ acts isometrically. 
Equivalently, linear operators of the form $A|_{E^i}, A \in \tilde H$ act conformally with respect to $\langle \cdot, \cdot \rangle^i$.

We now define $\langle \cdot, \cdot \rangle^j, 1 \leq j \leq p, j \neq i$ as follows: for each such $j$, fix an $M \in H$ for which $M E^i = E^j$
(such an $M$ exists since $\mathfrak S = H / \tilde H$ acts transitively on $\{1, \cdots, p\}$ by assumption) and define 
\begin{align}\label{eq:pushInnerPRoduct}
\langle v, w \rangle^j = \langle M^{-1} v, M^{-1} w \rangle^i \, , \quad v, w \in E^j \, .
\end{align}
This definition is independent of $M$: if $M' E^i = E^j$ for some other $M' \in H$, then 
$\langle (M')^{-1} v, (M')^{-1} w \rangle^i = \langle M^{-1} v, M^{-1} w \rangle^i = \langle v, w \rangle^j$ holds for all $v, w \in E^j$.
By a similar computation, one checks that if $A \in H$ maps $A E^i = E^j$, then $A$ is conformal with respect to the inner products
$\langle \cdot, \cdot \rangle^i, \langle \cdot, \cdot \rangle^j$, respectively. This completes the proof when $\mathfrak S \cong H / \tilde H$ 
acts transitively on $\{ 1, \cdots, p\}$.

Let us now address the situation when $\mathfrak S$ does not act transitively on $\{ 1, \cdots, p\}$. In this case, by a standard argument
 there is a unique  partition of  $\{ 1, \cdots, p\}$ into disjoint sets $\mathcal P_l, 1 \leq l \leq k$, such that for each partition atom $\mathcal P_l$, we have 
 (1)  $\mathfrak S \mathcal P_l = \mathcal P_l$, and (2) $\mathfrak S$ acts transitively
on $\mathcal P_l$.
For each $\mathcal P_l$, repeat the construction of $\langle \cdot, \cdot \rangle^i$ for some
fixed arbitrary $i \in \mathcal P_l$, and then define $\langle \cdot, \cdot \rangle^j, j \in \mathcal P_l, j \neq i$ as in \eqref{eq:pushInnerPRoduct}
for some arbitrary $M \in H$ sending $M E^i = E^j$ (such an $M$ exists since $\mathfrak S$ acts transitively on $\mathcal P_l$ by construction). 
Lemma \ref{lem:item4} now follows from the previous arguments,
 since for all $A \in H$, we can have $A E^i = E^j$ only if $i, j$ belong to the same $\mathcal P_l$ for some $1 \leq l \leq k$.
\end{proof}

\subsubsection*{Proof of Theorem \ref{thm:classification}}

We first give the following preliminary Lemma. For $z \in \supp \mu$, define
\[
O_z = \{ T z : (T, A) \in \Sc\} \, .
\]
Note that $O_z \subset \supp \mu$ holds by Lemma \ref{lem:almostSurelyToTopological}. 
Using ergodicity of $\mu$ and the strong Feller property, we get the following.

\begin{lemma}\label{lem:fullOrbitMeasu}
	For all $z \in \supp \mu$, we have $\mu(O_z) = 1$.
\end{lemma}
\begin{proof}
	
	First, let us check that $O_z$ is a $(P_t, \mu)$-invariant set in the sense of 
	Definition \ref{defn:PtmuInvariant}. Fix $t > 0$ and let $y \in O_z$.
	Then, $y = T z$ for some $(T, A) \in \Sc$. Now, fix a $\P$-generic $\omega \in \Omega$
	and set $T' = T^t_\omega, A' = \Ac^t_{\omega}$. Noting $(T', A') \in \Sc$ with probability
	1, we see that $T' y = T' \circ T z$, hence $T' y \in O_z$. Since $y \in O_z$ was arbitrary, we conclude
	$\Tc^t_\omega y \in O_z$ for any $t \geq 0$ with probability $1$, hence
	$O_z$ is $(P_t, \mu)$-invariant.
	
	It follows from ergodicity for $\mu$ (Definition \ref{defn:PtmuInvariant}) that
	$O_z$ has zero or full $\mu$-measure. To check $\mu(O_z) > 0$, assume otherwise
	and observe that by stationarity, $P_t(y, O_z) = 0$ for $\mu$-almost all $y \in Z$. From the
	ultra-Feller property for the semigroup $(P_t)$ as in Lemma \ref{lem:ultraFeller}, we conclude
	$P_t(z, O_z) = 0$, a contradiction (note $\{ y \in Z : P_t (y, O_z) = 0 \}$ must be
	dense in $\supp \mu$). We conclude $\mu(O_z) > 0$, hence $\mu(O_z) = 1$.
\end{proof}

\begin{proof}[Proof of Theorem \ref{thm:classification}]
Fix $z_0 \in \supp \mu$, thought of as a reference point, and consider the $SL_2(\R)$ subgroup 
\[
H_{z_0} = \{ A \in SL_2(\R) : A_* \bar \nu_{z_0} = \bar \nu_{z_0}\} \, .
\]
Note that $H_{z_0}$ is closed by Lemma \ref{lem:subgroups}.
We claim that if $H_{z_0}$ is compact we are in case (a), while if $H_{z_0}$ is noncompact
then we are in case (b). Crucially, this distinction does not depend on the choice of reference
point $z_0 \in Z$; see Remark \ref{rmk:choiceOfReferencePt} below for a discussion of this point.

\bigskip

\noindent {\bf Case (a): $H_{z_0}$ is compact.} 
By Lemma \ref{lem:subgroups} there is an inner product $\langle \cdot, \cdot \rangle_{z_0}$ 
with respect to which all matrices in $H_{z_0}$ act as isometries. We define the family
$\{ \langle \cdot, \cdot \rangle_z \}_{z \in \supp \mu}$ as follows. For each $z \in \supp \mu$,
fix $y \in O_{z_0} \cap O_z$ (such a point exists since $\mu(O_{z_0} \cap O_z) = 1$ by Lemma \ref{lem:fullOrbitMeasu}) 
and let $(T, A) , (T', A') \in \Sc$ be such that $T z_0 = y, T' z = y$.

For $v, w \in \R^d$ we define
\[
\langle v, w \rangle_z = \langle A_{z_0}^{-1} \circ A_z' v, A_{z_0}^{-1}\circ A_z' w \rangle_{z_0} \, .
\]
Let us check this definition does not depend on the exact
choice of $(T, A), (T', A')$. If $(\bar T, \bar A), (\bar T', \bar A') \in \Sc$ are any other elements
for which $\bar T z_0 =  y, \bar T' z =  y$, then Lemma \ref{lem:almostSurelyToTopological}
implies $(\bar A_{z_0})^{-1} \bar A_z' (A_z')^{-1} A_{z_0} \in H_{z_0}$, and so
\[
\langle A_{z_0}^{-1} \circ A_z' v, A_{z_0}^{-1}\circ A_z' w \rangle_{z_0} = \langle \bar A_{z_0}^{-1} \circ \bar A_z' v, \bar A_{z_0}^{-1} \circ \bar A_z' w \rangle_{z_0}
\]
holds by Lemma \ref{lem:subgroups}(a).
By a similar proof, one checks that for each $(T, A) \in \Sc$
and $z \in \supp \mu$, we have that $A_z: (\R^d, \langle \cdot, \cdot \rangle_z) \to (\R^d, \langle \cdot, \cdot \rangle_{T z})$ is an isometry.

To prove continuity of $z \mapsto \langle \cdot, \cdot \rangle_z$ we do the following.
For each $z \in \supp \mu$, the inner product $\langle \cdot, \cdot \rangle_z$ gives rise to
a Euclidean volume on $\R^d$ and an induced volume $\tilde \nu_z$ on $P^{d-1}$. By the isometry property, 
it follows that for all $(T, A) \in \Sc$, we have that
$(A_z)_* \tilde \nu_z = \tilde \nu_{T z}$ for all $z \in \supp \mu$.
Thus $(\tilde \nu_z)_{z \in \supp \mu}$
defines an invariant measure family on $\supp \mu$. Repeating the proof of Proposition \ref{prop:contVaryingMeasure}
for this new invariant measure family, 
we conclude $(\tilde \nu_y)$ is continuously varying in the weak$^*$ topology.

From the weak$^*$ continuity of $z \mapsto \tilde \nu_z$ and the fact that $\tilde \nu_z \ll \Leb_{P^{d-1}}$ for all
$z$, we conclude that the densities $\rho_z := \frac{d \tilde \nu_z}{d \Leb_{P^{d-1}}}$,  $z \mapsto \rho_z : P^{d-1} \to \R$, 
vary continuously in the uniform norm on $C(P^{d-1}, \R)$. 
It is now straightforward to check that the corresponding inner products 
$z \mapsto \langle \cdot, \cdot \rangle_z$ vary continuously.

\begin{remark}\label{rmk:badFiberMeasures}
	It is a subtle point in the proof of Theorem \ref{thm:classification}(a) above that 
	the original invariant measure family $(\bar \nu_z)_{z \in \supp \mu}$ need not coincide
	with the measure family $(\tilde \nu_z)_{z \in \supp \mu}$. Indeed, we do not rule out the possibility that
	the $(\bar \nu_z)$ consist of some combination of atomic, singular continuous and 
	absolutely continuous measures. As such, one cannot deduce continuity of the resulting inner products
	$\langle \cdot , \cdot \rangle_z, z \in \supp \mu$ directly from the $(\bar \nu_z)$. As we will see below, the proof
	of Theorem \ref{thm:classification}(b) has a similar complication which must be addressed.
	
	By comparison, Theorem 6.8 in \cite{baxendale1989lyapunov} avoids this subtlety
	for two reasons: (1) in that framework, under a nondegeneracy condition
	it follows that the fiber measures $\bar \nu_z$ are automatically absolutely continuous w.r.t. the volume on $P^{d-1}$; 
	and (2) Theorem 6.8 in \cite{baxendale1989lyapunov} invokes an additional hypothesis 
	that we are not able to justify either at the level
	of generality of Theorem \ref{thm:classification} or for the Lagrangian flow corresponding
	to the infinite-dimensional Systems \ref{sys:NSE}, \ref{sys:3DNSE}.
\end{remark}

\bigskip

\noindent {\bf Case (b): $H_{z_0}$ is noncompact}

Let $E^i_{z_0} = E^i, 1 \leq i \leq p$ be as in Lemma \ref{lem:subgroups}(b) applied to $H = H_{z_0}$.
For each $i$, let $\langle \cdot, \cdot \rangle^i$ denote the corresponding inner product on $E^i = E^i_{z_0}$. 
For $z \in \supp \mu$ we define $E^i_{z}$ as follows. Fix $y \in O_{z_0} \cap O_z$, as in the
proof for case (a), and fix $(T, A), (T', A')$ for which $T {z_0} = y, T' z = y$.
 We define
	\[
	E^i_{z} = (A'_z)^{-1} \circ A_{z_0} (E^i_{z_0}) \, .
	\]
	We also define the inner products $\langle \cdot, \cdot \rangle^i_z$ on $E^i_z$ by setting,
	for $v, w \in E^i_z$,
	\[
	\langle v, w \rangle^i_z = \langle (A_{z_0})^{-1} A'_z v, (A_{z_0})^{-1} A'_z w \rangle^i \, .
	\]

	As in the proof of case (a), one checks that the above definitions do not depend on the
	exact choices of $y \in O_z \cap O_{z_0}$ or $(T, A), (T', A') \in \Sc$. By a similar check, the invariance property for 
	the $E^i_{z}, z \in \supp \mu$ similarly holds, and moreover, for $(T, A) \in \Sc$ and $z \in \supp \mu$, we have that
	$A_z : (E^i_z, \langle \cdot, \cdot \rangle^i_z) \to  (E^{\pi(i)}_{T z}, \langle \cdot, \cdot \rangle^{\pi(i)}_{T z})$
	is conformal.
	
Let us now prove the continuity statement.  
Observe that since $d \leq 3$, there are two cases: either $\dim E^i_z \equiv 1$ for all
$i$ or $\dim E^i_z = 2$ for some $i, z$. If the former, local continuity of $z \mapsto E^i_z$
up to relabeling follows immediately from the fact that $\bar \nu_z|_{E^i_z}$ is a delta mass
supported on the projective point corresponding to $E^i_z$.
If the latter, then by Claim \ref{cla:splitSubspaces} we must have that $p \leq 2$ and that 
at most one of the $E^i_z$ is two-dimensional for each $z \in \supp \mu$. We focus on the case
$p = 1$; essentially the same proof applies when $p = 2$. Hereafter let us write $E_z := E^1_z$. Note that this
can only occur when $d = 3$, which hereafter we assume.

In analogy with the proof of Theorem \ref{thm:classification}(a), consider for each $z \in \supp \mu$,
 the Euclidean volume $m_z$ on $E_z \subset \R^3$ induced by the inner product $\langle \cdot , \cdot \rangle_z := \langle \cdot, \cdot \rangle_z^1$.
This induces a normalized volume $\tilde \nu_z$ on the projectivization of $E_z$ in $P^{2}$.
As in the proof for case (a), the fiber measure family $(\tilde \nu_z)_{z \in \supp \mu}$ is invariant as in \eqref{eq:invariantMeasuresFamily}.
This follows from the conformality property for the inner product $\langle \cdot, \cdot \rangle_z$. 
As in the proof of case (a),
we can repeat the arguments of Proposition \ref{prop:contVaryingMeasure}, from which we obtain that the family $(\tilde \nu_z)$ is weak$^*$ continuous.
Continuity of $z \mapsto E_z$ now follows. \qedhere
\end{proof}

We conclude Section \ref{subsubsec:classifyRDS} with several remarks.

\begin{remark}\label{rmk:choiceOfReferencePt}
The determination between case (a) and (b) made at the beginning of the proof of Theorem \ref{thm:classification}
does not depend on the reference point $z_0 \in \supp \mu$. Indeed, given $z, z' \in \supp \mu$ one can obtain a 
group isomorphism $H_z \to H_{z'}$ as follows: fix $y \in O_z \cap O_{z'}$ and let $(T, A), (T', A') \in \Sc$ be such that
$T z = y, T' z = y$. Then, the mapping $H_z \to H_{z'}$ sending $H_z \ni M \mapsto  (A'_{z'})^{-1} A_z M  A^{-1}_z A'_{z'} \in H_{z'}$
is an isomorphism from $H_z$ to $H_{z'}$.
\end{remark}

\begin{remark}\label{rmk:why3D}
The restriction to $d \leq 3$ is only relevant in case (b) of Theorem \ref{thm:classification}. For $d \geq 4$ the result is likely
to be true, but the proof is lengthier due to the fact that among the $E^i_z, z \in \supp \mu, 1 \leq i \leq p$ there may be arbitrarily many subspaces
of dimension $\geq 2$. Thus, the trick applied in case (b) above must be applied to the projectivization of the Euclidean volume on each $E^i_z$ separately,
and continuity derived in this way. Since the case $d \leq 3$ suits the purpose of our main application in this paper, we leave off the $d \geq 4$ case to a future work.
\end{remark}

\begin{remark}\label{rmk:whyNovelRDS}
Let us summarize the differences between Theorem 6.8 in \cite{baxendale1989lyapunov} and the analogue pursued here in Section \ref{subsubsec:classifyRDS}.
To start, Theorem 6.8 of \cite{baxendale1989lyapunov} proves the classification in Theorem \ref{thm:classification} above in the special case when $Z$
is a locally compact Riemannian manifold, $\Tc$ is the stochastic flow of diffeomorphisms generated by a hypoelliptic SDE satisfying suitable nondegeneracy properties, and
$\Ac$ is its corresponding derivative cocycle.

In comparison, Theorem \ref{thm:classification} does not require that $\Ac$ be the derivative cocycle of $\Tc$. This requires that we work 
with the product space $\Cc$ of pairs of mappings and cocycles, as is done in Lemma \ref{lem:almostSurelyToTopological}.
Moreover, and arguably of greater consequence, is the fact that the base RDS $\Tc$ is not necessarily invertible, 
nor is its phase space $Z$ locally compact. These differences 
are emblematic of dynamics on infinite-dimensional spaces and are 
exemplified by our intended application to the Navier-Stokes equations 
and more generally to regularizing semilinear parabolic problems.
This raises numerous issues which we have dealt with over the course of Section \ref{sec:RDS}, e.g., 
the definition of the topology on observables with respect to which $(P^t)$ is a $C^0$-semigroup (Proposition \ref{prop:semigroupRegularity1}).

Finally, Theorem 6.8 of \cite{baxendale1989lyapunov}, of which the main result Theorem \ref{thm:classification} is an analogue, invokes
an additional hypothesis to get continuity of the obtained invariant inner products in case (a) (resp., finite union of proper linear subspaces in case (b)). 
This additional hypothesis is not accessible in our setting.
This brings up a significant subtlety (Remark \ref{rmk:badFiberMeasures}),
unique to our setting, which our argument addresses.
\end{remark}

\subsection{Sufficient condition for $\lambda_1 > 0$: approximate controllability criteria}\label{subsubsec:approxControlRuleOut}

We will now state a weaker version of the criterion (C) in Section \ref{subsec:FurstenbergFD} 
which can be used to rule out the alternatives (a) and (b) in Theorem \ref{thm:classification}.

\begin{definition}\label{defn:condCprimeRDS}
	We say that the cocycle $\Ac$ satisfies the \emph{approximate controllability condition} (C') if 
	there exist $z, z' \in \supp \mu$ such that $z'$ belongs to the support of the measure 
	$P_{t_0}(z, \cdot)$ for some $t_0 > 0$, and we have each of the following.
		\begin{itemize}
			\item[(a)] We have
				$Q_{t_0} ((z, \Id), B_\epsilon(z') \times \{ A \in SL_d(\R) : | A | > M \}) > 0$ for any $\epsilon, M > 0$.
			\item[(b)] For any $v \in P^{d-1}$, open $V \subset P^{d-1}$ and $\epsilon > 0$,
			we have $\widehat P_{t_0}((z, v), B_\epsilon(z') \times V) > 0$.
		\end{itemize}	
\end{definition}

We can now prove the following.

\begin{proposition}\label{prop:approxControlRuleOut}
Let $d \leq 3$. Let $\Ac$ be an $SL_d(\R)$ linear cocycle as in Section \ref{subsubsec:linearCocycleBasicRDS} over a continuous RDS $\Tc$ as in Section \ref{subsubsec:basicSetupRDS}
 satisfying (H1) -- (H3) for which the Markov semigroup $(P_t)$ has the strong Feller property.
Let $\mu$ be an ergodic stationary measure for which the approximate controllability condition (C') holds. 
Then, $\lambda^+ > \lambda^-$, and in particular $\lambda_1 > 0$, in the MET (Theorem \ref{thm:MET}).
\end{proposition}

\begin{proof}
	If $\lambda^+ = \lambda^-$, then Theorem \ref{thm:furstenberg} applies, and so
	either case (a) or case (b) holds in Theorem \ref{thm:classification}. 
	
	We start by ruling out (a). For $y \in \supp \mu$, write $| \cdot|_y$ for the norm
	corresponding to the inner product $\langle \cdot, \cdot \rangle_y$.
	Let 
	\[
	\kappa = \max\bigg\{  \max_{v \in \R^d \setminus \{ 0\}} \frac{|v|_z}{|v|} ,  
	\max_{v \in \R^d \setminus \{ 0\}} \frac{|v|_{z'}}{|v|} \bigg\}\, . 
	\]
	Fix $\epsilon > 0$ so that $\frac12 |\cdot|_y \leq |\cdot|_{z'} \leq 2 |\cdot|_y$ for all $y \in B_\epsilon(z')$.
	
	Now, condition (C')(a) says that there is a $\P$-positive measure set $E \subset \Omega$
	such that $\Tc^{t_0}_\omega z \in B_\epsilon(z')$ and $|\Ac^{t_0}_{\omega, z}| > 2 \kappa^2$
	 for all $\omega \in E$. Without loss we can assume
	$\{ (\Tc_\omega^{t_0}, \Ac_{\omega, \cdot}^{t_0}) : \omega \in E \} \subset \Sc_{t_0}$,
	perhaps on paring off an $\P$-measure zero set from $E$.	
	By Theorem \ref{thm:classification}(a), for all $\omega \in E$ 
	we deduce $|\Ac^{t_0}_{\omega,z}|_{z, y} = 1$, where
	$|\cdot|_{z, y}$ is the matrix norm induced by the norms $|\cdot|_z$ at $z$ and $|\cdot|_y$
	at $y = \Tc^{t_0}_\omega z$. From this we obtain the estimate 
	$|A^{t_0}_{\omega, z}| \leq 2 \kappa^2$ in the matrix
	norm induced from $|\cdot|$. This is a contradiction to (C')(a).
	
	\medskip
	
	Turning to case (b), 	
	take $\epsilon > 0$ sufficiently small so that (i) a labelling of the $E^i_y, y \in B_\epsilon(z')$
	exists for which $y \mapsto E^i_y$ is continuous for $1 \leq i \leq p$, and (ii) there is an open set $V \subset P^{d-1}$
	for which $V \cap (\cup_i E^i_y) = \emptyset$ for all $y \in B_\epsilon(z')$.
	
	Fix an arbitrary $1 \leq i \leq p$ and $v \in E^i_z \setminus (\cup_{j \neq i} E^j_z)$. 
	Condition (C')(b) implies that there is a $\P$-positive measure set $E \subset \Omega$
	such that for all $\omega \in E$, we have $\Tc_\omega^{t_0} z \in B_\epsilon(z')$ 
	and $\Ac^{t_0}_{\omega, z} v \in V$. As before, on paring off a $\P$-measure zero set
	we can assume $(\Tc^{t_0}_\omega, \Ac^{t_0}_{\omega}) \in \Sc_{t_0}$ for all $\omega \in E$,
	from which we deduce (Theorem \ref{thm:classification}(b)) that
	 $\Ac^{t_0}_{\omega, z} E^j_z = E^{\pi_\omega(j)}_{\Tc^{t_0}_\omega z}$ for
	all $\omega \in E$ and $1 \leq j \leq p$, where $\pi_\omega$ is some permutation on $\{ 1, \cdots, p\}$. 
	But at $j = i$ this is a contradiction, since $v \in E^i_z$ yet
	$\Ac^{t_0}_{\omega, z} v \notin E^l_{\Tc_\omega^{t_0} z}$ for any $l \in \{ 1, \cdots, p\}$
	by construction. 
\end{proof}

%% Application to the finite dimensional fluids
\section{Lie brackets and H\"{o}rmander's condition}\label{sec:galerkin}
%!TEX root = master.tex
The main goal of this section is to explore how noise in the low modes of a fluid model spreads to other variables coupled to the flow. Specifically, for $(u_t)$ given by Systems \ref{sys:2DStokes} and \ref{sys:Galerkin}, we will show that the projective processes $(u_t,x_t,v_t)$, $(u_t,x_t,\check{v}_t)$, and the matrix process $(u_t,x_t,A_t)$ are all generated by vector fields satisfying the parabolic H\"ormander condition in both $2$ and $3$ dimensions (Definition \ref{def:Hor}). Using the a priori estimates on $(u_t)$ and that $\T^d \times P^{d-1}$ is compact, H\"ormander's theorem (see e.g. \cite{Hormander67,Hormander85} and \cite{DaPrato14,HairerNotes11}) then implies $(u_t,x_t)$, $(u_t,x_t,v_t)$, $(u_t,x_t,\check{v}_t)$ have absolutely continuous Markov kernels (with respect to Lebesgue measures) and unique stationary measures. 
Similarly, $(u_t,x_t,A_t)$ also has an absolutely continuous Markov kernel and therefore the arguments
given in Section \ref{subsubsec:ruleOutFurstFDOutline} are validated.
 Theorem \ref{thm:Lyap} hence follows for Systems \ref{sys:2DStokes} and \ref{sys:Galerkin}. 

 In what follows it is technically more convenient to deal with the space $\S^{d-1}$ in place of $P^{d-1}$ while still denoting $v_t$ and $\check{v}_t$ the corresponding versions in $\S^{d-1}$. Since $P^{d-1}$ and $\S^{d-1}$ are locally diffeomorphic, proving H\"{o}rmander's condition on $\S^{d-1}$ implies H\"{o}rmander's condition for $P^{d-1}$. 
 \subsection{Preliminaries}

Recall the orthogonal $L^2(\T^d)$ basis $\{e_k\}_{k\in\Z^d_0}$ and the family of $d\times(d-1)$ matrices $\{\gamma_k\}_{k\in\Z^d_0}$ introduced in Section \ref{subsubsec:noiseProcess} satisfying $\gamma^\top_k k = 0$ and $\gamma_k^\top\gamma_k = \Id$. We will denote for each $k\in \Z^d_0$ the column vectors $\{\gamma_k^1,\ldots \gamma^{d-1}_k\}$ of the matrix $\gamma_k$. These vectors consequently form an orthonormal basis for the subspace of vectors in $\R^d$ perpendicular to $k$. Note that for each $k\in \Z^d_0$ and $i \in \{1,\ldots,d-1\}$, $e_k\gamma_k^i$ is a divergence-free, mean-zero vector field on $\T^d$ and the collection $\{e_k\gamma_k^i \,:\, k\in\Z^d_0, \, i = \{1,\ldots,d-1\}\}$ forms an orthogonal basis for $\Wbf$ with respect to the inner product
\[
	\langle u^1, u^2\rangle_{\Wbf} = \int_{\T^d} u^1(x)\cdot u^2(x)\,\dx.
\]
This means that given a $u\in \Wbf$, we can write
\[
	u = \sum_{i=1}^{d-1}\sum_{k\in\Z^d_0} (u)_k^i e_k\gamma_k^i, \quad \text{where},\quad (u)_k^i = \frac{1}{\pi(2\pi)^{d-1}}\langle u, e_k \gamma_k^i\rangle_{\Wbf}.
\]
It follows that, given $(u_t)$ solving any of Systems \ref{sys:2DStokes} or \ref{sys:Galerkin}, we can write the equations for $(x_t,v_t)$ in $\T^d\times\S^{d-1}$ as
 \begin{align}\label{eq:xt-basis}
\frac{\dee}{\dt}x_t & = \sum_{i=1}^{d-1}\sum_{k\in\Z^d_0}(u_t)_k^i e_{k}(x_t)\gamma^i_k\\ \label{eq:vt-basis}
\frac{\dee}{\dt}v_t & = \sum_{i=1}^{d-1}\sum_{k\in\Z^d_0}(u_t)_k^i (k\cdot v_t) e_{-k}(x_t)\Pi_{v_t}\gamma^i_k.
\end{align}
Likewise the inverse transpose projective process $(\check{v}_t)$ in $\S^{d-1}$ is given by
\[
	\frac{\dee}{\dt}\check{v}_t  = -\sum_{i=1}^{d-1}\sum_{k\in\Z^d_0}(u_t)_k^i (\gamma_k^i\cdot \check{v}_t) e_{-k}(x_t)\Pi_{\check{v}_t}k.
\]
and the matrix process $(A_t)$ in $SL^d(\R)$ satisfies
\begin{equation}\label{eq:At-basis}
	\frac{\dee}{\dt}A_t  = \sum_{i=1}^{d-1}\sum_{k\in\Z^d_0}(u_t)_k^i e_{-k}(x)(\gamma_k^i\tensor k)A_t.
\end{equation}

We are interested in studying the hypoellipticity of the processes $(u_t,x_t,v_t)$ and $(u_t,x_t,A_t)$ , when $(u_t)$ is governed by System \ref{sys:2DStokes} or \ref{sys:Galerkin}. Recall that Systems \ref{sys:2DStokes} and \ref{sys:Galerkin} both live in a finite dimensional subspace $\hat{\Hbf}$ of $\Hbf$ (see Section \ref{subsec:notation}).
In both cases the process of interest will take the form of an abstract degenerate SDE
\begin{equation}\label{eq:finite-SDE}
	\dee y_t = X_0(y_t)\dt + \sum_{j=1}^M X_j \dee W^j_t
\end{equation}
on $\hat{\Hbf}\times \cM$, where $\cM$ is a finite dimensional Riemannian manifold (either $\T^d\times \S^{d-1}$ or $\T^d\times SL_d(\R)$). Here $X_0$ is a vector field on $\hat{\Hbf}\times\cM$ associated to the drift, while $\{X_j\}_{j=1}^M$ is an enumeration of the vectors $\{q_ke_k\gamma^i_k\,:\, k\in \mathcal{K},\, i=1,\ldots,d-1\}$ in $\hat{\Hbf}$.

Recall the {\em Lie bracket} (or commutator) of two vector fields $X$ and $Y$ on a smooth manifold $\cY$ is defined for each $y\in \cY$ by
\[
	[X,Y](y) = D_X Y(y) - D_Y X(y)
\]
where $D_X$ and $D_Y$ denote the directional derivatives in the direction $X$ and $Y$ respectively. The H\"ormander condition is now stated as follows:  
\begin{definition}[Parabolic H\"{o}rmander Condition] \label{def:Hor}
A family of vector fields $\{X_k\}_{k=0}^M$ on a smooth manifold $\mathcal{M}$ is said to satisfy the {\it parabolic H\"{o}rmander condition} if for each $y\in \cY$ the vectors
\[
\begin{aligned}
	&X_k(y), &&k = 1,\ldots, M\\
	&[X_k,X_j](y), &&k = 1,\ldots, M,\, j = 0,\ldots, M\\
	&\big[X_k,[X_j,X_\ell]\big](y), &&k = 1,\ldots, M,\, j,\ell = 0,\ldots, M\\
	& \quad \vdots && \quad\vdots
\end{aligned}
\]
span $T_y\cY$.
\end{definition}

\begin{theorem}[\cite{Hormander67}; see also \cite{Hormander85,DaPrato14,HairerNotes11}]
Let $P_t(y,A) = \PP\left(y_t \in A | y_0 = y\right)$ be the Markov kernel associated to the finite dimensional SDE \eqref{eq:finite-SDE}. 
If Definition \ref{def:Hor} is satisfied, then $P_t(y,\cdot)$ is absolutely continuous with respect to $\Leb_{\hat{\Hbf}\times \cM}$. 
\end{theorem}

\subsection{Lie brackets for the projective process}
In this section we study the spanning properties of Lie brackets for the process $(x_t,v_t)$ in $\T^d\times \S^{d-1}$. The equations \eqref{eq:xt-basis} and $\eqref{eq:vt-basis}$ can be written as
\[
	\frac{\dee}{\dt}\begin{pmatrix} x_t\\ v_t\end{pmatrix} = V(u_t,x_t,v_t)
\]
where $V(u,x,v)$ is the vector field defined for each $(u,x,v)\in \Hbf\times \T^d\times\S^{d-1}$ by
\[
	V(u,x,v) = \sum_{i=1}^{d-1}\sum_{k\in\Z^d_0}  \begin{pmatrix} (u)^i_ke_k(x) \gamma_k^i\\ (u)^i_k(k\cdot v)e_{-k}(x)(\Pi_{v}\gamma^i_k)\end{pmatrix} \in T_x\T^d \times T_v\S^{d-1}.
\]
Note that $V(u,x,v)$ is linear in $u$ and therefore the Lie-bracket $[e_k\gamma^i_k, V]$ does not depend on $u$ and is readily seen to be given by
\[
	[e_k\gamma^i_k, V](x,v) = \begin{pmatrix} e_k(x) \gamma_k^i\\ (k\cdot v)e_{-k}(x)(\Pi_{v}\gamma^i_k)\end{pmatrix}.
\]
The following Lemma gives sufficient conditions for $[e_k\gamma^i_k, V]$ to span $T_x\T^d\times T_v\S^{d-1}$.
\begin{lemma}
\label{lem:v-span}
Let $k^1, \ldots k^d$ be $d$ linearly independent elements of $\Z^d_0$ and define $K = \{k^1,\ldots,k^d\}\cup\{-k^1,\ldots,-k^d\}\subseteq \Z^d_0$. Then at each point $(x,v)\in \T^d\times\S^{d-1}$, we have
\[
	\Span\big\{[e_k\gamma^i_k,V](x,v)\,:\, k \in K,\, i = 1,\ldots d-1\} = T_x\T^d \times T_v\S^{d-1}.
\]
\end{lemma}
\begin{proof}
	Let $k\in K$. Using the identity $e_k^2 + e_{-k}^2 = 1$ and the fact that $-k\in K$, we find that for each $(x,v)\in \T^d\times\S^{d-1}$ (recall the symmetry $\gamma_{-k} = -\gamma_k$)
\[
	e_k(x)[e_k\gamma^i_k,V](x,v) - e_{-k}(x)[e_{-k}\gamma^i_{-k},V](x,v) = \begin{pmatrix}\gamma^i_k \\ 0\end{pmatrix}
\]
and
\[
	e_{-k}(x)[e_k\gamma^i_k,V](x,v) + e_{k}(x)[e_{-k}\gamma^i_{-k},V](x,v) = \begin{pmatrix}0 \\ (k\cdot v)(\Pi_{v}\gamma^i_k)\end{pmatrix}.
\]
Therefore it suffices to show that 
\begin{equation}\label{eq:span-vcond1}
	\Span\big\{\gamma^i_k\,:\, k \in K,\, i \in \{1,\ldots d-1\}\big\} = \R^d,
\end{equation}
and for each $v \in \S^{d-1}$
\begin{equation}\label{eq:span-vcond2}
	 \Span\big\{(k\cdot v)(\Pi_{v}\gamma^i_k)\,:\,k \in K,\, i \in \{1,\ldots d-1\}\big\} = T_v\S^{d-1}.
\end{equation}
Condition \eqref{eq:span-vcond1} follows from the linear independence of $k^1$ and $k^2$ and the fact that $\{\gamma_k^i\}_{i=1}^{d-1}$ spans the space perpendicular to $k$. Condition \eqref{eq:span-vcond2} follows from the fact that by linear independence of $k^1,\ldots, k^d$, that for each $v\in \S^{d-1}$, there exists a $k\in K$ such that $v\cdot k \neq 0$ and therefore, since $\{\gamma_k^i\}_{i=1}^{d-1}$ spans the space perpendicular to $k$, the vectors $\{\Pi_v\gamma^i_k\}_{i=1}^{d-1}$ span $T_v\S^{d-1}$.

\end{proof}

\begin{remark}
It is not difficult to see that we may replace $v_t$ with $\check{v}_t$ in the above Lemma, without changing the proof much. The only difference being that condition \eqref{eq:span-vcond2} is now replaced with
\[
	\Span\big\{(\gamma^i_k \cdot v)(\Pi_{v}k)\,:\,k \in K,\, i \in \{1,\ldots d-1\}\big\} = T_v\S^{d-1}
\]
which can be deduced from the fact that by linear independence of $k_1,\ldots,k^d$, there exists at least $d-1$ linearly independent elements $\hat{k}^1,\ldots \hat{k}^{d-1}$ of $K$ such $\gamma^i_{\hat{k}^j} \cdot v \neq 0$ for some $i=1,\ldots, d-1$ and such that $\{\Pi_v \hat{k}^j\,:\, j=1,\ldots, d-1\}$ spans $T_v\S^{d-1}$.
\end{remark}

\subsection{Lie brackets for the matrix process}
We would also like to study the spanning properties of Lie brackets for the process $(x_t,A_t)$ in $\T^d\times SL_d(\R)$. Similarly to the $(x_t,v_t)$ process, equations \eqref{eq:xt-basis} and $\eqref{eq:At-basis}$ can be written as
\[
	\frac{\dee}{\dt}\begin{pmatrix} x_t\\ A_t\end{pmatrix} = G(u_t,x_t,A_t)
\]
where for each $(u,x,A)\in \Hbf\times \T^d\times SL_{d}(\R)$
\[
	G(u,x,A) = \sum_{i=1}^{d-1}\sum_{k\in\Z^d_0}  \begin{pmatrix} (u)^i_ke_k(x) \gamma_k^i\\ (u)^i_ke_{-k}(x)(\gamma^i_k\tensor k)A\end{pmatrix} \in T_x\T^d \times T_ASL_{d}(\R).
\]
Again, $G(u,x,A)$ is linear in $u$ and so the Lie-bracket $[e_k\gamma^i_k, G]$ does not depend on $u$. 

\begin{lemma}
\label{lem:A-span}
Let $k^1,\ldots ,k^{d+1}$ be $d+1$ elements of $\Z^d_0$ given by $k^1=(0,1)$, $k^2=(1,0)$, $k^3=(1,1)$ for $d=2$ and $k^1=(0,0,1)$, $k_2=(0,1,0)$, $k^3=(0,0,1)$, $k^4 = (1,1,1)$ for $d=3$. Define $K = \{k^1,\ldots,k^{d+1}\}\cup\{-k^1,\ldots,-k^{d+1}\}\subseteq \Z^d_0$. Then at each point $(x,A)\in \T^d\times SL_{d}(\R)$, we have
\[
	\Span\big\{[e_k\gamma^i_k,G](x,A)\,:\, k \in K,\, i \in \{1,\ldots d-1\}\big\} = T_x\T^d \times T_A SL_{d}(\R).
\]
\end{lemma}
\begin{proof}
	Following the same proof strategy as in the proof of Lemma \ref{lem:v-span}, we may conclude that it suffices to show that
\[
 \Span\big\{(\gamma^i_k\tensor k)A\,:\,k\in K,\, i\in\{1,\ldots, d-1\}\big\} = T_A SL_d(\R).
\]
Using that the Lie algebra $\sl_d(\R)$ of traceless $d\times d$ matrices is linearly isomorphic to $T_A\SL_d(\R)$ by right (or left) multiplication by $A$, the above spanning condition is equivalent to showing that
\begin{equation}\label{eq:span-Acond}
	\Span\big\{(\gamma^i_k\tensor k)\,:\,k\in K,\, i\in\{1,\ldots,d-1\} \big\} = \sl_d(\R).
\end{equation}
The above condition \eqref{eq:span-Acond} follows from the fact that for the vectors $k^1,\ldots k^{d+1}$ given, the $d^2-1$ matrices $\big\{(\gamma^i_k\tensor k)\,:\,k = \{k^1,\ldots,k^{d+1}\}\,, i\in\{1,\ldots,d-1\} \big\}$ are all linearly independent in $\sl_d(\R)$. Since $\sl_d(\R)$ is $d^2-1$ dimensional, condition \eqref{eq:span-Acond} must hold.
\end{proof}

\subsection{H\"{o}rmander condition for Stokes and Galerkin-Navier-Stokes systems}
We now turn to study the hypoellipticity of the projective process $(u_t,x_t,v_t)$ and matrix process $(u_t,x_t,A_t)$ when $(u_t)$ satisfies either Systems \ref{sys:2DStokes} or \ref{sys:Galerkin}. We will define the vector field $U^S$ on $\Hbf_{\mathcal{K}}$ associated with the Stokes System \ref{sys:2DStokes} by
\[
	U^{S}(u) := -\sum_{i=1}^{d-1}\sum_{k\in \mathcal{K}}|k|^2 (u)_k^i e_k\gamma_k^i
\]
and the vector field $U^{NS}$ on $\Hbf_N$ associated with the Galerkin-Navier-Stokes System \ref{sys:Galerkin} by
\[
	U^{NS}(u) := -\sum_{i=1}^{d-1}\sum_{|k|_\infty \leq N} \left(B_k^i(u,u) +|k|^2 (u)_k^i\right) e_k\gamma_k^i
\]
where for each $u\in \Hbf_{N}$ (recall the definition of $B$ from Section \ref{subsec:notation}), 
\[
	B_k^i(u,u) := \frac{1}{\pi(2\pi)^{d-1}}\langle B(u,u), e_k\gamma_k^i \rangle_{\Wbf}.
\]

The following Lemma gives sufficient conditions for $(u_t,x_t,v_t)$ to satisfy the parabolic H\"{o}rmander condition:
\begin{lemma}\label{lem:Hormander-span} Let $\{X_j\}_{j=1}^M$ denote an enumeration of the vectors $\{q_ke_k\gamma^i_k\,:\, k\in \mathcal{K},\, i=1,\ldots,d-1\}$ and let $X_0$ be a vector fields on $\hat{\Hbf}\times\T^d\times \S^{d-1}$ of the form
\[
	X_0(u,x,v) = U(u) + V(u,x,v).
\]
The following holds:
\begin{enumerate}
\item If $U(u) = U^S(u)$ and $\mathcal{K}$ contains the elements $(1,0),(0,1)$ and their inversions for $d=2$ and the elements $(1,0,0)$, $(0,1,0)$, and $(0,0,1)$ and their inversions for $d=3$, then $\{X_j\}_{j=0}^M$ satisfies the parabolic H\"{o}rmander condition.
\item If $U(u) = U^{NS}(u)$ and $\mathcal{K}$ contains the elements $(1,0)$ and $(1,1)$ and their inversions for $d=2$ and the elements $(1,0,0)$, $(0,1,0)$, and $(0,0,1)$ and their inversions for $d=3$, then $\{X_j\}_{j=0}^M$ satisfies the parabolic H\"{o}rmander condition.
\end{enumerate}
\end{lemma}
\begin{proof} We will consider only the Galerkin-Navier-Stokes case, since the Stokes case is even simpler. Fix $(u,x,v)\in\Hbf_N\times\T^d\times\S^{d-1}$ and denote $\mathcal{V}(u,x,v)$ the span of the the iterated Lie brackets of $\{X_j\}_{j=0}^M$. We have for each $k\in \mathcal{K}$ and $i=1,\ldots,d-1$
\[
	[e_k\gamma_k^i,X_0] = [e_k\gamma_k^i,U^{NS}] + [e_k\gamma_k^i,V]
\]
and because of the linear dependence of the vector field $V$ on $u$, we obtain
\[
	\big[e_j\gamma^i_j,[e_k\gamma_k^i,X_0]\big] = \big[e_j\gamma^i_j,[e_k\gamma_k^i,U^{NS}]\big].
\]

We will find it useful to use the following result adapted from \cite{E2001-lg} and \cite{Romito2004-rc}.
\begin{lemma}\label{lem:mat-e}
Suppose that $K \subseteq \Z^d_0$ satisfies $K = -K$, then at each $u\in \Hbf_N$ and for each $i=1,\ldots d-1$
\[
\Span\left\{\big[e_k\gamma^i_k,[e_j\gamma_j^i,U^{NS}]\big]: j,k\in K \right\} = \Span\{e_{j+k}\gamma^i_{j+k}, e_{j-k}\gamma^i_{j-k},e_{k-j}\gamma^i_{k-j}, e_{-j-k}\gamma^i_{-j-k}: j,k\in K \}.
\]
\end{lemma}
Using the fact that $(1,0)$ and $(1,1)$ and $(1,0,0)$, $(0,1,0)$, and $(0,0,1)$ are generators for the groups $(\Z^2,  + )$ and $(\Z^3,  + )$ respectively, we can iterate Lemma \ref{lem:mat-e} for fixed $i$, taking further Lie brackets with of these new directions. Then repeating the same argument for each $i=1,\ldots, d-1$ to obtain all directions in $\Hbf_N$ and conclude that
\[
	\Hbf_N \subseteq \mathcal{V}(u,x,v).
\]
This means that in order for $\{X_j\}_{j=0}^M$ to satisfy the parabolic H\"{o}rmander condition, it suffices to show that
\[
	\Span \big\{[e_k\gamma^i_k,V]\,:\, k \in \mathcal{K},\, i \in \{1,\ldots d-1\}\big\} = T_v\S^{d-1}. 
\]
This follows from Lemma \ref{lem:v-span}.
\end{proof}

Analogously we have sufficient conditions for $(u_t,x_t,A_t)$ to satisfy the parabolic H\"{o}rmander condition. The proof is almost exactly the same as the proof of Lemma \ref{lem:Hormander-span}, with $V$ replaced with $G$. We omit the proof.
\begin{lemma}
 Let $\{X_j\}_{j=1}^M$ denote an enumeration of the vectors $\{q_ke_k\gamma^i_k\,:\, k\in \mathcal{K},\, i=1,\ldots,d-1\}$ and let $X_0$ be a vector field on $\hat{\Hbf}\times\T^d\times SL_d(\R)$ given by
\[
	X_0(u,x,A) = U(u)+ G(u,x,A),
\]
The following holds:
\begin{enumerate}
\item If $U(u) = U^S(u)$ and $\mathcal{K}$ contains the elements $(1,0),(0,1), (1,1)$ and their inversions for $d=2$ and the elements $(1,0,0)$, $(0,1,0)$,$(0,0,1)$,$(1,1,1)$ and their inversions for $d=3$, then $\{X_j\}_{j=0}^M$ satisfies the parabolic H\"{o}rmander condition.
\item If $U(u) = U^{NS}(u)$ and $\mathcal{K}$ contains the elements $(1,0)$ and $(1,1)$ and their inversions for $d=2$ and the elements $(1,0,0)$, $(0,1,0)$, and $(0,0,1)$ and their inversions for $d=3$, then $\{X_j\}_{j=0}^M$ satisfies the parabolic H\"{o}rmander condition.
\end{enumerate}
\end{lemma}

%% Application to infinite dimensional fluids
\section{Strong Feller for the Lagrangian and projective processes} \label{sec:Feller}  %\label{sec:2dnse}
%!Tex root = master.tex

In Section \ref{sec:Feller} we will prove Proposition \ref{prop:Feller}. We show the proof for the $(u_t,x_t,v_t)$ process; the $(u_t,x_t,\check{v}_t)$ process is the same. 
Note that strong Feller for $(u_t,x_t,v_t)$ implies the same for $(u_t)$ and $(u_t,x_t)$ due to the structure of the coupling. 

\subsection{The cutoff process} 
As described in Section \ref{subsubsec:feller} the main strategy involves proving gradient estimates on a suitable cut-off process $w^\rho_t$. To begin, define the following augmented system
\begin{align}
\partial_t u_t & =-B(u_t,u_t) -Au_t + Q \dot{W}_t^u \\ 
\partial_t x_t & = u_t(x_t) \\ 
\partial_t v_t & = \Pi_{v_t} \grad u_t(x_t) v_t \\
\partial_t z_t & = \dot{W}^z_t,
\end{align}
where $W^u_t$ is a cylindrical Wiener process on $\Wbf$ and $W^z_t \in \Real^{2d}$ is a finite dimensional Wiener process indepenent from $W^u_t$. 
We denote this augmented process by $w_t = (u_t,x_t,v_t,z_t)\in \Hbf \times \cM$, where $\cM = \T^d \times \S^{d-1} \times \Real^{2d}$, which satisfies the abstract SPDE
\begin{align}
\partial_t w_t = \widehat{F}(w_t) - Aw_t + \widehat{Q}\dot{W}_t, \label{eq:SDE-w}
\end{align}
where $\widehat{F}$ and $\widehat{Q}\dot{W}$ are given by
\[
	\widehat{F}(u,x,v,z) = \begin{pmatrix}-B(u,u)\\ u(x) \\ \Pi_{v}\grad u(x)v \\ 0 \end{pmatrix},\quad \widehat{Q}\dot{W} = \begin{pmatrix}
	Q\dot{W}^u\\ 0 \\ 0 \\ \dot{W}^z
	\end{pmatrix},
\]
(with extended definitions $Aw = (\nu(-\Delta) u,0,0,0)$ in $d=2$ and $Aw = (\nu(-\Delta)u + \eta \Delta^2 u,0,0,0)$ in $d=3$). 
For the remainder of this section, we will refer to the initial data of the process simply as 
\begin{align}
w_0 =: w. 
\end{align}

Our goal will be to prove strong Feller for the augmented process \eqref{eq:SDE-w}. 
As $z_t$ is completely uncoupled from $(u_t,x_t,v_t)$, by restricting the class of test functions, this implies strong Feller for the original $(u_t,x_t,v_t)$ process. 
Further, note that by restricting the class of test functions, strong Feller for the process defined with $v_t \in \S^{d-1}$ implies strong Feller for the process defined directly with $v_t \in P^{d-1}$ by relating elements in $P^{d-1}$ to representatives in $\S^{d-1}$.

To define $w^\rho$, we will couple $z_t$ to the $x_t$ and $v_t$ variables to regularize the dynamics. Specifically, as in \cite{EH01}, define a smooth, non-negative cutoff function $\chi$ satisfying 
\begin{align}
\chi(z) & = \begin{cases}
0 \quad z < 1 \\ 
1 \quad z > 2 
\end{cases} 
\end{align}
and let $\chi_\rho(x) = \chi(x/\rho)$ for $\rho > 0$. We then define a regularized drift $F_\rho(w)$ by
\begin{align}
F_\rho(u,x,v,z) = (1-\chi_{3\rho}(\norm{u}_{\Hbf})) \widehat{F}(u,x,v,z) + \chi_{\rho}(\|u\|_{\Hbf})H(v,z),   
\end{align}
where $H(v,z)$ is a bounded vector-field on $\Hbf\times\cM$ given by
\begin{equation}
H(v,z) = \begin{pmatrix} 0 \\ \sum_{j=1}^d \hat{e}_j \frac{z_j}{\left(1 + \abs{z_j}^2\right)^{1/2} } \\  \Pi_{v}\sum_{j=1}^d\hat{e}_j\frac{z_{d+j}}{\left(1 + \abs{z_{d+j}}^2 \right)^{1/2}} \\ 0\end{pmatrix}
\end{equation}
and where we are denoting $\{\hat{e}_j\}_{j=1}^d$ the canonical basis elements in $\R^d$, and we are using that for each $v\in \S^{d-1}$, $\{\Pi_v e_j\}_{j=1}^d$ span $T_v\S^{d-1}$. The cutoff/regularized process $w_t^\rho = (u_t^\rho,x_t^\rho,v_t^\rho,z_t)$ then satisfies the SPDE (replacing $\widehat{Q} \mapsto Q$ for notational simplicity), 
\begin{align}\label{eq:cut-off-eq}
\partial_t w_t^\rho = F_\rho(w_t^\rho) - A w_t^\rho + Q \dot{W}_t, 
\end{align}
It is for this process we will prove a gradient estimate on the Markov semigroup. 
As in \cite{RomitoXu11,EH01}, the purpose of the cutoff is to regularize the nonlinearity so that the flow is globally Lipschitz, which is very convenient for the Malliavin calculus and high/low frequency splitting methods employed below. However, when the nonlinearity is turned off, the hypoellipticity disappears. Recovering the hypoelliptic effect is the purpose of the additional noise coming from the coupling with $z_t$. 
In \cite{RomitoXu11,EH01}, this role is played by multiplicative white noise. This is too singular to carry out directly on the Navier-Stokes equations; in \cite{RomitoXu11} it is dealt with by further mollifying the nonlinearity. 
One can view the use of $z_t$ as providing a suitable regularization of the multiplicative white noise.  

In what follows we denote (via a slight abuse of notation) for $H^\gamma$, $L^2$, and $\Hbf$, 
\begin{align}
\norm{w_t}_{H^\gamma} & := \norm{u_t}_{H^\gamma} + \abs{z_t}.
\end{align}

We denote $T_v \cM$ the tangent space of $\cM$ at $(x,v,z)$ (initial data for $w_t$); note that the tangent space only depends on $v$. 

 We are now ready to begin the proof of Proposition \ref{prop:Feller}. 
The proof requires a number of estimates on $w_t^\rho$, its Jacobian (Frechet derivative with respect to the initial data), various approximate Jacobians and approximate inverse Jacobians, and the Malliavin derivatives thereof. These are outlined in Section \ref{sec:FundEstsCutoff} below after the main bulk of the proof. 
Finally, we emphasize that for the rest of the section, the implicit constants are \emph{always independent of $t$, $T$, $\norm{h}_{\Hbf \times T_v \cM}$, and $\norm{w_t}_{\Hbf}$ unless specifically indicated otherwise}. Moreover, we are always assuming $T \leq 1$. 

The main effort in the proof of Proposition \ref{prop:Feller} is to obtain the following derivative estimate on the cutoff process, the proof of which comprises the rest of Section \ref{sec:Feller}. 
\begin{proposition}\label{prop:GradPhi} 
There exists $a_\ast,b_\ast > 0$ such that for all $\rho$ sufficiently large, there exists a $T^* >0$ and a constant $C_\rho >0$ depending only on $\rho$ such that for all $\phi \in C^2_b(\Hbf\times\mathcal{M})$ and for $t < T^\ast$ the mapping $w \mapsto \widehat{P}_\rho^t \phi(w)$ is differentiable and for each $w \in \Hbf \times \mathcal{M}$ the derivative $D \widehat{P}_\rho^t \phi(w)$ is a bounded linear operator on $\Hbf\times T_v \cM$ and satisfies for each $h \in \Hbf\times T_v \cM$
\begin{equation}\label{eq:grad-est-trunc}
\big|D \widehat{P}_t ^\rho \phi(w)h\big| \leqc_\rho t^{-a_\ast} \left(1 +  \norm{w}_{\Hbf}^{b_\ast}\right) \norm{\phi}_{L^\infty}\|h\|_{\Hbf\times T_v \cM}.
\end{equation}
\end{proposition}
Indeed, we do not expect that such a gradient estimate \eqref{eq:grad-est-trunc} is available for $\widehat{P}_t$. None-the-less, estimate \eqref{eq:grad-est-trunc} {\em is} enough to prove the strong Feller property for $w_t$, $\widehat{P}_t$.
\begin{proof}[\textbf{Proof of Proposition \ref{prop:Feller}}] 
Let $\phi$ be a bounded, measurable observable on $\Hbf\times \cM$. 
Let $t \leq 1$ be chosen small shortly. 
Let $w^{1}, w^2 \in \Hbf \times \cM$ be such that $d(w^1,w^2) \leq 1$. 
Naturally, we estimate the non-cutoff process by approximation, 
\begin{equation}
\begin{aligned}
\abs{\widehat{P}_t \phi(w^{1}) - \widehat{P}_t \phi(w^{2})} & \leq \abs{\widehat{P}_t \phi(w^{1}) - \widehat{P}_t^{\rho} \phi(w^{1})} + \abs{\widehat{P}_t \phi(w^{2}) - \widehat{P}_t^{\rho} \phi(w^{2})} \nonumber \\ & \quad + \abs{\widehat{P}_t^\rho \phi(w^{1}) - \widehat{P}_t^{\rho} \phi(w^{2})}. \label{ineq:StrFel3e}
\end{aligned}
\end{equation}
For the first two terms in \eqref{ineq:StrFel3e}, note that 
\begin{equation}
\abs{\widehat{P}_t \phi(w^{i}) - \widehat{P}_t^{\rho} \phi(w^{i})} = \abs{\EE \phi(w_t(w^i) - \EE \phi(w^{\rho}_t(w^i))} \leq \norm{\phi}_{L^\infty} \P\left(\sup_{s \in (0,t)}\norm{w_s(w^i)}_{\Hbf} > \rho\right), 
\end{equation}
where $i=1,2$. Then, by the moment bounds in Proposition \ref{prop:WPapp}, this gives the following (with implicit constant independent of $t$), 
\begin{align}
\abs{\widehat{P}_t \phi(w^{i}) - \widehat{P}_t^{\rho} \phi(w^{i})} & \lesssim \rho^{-1} \norm{\phi}_{L^\infty}\norm{w^i}_{\Hbf}. 
\end{align}
We may now choose $\rho$ sufficiently large depending only on $\norm{\phi}_{L^\infty}$, $\norm{w^i}_{\Hbf}$, and $\eps$ such that 
\begin{align}
\abs{\widehat{P}_t \phi(w^{1}) - \widehat{P}_t \phi(w^{2})}  \leq \abs{\widehat{P}_t^\rho \phi(w^{1}) - \widehat{P}_t^{\rho} \phi(w^{2})} + 2\eps. 
\end{align}
Once we have fixed $\rho$, we may now fix $t < T_\ast$ such that \eqref{eq:grad-est-trunc} holds for the cutoff process. 
By an adaptation of [\cite{DPZ96}, Lemma 7.1.5], we see that Proposition \ref{prop:GradPhi} implies (using $d(w^1,w^2) \leq 1$), 
\begin{align}
\abs{\widehat{P}^\rho_t \phi(w^{1}) - \widehat{P}^\rho_t \phi(w^{2})} \lesssim t^{-a_\ast}\norm{\phi}_{L^\infty} (1 + \norm{w^1}^{b_\ast}_{\Hbf}) d(w^1,w^2),  \label{ineq:LocLipPr}
\end{align}
where for $w^i = (u^i,x^i,v^i,z^i) \in \Hbf\times \cM$, we denote $d(w^1,w^2)= \|u^1-u^2\|_{\Hbf} + d_{\cM}\left((x^1,v^1,z^1),(x^2,v^2,z^2)\right)$ where $d_{\cM}$ is the geodesic distance on $\cM$.
Therefore, for the third term in \eqref{ineq:StrFel3e}, we may apply \eqref{ineq:LocLipPr} and choose $d(w^1,w^2)$ sufficiently small such that 
\begin{align}
\abs{\widehat{P}_t \phi(w^{1}) - \widehat{P}_t \phi(w^{2})} < 3\eps. 
\end{align}
Hence, $\widehat{P}_t$ is strong Feller.
\end{proof} 

\subsection{Derivative estimate for cutoff process via Malliavin calculus} 
\emph{In what follows, we will drop the $\rho$ superscripts and $w_t$ will denote the solution to the cut-off equation \eqref{eq:cut-off-eq}}.

First, let us recall some basics on Malliavin calculus. For much of this section we will be dealing with random variables $X = (h,m) \in \cH\times \mathfrak{M}$, where $\cH$ is a Hilbert space and $\fM$ is a smooth finite dimensional Riemannian manifold. The Malliavin derivative $\MalD_g X$ of $X$ in direction $g = (g_t) \in L^2(\R_+,\Wbf)$ is defined by
\[
	\MalD_g X := \frac{\dee}{\dee h}X(W + hG)|_{h=0}, \quad G= \int_0^\cdot g_s\ds,
\]
when the limit exists (in the Fr\'{e}chet sense). If the above limit exists for such a random variable $X$, we say that $X$ is {\em Malliavin differentiable}. In practice $\MalD_gX$ admits a representation of the form
\begin{equation}\label{eq:D_sX-def}
\MalD_g X = \int_0^{\infty} \MalD_s X g_s\, \ds,
\end{equation}
where for almost every $s\in \R_+$, $\MalD_s X$ is a random, bounded linear operator from $\Wbf$ to $\cH\times T_m \fM$ (see \cite{Nualart06} for more details). We will commonly use the following norm of $\MalD_sX$
\[
	\|\MalD_sX\|_{\Wbf \to \mathcal{H}\times T_m\fM} := \sup_{\substack{f\in \Wbf\\ \|f\|_{\Wbf}=1}} \|\MalD_sX f\|_{\mathcal{H}\times T_m\fM}, 
\]
where $\MalD_sX f$ denotes the action of $\MalD_sX$ on $f\in \Wbf$. Formally, one can view the quantity $\MalD_s Xf$ as the limit of the directional derivatives $\MalD_g X$ when $g$ approaches a delta function at times $s$ times $f$.

We will also be taking the Malliavin derivative of various stochastic processes $(X_t) = (h_t,m_t)$ on $\mathcal{H}\times\fM$. It is a standard fact in the theory of Malliavin calculus that if $X_t$ is adapted to the filtration $\mathcal{F}_t$ generated by $W_t$, then $\MalD_s X_t= 0$ if $s \geq t$. For example, for the process $w_t$ in $\Hbf \times \cM$, we have an exact formula for $\MalD_sw_t$. Indeed, it is straightforward to show that $\MalD_g w_t$ satisfies the equation
\[
	\partial_t \MalD_g w_t = DF(w_t)\MalD_g w_t + A\MalD_g w_t + Qg_t, \quad \MalD_g w_0 = 0.
\]
Then, if one defines for $0\leq s \leq t$ the Jacobian $J_{s,t}$ (viewed as a bounded linear operator from $\Hbf\times T_{v_s}\cM$ to $\Hbf\times T_{v_t}\cM$) as the solution to the equation
\[
	\partial_t J_{s,t} = DF(w_t)J_{s,t} + AJ_{s,t}\quad J_{s,s} = \Id.
\] 
Then Duhamel's formula implies that
\[
	\MalD_g w_t = \int_0^\infty J_{s,t}Qg_s\ds,
\]
consequently, by equation \ref{eq:D_sX-def}, this implies the following formula for $\MalD_s w_t$
\[
\MalD_s w_t = \begin{cases} 
J_{s,t}Q &  s < t \\ 
0 &  s > t
\end{cases}.
\]

For real-valued random variables, the Malliavin derivative can be realized as a Fr\'{e}chet differential operator $\MalD: L^2(\Omega)\to L^2(\Omega; L^2(\R_+;\Wbf))$. The adjoint operator $\MalD^*:L^2(\Omega; L^2(\R_+;\Wbf)) \to L^2(\Omega)$ is referred to as the {\em Skorohod integral}, whose action on $g\in L^2(\Omega; L^2(\R_+;\Wbf))$ we denote by
\[
	\int_0^\infty \langle g_t,\delta W_t\rangle_{\Wbf} := \MalD^*g.
\]
The Skorohod integral can be viewed as an extension of the usual It\^{o} integral. In fact, when $g_t$ is adapted to the filtration $\mathcal{F}_t$ generated by $W_t$, then $\int_0^\infty \langle g_t,\delta W_t\rangle_{\Wbf}$ coincides with the usual It\^{o} integral $\int_0^\infty \langle g_t,\dee W_t\rangle_{\Wbf}$. Additionally, there is an analogue of the It\^{o} isometry for the Skorohod integral, which implies the following bound (see \cite{Nualart06} or \cite{DaPrato14})
\[
	\E \left(\int_0^\infty \langle g_t, \delta W_t\rangle_{\Wbf}\right)^2\leq \EE \int_0^\infty \norm{g_{t}}_{\Wbf}^2 + \E \int_0^\infty \int_0^\infty \norm{\MalD_sg_{t}}_{\Wbf\to \Wbf}^2\ds\dt.
\]

A fundamental result in the theory of Malliavin calculs is the Malliavin integration by parts formula, stated below for the process $w^t$.
\begin{proposition}\label{lem:MalIBP}
Let $\phi$ be a bounded differentiable function on $\Hbf\times\cM$ with bounded derivatives
and $g_t$ be a process satisfying 
\begin{equation}\label{eq:control-bound}
\EE \int_0^T \norm{g_t}_{\Wbf}^2 \dt + \E \int_0^T \int_0^T \norm{\MalD_sg_t}_{\Wbf\to \Wbf}^2\ds\dt < \infty, 
\end{equation}
then the following relation holds
\[
	\E \MalD_g\phi(w_t) = \E\left(\phi(w_t)\int_0^t \langle g_s , \delta W_s\rangle_{\Wbf} \right).
\]
\end{proposition}

As discussed in Section \ref{subsubsec:ErgodicProperties}, this formula can be used to obtain a gradient estimate on the Markov semigroup if for any $h$, one can obtain a control $g$ (depending on $h$) satisfying \eqref{eq:control-bound} such that for some time $T$ we have the equality $\MalD_g w_T = Dw_T h$. 
This however, does not appear to be possible to do in general. We will instead find a control $g$ which satisfies this approximately, so that for some time $T>0$ we have
\[
	\MalD_g w_T = Dw_T h + r_T,
\]
where $r_T$ is a remainder that we make small for small $T$.

Indeed most of the work of this section is to prove the following key Lemma.
\begin{lemma} \label{lem:RTbd+SkrdBd}
For all $\rho > 0$, there exists constants $a_\ast, b_\ast > 0$ such that for $T$ sufficiently small there exists a control $g = (g_{t})_{t\in[0,T]}$ satisfying
\begin{equation}\label{eq:Skoro-est}
\EE \int_0^T \norm{g_{t}}_{\Wbf}^2 \dt + \E \int_0^T \int_0^T \norm{\MalD_sg_{t}}_{\Wbf\to \Wbf}^2\ds\dt\leqc_\rho T^{-2a_\ast} (1+\norm{w}_{\Hbf})^{2b_\ast}\|h\|_{\Hbf \times T_v \cM}^2, 
\end{equation}
such that
\begin{align}
\EE\norm{r_T}_{\Hbf \times T_{v_T} \cM}^2 \lesssim_{\rho} T \norm{h}_{\Hbf \times T_v \cM}^2. \label{eq:Remain-est}
\end{align}
\end{lemma} 

Lemma \ref{lem:RTbd+SkrdBd} is indeed enough to prove Proposition \ref{prop:GradPhi}.
\begin{proof}[\textbf{Proof of Proposition \ref{prop:GradPhi}}] 
Using the control from Lemma \ref{lem:RTbd+SkrdBd}, we can now estimate the derivative of the semi-group in direction $h$ at time $2T$ for $\phi \in C^2$
\begin{equation}
\begin{aligned}
D \widehat{P}_{2T}\phi(w)h &= \E \left(D\widehat{P}_{T}\phi(w_T)Dw_T h\right) \\
&= \E \left(D \widehat{P}_{T}\phi(w_T)\MalD_{g}w_T\right) - \E \left(D \widehat{P}_{T}\phi(w_T)r_T\right),
\end{aligned}
\end{equation}
and using the Malliavin integration by parts formula (Proposition \ref{lem:MalIBP}), 
\begin{equation}\label{eq:MalIBPform-rem}
	D \widehat{P}_{2T}\phi(w)h = \E\left(\widehat{P}_T\phi(w_T)\int_0^T \langle g_t, \delta W(t)\rangle_{\Wbf}\right) - \E \left(D \widehat{P}_{T}\phi(w_T)r_T\right),
\end{equation}
where recall that the stochastic integral is interpreted as a Skorohod integral, since the control is not adapted. The Skorohod integral can be estimated by an extension of It\^{o} isometry (see e.g. \cite{Nualart06,HM11} and the references therein) and \eqref{eq:Skoro-est}, giving 
\begin{equation}\label{eq:skoro-ito-est}
\begin{aligned}
	\E \left(\int_0^T \langle g_t, \delta W_t\rangle_{\Wbf} \dt\right)^2 &\leq \EE \int_0^T \norm{g_{t}}_{\Wbf}^2 \dt + \E \int_0^T \int_0^T \norm{\MalD_sg_{t}}_{\Wbf\to \Wbf}^2\ds\dt\\
	&\leqc_\rho T^{-2a_\ast} (1+\norm{w}_{\Hbf})^{2b_\ast}\|h\|_{\Hbf \times T_v \cM}^2.
\end{aligned}
\end{equation}

To finish the proof, introduce the following semi-norm $\|\cdot\|_{a_\ast,b_\ast,T_*}$ on $C([0,T_*];C^1(\Hbf \times \cM))$, for $a_\ast,b_\ast >1$ and $1 \geq T_* >0$ by
\[ 
	\|f\|_{a_\ast,b_\ast,T_*} = \sup_{\substack{t\in[0,T_*]\\ w \in \Hbf \times \cM \\ h \in \Hbf \times T_v \cM,\, h \neq 0}} \frac{t^{a_\ast}|Df_t(w)h|}{\|h\|_{\Hbf \times T_v \cM}(1+\norm{w}_{\Hbf})^{b_\ast}}.
\]
Then it follows from \eqref{eq:Remain-est} and \eqref{eq:skoro-ito-est} that for $2T < T_*$, 
\begin{equation}
\begin{aligned}
	|DP_{2T}\phi(w)h| &\leqc \|\phi\|_{L^\infty}T^{-a_\ast} (1+\norm{w}_{\Hbf}^{b_\ast})\|h\|_{\Hbf \times T_v \cM}\\
	&\quad + \|P\phi\|_{a,b,T_*}T^{-a_\ast}\sqrt{\E(1+\norm{w_T}_{\Hbf})^{2b_\ast}}\sqrt{\E\|r_T\|_{\Hbf \times T_{v_T} \cM}^2}\\
&\leqc  \|\phi\|_{L^\infty}T^{-a_\ast} (1+\norm{w}_{\Hbf})^{b_\ast}\|h\|_{\Hbf\times T_v \cM}\\
	&\quad + \|P\phi\|_{a_\ast,b_\ast,T_*}T^{-a_\ast+\frac{1}{2}}(1+\norm{w}_{\Hbf})^{b_\ast}\|h\|_{\Hbf \times T_v \cM}
\end{aligned}
\end{equation}
and therefore
\[
	\|P\phi\|_{a_\ast,b_\ast,T_*} \leqc \|\phi\|_{L^\infty} + T_*^{\frac{1}{2}} \|P\phi\|_{a_\ast,b_\ast,T_*},
\]
then by taking $T_*$ small enough we obtain
\[
	\|P\phi\|_{a_\ast,b_\ast,T_*} \leqc \|\phi\|_{L^\infty}. 
\]
This is the a priori estimate stated in \eqref{eq:grad-est-trunc}. 
\end{proof} 

\subsection{Construction of control and estimates of remainder} 
 The rest of the section is dedicated to proving Lemma \ref{lem:RTbd+SkrdBd}. First, we implement a splitting into high and low frequencies similar to that of \cite{EH01,RomitoXu11}. This will allow us to build a control that works differently on the high and low frequencies. To this, denote the set $K_L\subseteq \Z^d_0$ of low modes by 
\[\
K_L = \{ k \in \Z^d_0; \quad|k|_\infty \leq L \},
\]
where $L$ is as in Assumption \ref{a:Highs}. Let $\Pi_L: \Hbf \to \Hbf$ denote the corresponding orthogonal projection onto the ``low modes'' belonging to $K_L$ and $\Pi_H = I - \Pi_L$ be the complementary projection onto the ``high modes'' belonging to $\Z^d_0 \backslash K_L$. Let $\Hbf_L$ and $\Hbf_H$ denote the ranges of $\Pi_L$ and $\Pi_H$ respectively so that we have the orthogonal decomposition
\[
	\Hbf = \Hbf_L\oplus\Hbf_H.
\]
Given $w = (u,x,v,z) \in \Hbf \times \mathcal{M}$, we will extend the definition of $\Pi_L$ and $\Pi_H$ to $\Hbf\times\cM$ so that $\cM$ is included with the low modes by
\[
w^L = \Pi_L w = (u^L,x,v,z) \quad \text{and}\quad w^H = \Pi_H w= u^H.
\]
Naturally this defines low and high processes $w^L_t$ and $w^H_t$, which satisfy (note of course they are coupled)
\begin{align}
\partial_t w_{t}^L & =  F_L(w_t) - A_Lw_t^L + Q_L\dot{W}^L_t \\ 
\partial_t w_t^H & = F_H(w_t) - A_Hw_t^H+ Q_H\dot{W}^H_t,
\end{align}
where $F_L(w) = \Pi_L F(w)$, $F_H(w) = \Pi_H F(w)$, $A_Hw = \Pi_H Aw$, $Q_L = \Pi_LQ$ and $Q_H = \Pi_H Q$. We also define the finite dimensional matrix $U^L_{s,t}$ which we view a linear operator from $\Hbf_L\times T_{v_s}\cM$ to $\Hbf_L \times T_{v_t}\cM$ as well as the bounded linear operator $U^H_{s,t}$ from $\Hbf_H$ to $\Hbf_H$ by
\[
	\partial_t U^L_{s,t} = - A_L U^L_{s,t} + D_LF_L(w_t)U^L_{s,t}, \quad U^L_{s,s} = \Id,
\]
and for $0\leq s \leq t$
\[
	\partial_t U^H_{s,t} = - A_H U^H_{s,t} + D_H F_H (w_t)U^H_{s,t}, \quad U^H_{s,s} = \Id.
\]
Both $U^L_{s,t}$ and $U^H_{s,t}$ serve as approximations for the full Jacobian $J_{s,t}$ of the flow $w \mapsto w_t$ projected onto the low and high-modes when $t$ is small. We see that $U_{s,t}^L$ is an invertible operator: denote it's inverse by 
\[
V^L_{s,t} = (U^L_{s,t})^{-1}.
\]
When $s=0$, we write $U^L_t = U_{0,t}^L$ and $V^L_t = V^L_{0,t}$. Usig the fact that $U^L_t$ is invertible we can write $U^L_{s,t} = U_{t}^L V^L_s$. 
\begin{definition}
Define the {\em partial Malliavin matrix} $\cC_t^L :\Hbf_L\times T_v\cM\to \Hbf_L\times T_v\cM$ by
\[
\cC_t^L := \int_0^t V^L_s Q_L(V^L_s Q_L)^\top\ds.
\] 
\end{definition}
\begin{remark} 
$\cC_t^L$ is the analogue of the reduced Malliavin matrix, introduced by Norris \cite{Norris86}, in order to simplify Malliavin's proof of H\"{o}rmander's theorem. The name {\em partial Malliavin matrix} comes from \cite{EH01}, and indicates that it is a finite dimensional Mallaivin matrix associate to the low modes.
\end{remark}

One of main results of Section \ref{sec:Feller} is the non-degeneracy of $C_t^L$, which allows us to build the low frequencies part of the control $g_t$. 
That is, we have the following; the proof is involved and is carried out in Section \ref{subsubsec:mal-nondeg} below. 
\begin{lemma}\label{lem:Malmatrix-bound}
The matrix $\cC_T^L$ is almost surely invertible on $\Hbf_L\times T_v\cM$. Furthermore, there exists constants $a,b$ such that for all $p\geq 1$
\[
	\E|(\cC_T^L)^{-1}|^p \leqc_{\rho,p} T^{-ap}(1+|z|)^{bp}.
\]
\end{lemma}
 
Using Lemma \ref{lem:Malmatrix-bound}, we can now construct the control. 
Specifically, fix an $h\in \Hbf\times T_v\cM$, a $T\in (0,1)$, a frequency cut-off $N$ chosen as $N := T^{-2a}(1+|z|)^{2b}$ ($a$ and $b$ as in Lemma \ref{lem:Malmatrix-bound}) and define $t\mapsto g_{t} \in \Wbf$ by
\begin{align}
& g^L_{t} = (V^L_t Q_L)^{\top}(\cC_T^L)^{-1}V^L_T D w^L_T h \label{def:vL} \\ 
& g^H_{t} = - Q_H^{-1} \Pi_{\leq N} D_L F_H(w_t)\zeta_t + 2T^{-1}Q_H^{-1}U^H_{0,t}h_H\1_{[T/4,3T/4]}(t), \label{def:vH}
\end{align}
 where $\Pi_{\leq N}$ is a projection onto frequencies less than $N$ and $(\zeta_t)$ is a process belonging for each $t\in[0,T]$ to $\Hbf_L\times \T_{v_t}\cM$ and solving the following system
\begin{equation} \label{def:xyctrl}
\begin{aligned}
	&\dot{\zeta}_t = - A_L \zeta_t + D_L F_L(w_t)\zeta_t + Q_Lg^L_t + D_H F_L(w_t) \xi_t,\\
	&\dot{\xi}_t = - A_H \xi_t + D_HF_H(w_t) \xi_t + \Pi_{> N}D_L F_H(w_t)\zeta_t + 2T^{-1}U^H_{0,t}h_H\1_{[T/4,3T/4]}(t),
	\end{aligned}
\end{equation}
with $\xi_0 = 0$ and  $\zeta_0 = 0$. If one assumes that a solution to \eqref{def:xyctrl} exists and is unique (this is proved in Lemma \ref{lem:Ectrl} below), then we find that the choice of control is made specifically so that the remainder $r_T$ assumes a nice form.
\emph{In what follows the implicit constant is always independent of $N$ unless otherwise indicated}.  

\begin{lemma}
Assume that $g$ is defined as above and that exists a unique solution $(\zeta_t,\xi_t)$ to \eqref{def:xyctrl} in the space $L^2(\Omega;L^\infty([0,T];\Hbf\times T_{v_t}\cM))$, then the remainder $r_T = \MalD_g w_T - Dw_Th$ satisfies
\begin{align}
& r_T^L = \int_0^T U^L_{t,T}D_H F_L(w_t)\xi_t\,\dt \label{def:RTL} \\ 
& r_T^H = \int_0^T U^H_{t,T} \Pi_{>N}D_LF_H(w_t)\zeta_t\,\dt - \int_0^T U^H_{t,T}D_L F_H(w_t)D_H w^L_t h_H\,\dt - D_L w^H_T h_L.  \label{def:RTH}
\end{align}
\end{lemma}
\begin{proof}	
Using \eqref{def:xyctrl}, we obtain the following formulas for the Malliavin derivatives at time $T$:
\begin{align}
&	\MalD_g w^L_T = D w^L_T h +  U^L_T\int_0^T V^L_tD_H F_L(w_t)\xi_t\, \dt \\ 
&	\MalD_g w^H_T = U^H_{0,T}h_H + \int_0^T U^H_{t,T} \Pi_{ > N}D_LF_H(w_t)\zeta_t\,\dt. 
\end{align}
Note that $\MalD_g w^L_T$ is equal to $D w^L_T h$ plus remainders, while $\MalD_g w^H_T$ is a perturbation of $U^H_{0,T}h_H$, that is, 
\[
	Dw^H_t h_H = U^H_{0,t}h_H + \int_0^t U^H_{s,t}D_L F_H(w_s)D_H w^L_s h_H\,\ds.
\]
Using this relation, we now write
\begin{equation}\label{eq:mal-der-remain}
\begin{aligned}
	&\MalD_g w^L_T = D w^L_T h + r^L_T\\
    &\MalD_g w^H_T = D w^H_T h + r^H_T 
\end{aligned}
\end{equation}
where $r^L_T$ and $r^H_T$ are given by \eqref{def:RTL} and \eqref{def:RTH}.
\end{proof}
Next, we construct a unique solution to \eqref{def:xyctrl} and provide the necessary quantitative estimates. These in turn will imply the existence of a suitable control $g_{t}$. 
\begin{lemma} \label{lem:Ectrl}
For all $T > 0$ sufficiently small (depending only on $\rho$), and all $p \geq 2$, there exists a unique solution $\eta_t = (\zeta_t,\xi_t)\in \Hbf\times T_{v_t}\cM$ on $[0,T]$ to the system \eqref{def:xyctrl} satisfying 
\begin{align*}
\left(\E\sup_{t\in[0,T]}\norm{\eta_t}_{\Hbf\times T_{v_t}\cM}^p\right)^{1/p} + \left(\E \sup_{s,t\in [0,T]}\norm{\MalD_s\eta_t}_{\Wbf\to \Hbf\times T_{v_t}\cM}^p\right)^{1/p} \lesssim  T^{-2a}( 1+  |z| )^{2b} \norm{h}_{\Hbf\times T_v\cM}.
\end{align*}
Note that $\eta_t$ is not adapted to the filtration $(\mathcal{F}_t)$.
\end{lemma}

\begin{proof}Formally we may re-write a solution to \eqref{def:xyctrl} as
\begin{align}\label{eq:zeta-fix}
\zeta_t & = \int_0^t U^L_{s,t}Q_L g^L_{s} ds + \int_0^t U_{s,t}^L D_H F_L(w_s) \xi_s ds \\ 
\label{eq:xi-fix}
\xi_t & = \frac{2}{T}\abs{[0,t] \cap [\tfrac{T}{4},\tfrac{3T}{4}]}U^H_{0,t} h_H   +  \int_0^t U^H_{s,t} \Pi_{> N}D_L F_H(w_s)\zeta_s ds. 
\end{align}
The Lemma is proved via a fixed point for the pair $\eta = \{(\zeta_t,\xi_t) , t\in [0,T]\}$ in the Banach space $\mathbb{X}_T$ defined by the following norm
\begin{align*}
\norm{\eta}_{\mathbb{X}_T} & :=  \left(\E\sup_{t\in[0,T]}\norm{\eta_t}_{\Hbf\times T_{v_t}\cM}^p\right)^{1/p} + \left(\E \sup_{s,t\in [0,T]}\norm{\MalD_s\eta_t}_{\Wbf\to \Hbf\times T_{v_t}\cM}^p\right)^{1/p}.
\end{align*}
Note that equations \eqref{eq:zeta-fix} and \eqref{eq:xi-fix} are linear and can be written more compactly on $\mathbb{X}_T$ as
\[
	\eta  = L_T\eta + F_T
\]
where $L_T$ and $F_T$ are given by
\[
	(L_T\eta)_t = \begin{pmatrix}\int_0^t U_{s,t}^L D_H F_L(w_s) \xi_s ds\\\int_0^t U^H_{s,t} \Pi_{> N}D_L F_H(w_s)\zeta_s ds\end{pmatrix},\quad (F_T)_t = \begin{pmatrix}\int_0^t U^L_{s,t}Q_L g^L_{s} ds \\ \frac{2}{T}\abs{[0,t] \cap [\tfrac{T}{4},\tfrac{3T}{4}]}U^H_{0,t} h_H \end{pmatrix}.
\]
Our goal will be to estimate $L_T\eta$ and $F_T$ in $\mathbb{X}_T$. Specifically, we will show that
\begin{align}\label{eq:L_T-est}
	&\|L_T\eta\|_{\mathbb{X}_T} \leqc_\rho T^{\frac{1}{2}}\|\eta\|_{\mathbb{X}_T}\\\label{eq:F_T-est}
	&\|F_T\|_{\mathbb{X}_T} \leqc_\rho T^{-2a}(1+|z|)^{2b}\|h\|_{\Hbf\times T_{v}\cM}.
\end{align}
This implies that for small enough $T$ (depending only on $\rho$), the mapping $\eta \mapsto L_T\eta + F_T$ is a contraction and maps the ball $\mathbb{B}_T = \{\eta \in \mathbb{X}_T : \|\eta\|_{\mathbb{X}_T} \leq 2\|F_T\|_{\mathbb{X}_T}\}$ into itself. By the contraction mapping theorem this implies the existence of a unique solution to $\eta = L_T\eta + F_T$ satisfying
\[
	\|\eta\|_{\mathbb{X}_T} \leq 2\|F_T\|_{\mathbb{X}_T} \leqc_\rho T^{-2a}(1 + |z|)^{2b}\|h\|_{\Hbf\times T_v\cM}. 
\]

To estimate $L_T\eta$ and $F_T$ in $\mathbb{X}_T$ we need to compute the Malliavin derivatives. We find for each $f\in \Wbf$
\[
(\MalD_s F_Tf)_t = \begin{pmatrix}\int_0^t [\MalD_s U_{r,t}^Lf]Q_L g^L_{r}\dr+ \int_0^t U_{r,t}^L Q_L\MalD_s g^L_{r}f\dr\\ [\MalD_s U^H_{0,t}f]h_H \frac{2}{T}\abs{[0,t] \cap [\frac{T}{4},\frac{3T}{4}]}\end{pmatrix}.
\]
and for each $\eta \in \mathbb{X}_T$ using the chain rule
\[
	\MalD_s (L_T\eta)f = [\MalD_s L_Tf]\eta + L_T \MalD_s\eta f,
\]
where
\[
	([\MalD_s L_Tf]\eta)_t = \begin{pmatrix}\int_0^t [\MalD_s U_{r,t}^Lf]D_H F_L(w_r)\xi_r\dr + \int_0^t U_{r,t}^L Q_LD^2F_L[\xi_r,J_{s,r}Qf]\dr\\
											 \int_0^t [\MalD_s U^H_{r,t}f] \Pi_{> N}D_L F_H(w_r)\zeta_r \dr +\int_0^t U^H_{r,t}\Pi_{>N}D^2 F_H[\zeta_r, J_{s,r}Qf]
	\end{pmatrix}
\]

We observe by Lemma \ref{lem:Malmatrix-bound}, Lemma \ref{lem:U,V-bounds}, and Lemma \ref{lem:JstBds}, that
\begin{align}
\EE  \sup_{0 < s \leq T} \abs{g^L_{s}}^p \lesssim  T^{-ap}(1+|z|)^{bp} \norm{h}_{\Hbf\times T_v\cM}. 
\end{align}
and by the product rule, Lemmas \ref{lem:U,V-bounds}, \ref{lem:MalliavinForAll}, and \ref{lem:MDbds}, there holds
\begin{align}
\E\sup_{s,t\in[0,T]} \|\MalD_s g^L_t\|_{\Wbf \to \Wbf_L}^p \leqc T^{-(2a +1)p}\left(1 +  |z| \right)^{2bp}\|h\|_{\Hbf\times T_v\cM}^p. 
\end{align}
Using the bounds and bounds on $U^L_{r,t},U^H_{r,t},\MalD_s U^L_{r,t}$ and $\MalD_s U^H_{r,t}$, in Lemmas \ref{lem:U,V-bounds} and \ref{lem:MalliavinForAll}, we moreover have
\[
\begin{aligned}
	\|F_T\|_{\mathbb{X}_T} &\leqc_\rho T \|g^L\|_{\mathbb{X}_T} + \|h_H\|_{\Hbf_H} \leqc T^{-2a}\left(1 +  |z| \right)^{2b}\|h\|_{\Hbf\times T_v\cM}. 
\end{aligned}
\]

To estimate $L_T\eta$ we use the bounds on $U^L_{r,t},U^H_{r,t}$ (from Lemma \ref{lem:U,V-bounds}) to obtain the almost sure bounds
\begin{equation}\label{eq:L_T-est-as}
\begin{aligned}
	\sup_{t\in [0,T]}\|(L_T\eta)_t\|_{\Hbf\times T_{v_t}\cM }&\leqc_\rho T\sup_{t\in [0,T]}\|\xi_t\|_{\Hbf_H} +  T^{\frac{1}{2}}\sup_{t\in [0,T]}|\zeta_t|\\
	 &\leqc_\rho T^{\frac{1}{2}}\sup_{t\in[0,T]}\|\eta_t\|_{\Hbf\times T_{v_t}\cM}. 
\end{aligned} 
\end{equation}
Additionally, using bounds on $J_{s,t},\MalD_s U^L_{r,t}$ and $\MalD_s U^H_{r,t}$ (from Lemmas \ref{lem:JstBds} and \ref{lem:MalliavinForAll}) we also find
\[
	\sup_{s,t\in [0,T]}\|([\MalD_sL_T]\eta)_t\|_{\Wbf\to \Hbf\times T_{v_t}\cM} \leqc_\rho T^{\frac{1}{2}}\sup_{t\in[0,T]}\|\eta_t\|_{\Hbf\times T_{v_t}\cM},
\]
and therefore by  estimate \eqref{eq:L_T-est-as} applied to $\MalD_s\eta$ instead of $\eta$, we find
\begin{equation}\label{L_T-MalD-est}
\begin{aligned}
	\sup_{s,t\in [0,T]}\|\MalD_s (L_T\eta)_t\|_{\Wbf\to \Hbf\times T_{v_t}\cM}\\
	&\leqc T^{\frac{1}{2}}\left(\sup_{t\in[0,T]}\|\eta_t\|_{\Hbf\times T_{v_t}\cM} + \sup_{s,t\in[0,T]}\|\MalD_s\eta_t\|_{\Wbf \to \Hbf\times T_{v_t}\cM}\right).
\end{aligned}
\end{equation}
Putting \eqref{eq:L_T-est-as} and \eqref{L_T-MalD-est} together and taking the $L^p(\Omega)$ norm gives estimate \eqref{eq:L_T-est}.
\end{proof}

We are now ready to prove Lemma \ref{lem:RTbd+SkrdBd}. 
\begin{proof}[{\bf Proof of Lemma \ref{lem:RTbd+SkrdBd}}]

First we prove the estimate \eqref{eq:Remain-est} on the remainder $r_T$. It is here where we will need to set the choice of $N$ depending on $T$ and $|z|$. To begin, we note that from equation \eqref{eq:xi-fix}, using the cut-off $\Pi_{>N}$, we obtain the following improved estimate on $\xi_t$
\begin{align*}
\left(\E\sup_{t\in[0,T]}\norm{\xi_t}_{\Wbf_H}^2\right)^{1/2} \lesssim  \norm{h}_{\Hbf\times T_{v}\cM} + N^{1-\sigma}T^{-2a}(1 + |z|)^{2b} \norm{h}_{\Hbf\times T_{v}\cM}. 
\end{align*}
Therefore, since $\sigma-1 >1$ and the definition $N = T^{-2a}(1+|z|)^{2b}$, we obtain the $T$ independent bound
\[
	\left(\E\sup_{t\in[0,T]}\norm{\xi_t}_{\Wbf_H}^2\right)^{1/2} \lesssim \norm{h}_{\Hbf\times T_v\cM}.
\]
Recall the definition of the remainders \eqref{def:RTL} and \eqref{def:RTH}. We estimate $r^L$ first. We find (noting $\abs{D_H F_L(w_t)\xi_t}\lesssim \chi_{3\rho}(\norm{u_t}_{\Hbf})\norm{u_t}_{H^{\gamma}} \norm{\xi_t}_{\Wbf_H}$ for any $\gamma > \frac{d}{2} + 1$ due to the frequency projection), 
\begin{align}
\abs{r_T^L} = \abs{U^L_T\int_0^T V^L_tD_H F_L(w_t)\xi_t\, \dt} \lesssim_{\rho} T \sup_{t\in[0,T]}\left(\abs{U^L_t}\abs{V^L_t}\norm{\xi_t}_{\Wbf_H}\right),
\end{align}
and therefore using almost sure bounds on $U^L_t$ and $V^L_t$ from Lemma \ref{lem:U,V-bounds}, 
\begin{align}
\EE \abs{r_T^L}^2 \lesssim_{\rho} T^2 \E\sup_{t\in[0,T]}\|\xi_t\|_{\Wbf_H}^2 \lesssim T^2 \norm{h}_{\Hbf\times T_v\cM}^2.
\end{align}
Hence, $r_T^L$ satisfies the estimate required for \eqref{eq:Remain-est}. 

Turn next to estimating $r^H_t$. We again use the the frequency truncation $\Pi_{\leq N}$ and the choice $N = T^{-2a}(1+|z|)^{2b}$ to find
\begin{align}
\norm{r_T^H}_{\Hbf_H} &\lesssim \int_0^T \frac{1}{ N (T-t)^{1/2}} \abs{\zeta_t} dt + \int_0^T \frac{1}{(T-t)^{1/2}} \abs{D_H w^L_t h_H} dt + \norm{D_L w^H_T h_L}_{\Hbf_H}. \\ 
&\lesssim_\rho T^{\frac{1}{2}+2a}(1+|z|)^{2b} \sup_{t\in[0,T]}|\zeta_t| + T^{\frac{1}{2}} \sup_{t\in[0,T]}|D_H w^L_t h_H| + \norm{D_L w^H_T h_L}_{\Hbf_H}. 
\end{align}
Using that Lemma \ref{lem:Ectrl} gives
\[
T^{4a}(1+|z|)^{4b}\E\sup_{t\in[0,T]}|\zeta_t|^2 \leqc 1,
\]
along with Lemma \ref{lem:CrossJacobian} for $D_Hw^L$ and $D_Lw^H$, we conclude that $r^H_T$ satisfies the estimate required for \eqref{eq:Remain-est}.

Next we show the estimate \eqref{eq:Skoro-est} on the control $g$. Recall from the proof of Lemma \ref{lem:Ectrl} that we can use the bounds on the partial Malliavin matrix $\cC^L_t$ to get the following estimate on $g^L$
\begin{align}
\EE  \sup_{0 < t \leq T} \abs{g^L_{t}}^2 + \E\sup_{s,t\in[0,T]} \|\MalD_s g^L_t\|_{\Wbf \to \Wbf_L}^2\lesssim  T^{-4a}(1+|z|)^{4b} \norm{h}_{\Hbf\times T_{v}\cM}^2. 
\end{align}
It remains to estimate $g^H$. Recall the following formula's for $g^H_t$ and $\MalD_sg^H_t$ 
\begin{equation}
	g^H_{t} = - Q_H^{-1} \Pi_{\leq N} D_L F_H(w_t)\zeta_t + 2T^{-1}Q_H^{-1}U^H_{0,t}h_H\1_{[T/4,3T/4]}(t)
\end{equation}
\begin{equation}
\begin{aligned}
	&\MalD_s g^H_tf = Q_H^{-1} \Pi_{\leq N} D^2F_H(w_t)[\zeta_t,J_{s,t}Qf] + Q_H^{-1} \Pi_{\leq N} D_L F_H(w_t)\MalD_s\zeta_tf\\
	&\hspace{1in} + 2T^{-1}Q_H^{-1}\MalD_sU^H_{0,t}fh_H\1_{[T/4,3T/4]}(t). 
\end{aligned}
\end{equation}
Using the cut-off $\Pi_{\leq N}$ and the \emph{lower bound} in Assumption \ref{a:Highs}, 
\[
\begin{aligned}
	\E\int_0^T\|g^H_t\|_{\Wbf_H}^2\dt &\leqc \E\int_0^T\|Q_H^{-1} \Pi_{\leq N} D_LF_H(w_t)\zeta_t\|_{\Wbf_H}^2 + T^{-1}\E\int_{T/4}^{3T/4}\|Q_H^{-1}U^H_{0,t}h_H\|_{\Wbf_H}^2\dt\\
	&\leqc \E\int_0^T\|\Pi_{\leq N}D_LF_H(w_t)\zeta_t\|_{H^\alpha}^2\dt + T^{-1}\E\int_{T/4}^{3T/4}\|U^H_{0,t}h_H\|_{H^\alpha}^2\dt\\
	&\leqc_\rho  N^4T\sup_{t\in[0,T]}|\zeta_t|^2 + T^{-2}(1 + \norm{w}_{\Hbf}^2) \|h_H\|_{\Hbf_H}^2,
\end{aligned}
\]
where in the last line we used \eqref{ineq:LtHgammaJst} on $U^H_{0,t}$ with $\gamma = \alpha - 1$. This is where we use the requirement $\sigma \in (\sigma-2(d-1),\sigma-\frac{d}{2})$. 
A similar calculation for $\MalD_sg^H_t$ yields
and
\[
\begin{aligned}
	\E\int_0^T\int_0^T\|\MalD_sg^H_t\|_{\Wbf\to \Wbf_H}^2\dt\ds &\leqc_\rho  N^4T^2\E\left(\sup_{t\in[0,T]}|\zeta_t| +\sup_{s,t\in[0,T]}\|\MalD_s \zeta_t\|_{\Wbf \to \Hbf_L\times T_{v_t}\cM}^2\right)\\
	&\hspace{.5in} + T^{-2}(1 + \norm{w}_{\Hbf}^2) \|h_H\|_{\Hbf_H}^2
\end{aligned}
\]
Using the estimate on $\xi_t$ from Lemma \ref{lem:Ectrl} and our choice of $N = T^{-2a}(1+|z|)^{2b}$ we find
\[
	\begin{aligned}
	\E\int_0^T\|g^H_t\|_{\Wbf_H}^2\dt + \E\int_0^T\int_0^T\|\MalD_sg^H_t\|_{\Wbf\to \Wbf_H}^2\dt\ds \leqc_\rho T^{-8a}(1+\norm{w}_{\Hbf})^{8b+2}\|h\|_{\Hbf\times T_{v}\cM}^2.
	\end{aligned}
\]
Therefore we have the desired estimate \ref{eq:Skoro-est} on $g_{t}$. 
\end{proof}

\subsection{Non-degeneracy of the partial Malliavin matrix}\label{subsubsec:mal-nondeg}

For simplicitly of presentation and brevity, we will only detail the proof in the case of non-degenerate noise on the Navier-Stokes equations (i.e. $L=1$), that is
\begin{align}
\abs{q_k} \approx \abs{k}^{-\alpha} \quad \forall k \in \Z_0^d.  
\end{align}
Once one has the hypoellipticity deduced in Section \ref{sec:galerkin}, the adaptation to the weaker Assumption \ref{a:Highs} is a well-understood extension using methods from previous works \cite{E2001-lg,Romito2004-rc,EH01,RomitoXu11,HairerNotes11}. This is discussed in more detail in Remark \ref{rmk:NorrisMal} below.

Define the set
\[
	\K = (\Z^d_0\times \{1,\ldots,d-1\}) \cup\{1,2,\ldots,2d\}.
\]
Note that each element $m \in \K$ is either a pair $(k,i)\in \Z^d_0\times\{1,\ldots,d-1\}$ or an integer $j\in \{1,\ldots,2d\}$. We will also denote the set $\K_L$ in a similar way with $\Z^d_0$ replaced by $K_L$and define $\K_H = \K \backslash \K_L$. The operator $\widehat{Q}$ on $\Wbf\times \R^{2d}$ gives rise to a family of vector fields $\{Q^m\}_{m\in \K}$ on $\Hbf\times \cM$ defined by 
\[
	Q^m = \begin{cases} q_ke_k\gamma_k^i &\quad\text{if } m = (k,i)\in \Z^d_0\times\{1,\ldots,d-1\}\\ \hat{e}^z_j &\quad\text{if } m = j\in \{1,\ldots,2d\}\end{cases}
\]
where we are denoting $\{\hat{e}^z_j\}$ the canonical basis on $\R^{2d}$. The pivotal lemma is the following non-degeneracy of the partial Malliavin matrix  $\cC_{t}^L$.
\begin{lemma} \label{lem:pMalNonDeg}
For all $p \geq 1$, $t < 1$, $\ep >0$ and $w\in \Hbf\times \cM$, there exists constant $a,b>1$ such that
\begin{equation}
\sup_{\substack{h \in \Hbf_L\times T_v\cM\\|h| = 1}}\P \bigg(\sum_{m\in\K_L}\int_0^t \big\langle V^L_s Q^m,h\big\rangle_{L}^2\ds < \eps \bigg) \lesssim_{p,\rho} t^{-ap}( 1 +|z| )^{bp}\eps^{p},
\end{equation}
where the constant is independent of $\ep$ and the initial data.
\end{lemma}

Above $\langle \cdot,\cdot \rangle_L$ denotes the Riemannian metric on $\Hbf\times \cM$. We omit the dependence on $v\in\S^{d-1}$.

Note that $\sum_{m\in\K_L} \int_0^t \left\langle V^L_s Q^m(w_s),h\right\rangle_{L}^2\ds = \langle h, \cC^L_{t} h\rangle_L$, so that Lemma \ref{lem:pMalNonDeg} is really about non-degeneracy of $\cC^L_{t}$. It is a standard fact in the theory of Malliavin calculus that Lemma \ref{lem:Malmatrix-bound} is sufficient to deduce the moment bounds on $(\cC_t^L)^{-1}$ stated in Lemma \ref{lem:pMalNonDeg}. 

To begin, we will need the following Lemma that relates time-derivatives of certain quantities to appropriate Lie brackets.
\begin{proposition}
\label{prop:commutator-Lemma}
Let $G$ be a bounded vector field on $\Hbf\times\cM$ whose range belongs to $\Hbf_L\times T\cM$ and with two bounded derivatives, then the following formula holds
\begin{equation}\label{eq:commutator-eq}
\begin{aligned}
	V_t^LG(w_t) &= G(w) + \int_0^tV_L^s\left([F,G]_L(w_s)- [A,G]_L(w_s)\right)\ds\\
	 &\hspace{.5in}+ \frac{1}{2}\sum_{m\in\mathbb{K}}\int_0^t V^L_sD^2 G(w_s)[Q^m,Q^m]\ds + \int_0^t V^L_s DG(w_s)Q\dee W_s
\end{aligned}
\end{equation}
and for and two differentiable vector fields $F,G$ over $\Hbf\times\cM$, we denote
\[
[F,G]_L \equiv \Pi_L [F,G](w) =  (DG_L)(w) F(w) - (DF_L)(w) G(w)
\]
and 
\[
[A,G]_L(w) \equiv D_L G_L(w) A_Lw - A_L G(w).
\]
\end{proposition}
\begin{proof}
The proof follows from It\^{o}`s formula on $G(w_t)$ and the fact that $V^L_t$ satisfies
\[
	V^L_t = \Id - \int_0^t V^L_s\left(D_LF_L(w_s) - A_L\right)\ds.
\]
\end{proof}

\begin{remark}
Note that since we assume that $\Range{G}(w) \subseteq \Hbf_L\times T_v\cM$ and the vector fields $\{Q_i\}_{i\in \K}$ have the property that $\Range{Q_i} \subseteq \Hbf_L\times T_v\cM$ if $i\in \K_L$ and $\Range{Q^m} \subseteq \Hbf_H$ if $im\in \K_H$ then the sum above converges by the fact that the noise is of Hilbert-Schmidt type and therefore the sum over high frequencies can be bounded
\[
\sum_{m\in \K_H} \big\|D^2 G(w)[Q^m,Q^m] \big\|_{\Hbf_L\times T_v\cM} \leq \|D_H^2 G(w)\|_{\Hbf_H\tensor\Hbf_H\to \Hbf_L\times T_v\cM}\sum_{k\in K_H} q_k^2 < \infty.
\]
\end{remark}

For conveneince, we define the following operator $\Lambda_L$ that maps smooth vector fields on $\Hbf\times \cM$ to smooth vector fields on $\Hbf\times \cM$ with range in $\Hbf_L\times T\cM$, defined by
\[
\Lambda_LG := [F,G]_L- [A,G]_L + \frac{1}{2}\sum_{m\in\mathbb{K}}D^2G[Q^m,Q^m].
\]
\begin{lemma}\label{lem:commutator-bounds} The following estimates hold for each $m\in\K_L$
\begin{equation}
|\Lambda_LQ^m|(w) \leqc_\rho 1,\quad |\Lambda_L^2 Q^m|(w) \leqc_\rho 1, \quad \sum_{j\in \K}|[Q^j,\Lambda_L Q^m]_L|^2(w) \leqc_\rho 1.
\end{equation}
\end{lemma}
\begin{proof}
The proof follows from the fact that below the cut-off  $\|u\|_{\Hbf} \leq 6\rho$, we can bound 
\[
 |[F,Q^m]_L| + |[A,Q^m]_L| \leqc_\rho (1+ \|u\|_{\Hbf}^2) \leqc_\rho 1.
\]
When $\|u\|_{\Hbf} > 6\rho$, the Navier-Stokes nonlinearity is turned off and the above non-linear term doesn't contribute, so we can just use $|[A,e_k\gamma_k^i]_L| \leqc 1$. There are also terms which are nonlinear in $z$, however they are bounded and have bounded derivatives, so that $|[F,\hat{e}^z_j]_L| \leqc 1$. The only other subtlety involves ensuring that the infinite sum in $m\in\K$ converges. However, this is due to the fact the $m\in\K_L$ and the noise is Hilbert-Schmidt.
\end{proof}

\begin{lemma}\label{lem:brackets-lower-bound}
The following uniform lower-bound holds every initial data $w = (u,x,v,z) \in \Hbf\times \cM$, and $h \in \Hbf_L\times T_v\cM$
\begin{equation}\label{eq:spanning-cond}
	\max\Big\{ \big|\big\langle Q^m,h\big\rangle_L\big|, \big|\big\langle \Lambda_L Q^m, h\big\rangle_L\big| : m\in \K_L\Big\} \geqc_{\rho} \frac{|h|}{( 1 + |z|)^3}.
\end{equation}
\end{lemma}
\begin{proof}
To show \eqref{eq:spanning-cond} we must consider the different behaviors of 
\[
\big\langle \Lambda_L Q^m, h\big\rangle_L = \big\langle [F,Q^m]_L,h\big\rangle_L - \big\langle [A,Q^m]_L, h\big\rangle_{L}
\]
for different values of the initial data $w \in \Hbf\times \cM$ due to the presence of the cut-off. We divide the proof into two cases using a parameter $\delta \in (0,1)$, which will be determined later. \\

{\em Case 1:} We first consider the case where $\chi_{\rho}(\|u\|_{\Hbf}) > \delta$. This case is the easiest, since we can use the $z$ process to help span the $(x,v)$ directions. Indeed notice that if we choose a $m\in \K_L$ so that $m = j \in \{1,\ldots,2d\}$, then  $Q^m = \hat{e}_j^z$, then one easily computes for $j=1,\ldots,d$
\[
     \big|\big\langle \Lambda_L Q^m(w),h\big\rangle\big| = \frac{\chi_{\rho}(\|u\|_{\Hbf})}{(1 + \abs{z^j}^2)^{3/2}}|\langle \hat{e}_j,h\rangle_L| \geq \frac{\delta}{( 1+  |z| )^3}|\langle\hat{e}_j,h\rangle_L|,
\]
where $\{\hat{e}_j\}_{j=1}^d$ is the cannonical basis for $\R^d$, taken here to be elements of $T_x\T^d \subseteq \Hbf_L\times T_v \cM$. Similarly for $j=d+1,\ldots 2d$, we have
\[
	\big|\big\langle \Lambda_L Q^k(w),h\big\rangle\big| \geq \frac{\delta}{(1 +  |z| )^3}|\langle \Pi_{v}\hat{e}_{j-d},h\rangle_L|.
\]
and $\{\Pi_{v}\hat{e}_{j}\}_{j=1}^d$ is a spanning set for $T_v\S^{d-1}\subseteq \Hbf_L\times T_v \cM$. Therefore we can easily conclude the lower bound
\[
	\max\Big\{ \big|\big\langle Q^m,h\big\rangle_L\big|, \big|\big\langle \Lambda_L Q^m, h\big\rangle_L\big| : m\in \K_L\Big\} \geqc \delta \frac{|h|}{( 1+  |z| )^3}.
\]

{\em Case 2:} We now consider the case $\chi_{\rho}(\|u\|_{\Hbf}) \leq \delta$. Here, we cannot rely on the regularization introduced by the $z$ process since we are in a region where it's coupling with $x$ and $v$ may be turned off or very small. Here, the drift is fully turned on and if we choose $m \in \K_L$ so that $m = (k,i)$ and $Q^m = q_ke_k\gamma_k^i$, we obtain
\[
	\Lambda_LQ^m(w) = q_k[V_0(w),e_k\gamma_k^i] -q_k[B(u,u),e_k\gamma_k^i]_L - q_k[A,e_k\gamma_k^i]_L -q_k\frac{1}{\rho}\chi^\prime(\|u\|_{\Hbf}/\rho)\frac{u_k}{\|u\|_{\Hbf}}H(v,z).
\]
Using the fact that we are in the region $\|u\|_{\Hbf} \leq 2\rho$, we have that
\begin{equation}\label{eq:B,Qk,h-estimate}
	|\langle [A,e_k\gamma_k^i]_L,h\rangle_L|+ |\langle[B(u,u),e_k\gamma_k^i]_L,h\rangle_{L}| \leqc_{\rho}\sum_{i=1}^{d-1}\sum_{k\in K_L} |\langle e_k\gamma_k^i,h\rangle_{L}|,
\end{equation}
additionally since $\chi_{\rho}(\|u\|_{\Hbf}) \leq \delta$ then
\begin{equation}
 \frac{1}{\rho}\chi^\prime(\|u\|_{\Hbf}/\rho)\frac{|u_k|}{\|u\|_{\Hbf}}|\langle H(v,z),h\rangle_L| \leqc_{\rho} \delta |h|.
\end{equation}
This implies that
\begin{equation}
	\delta|h| + |\langle \Lambda_LQ^m, h\rangle_L| + \sum_{i=1}^{d-1}\sum_{k\in K_L} q_k|\langle e_k\gamma_k^i,h\rangle_{L}|\geqc_\rho |\langle [V,e_k\gamma_k^i] ,h\rangle_L|, 
\end{equation}
which, in turn, implies that
\begin{equation}
\begin{aligned}
&\delta|h| + \max\Big\{ \big|\big\langle Q^m,h\big\rangle_L\big|, \big|\big\langle \Lambda_L Q^m, h\big\rangle_L\big| : m\in \K_L\Big\}\\
 &\hspace{1in}\geqc_\rho \max\Big\{ \big|\langle [V, e_k\gamma_k^i],h\rangle_L\big|, \big|\big\langle e_k\gamma_k^i, h\big\rangle_L\big| : k\in K_L,\, i\in \{1,\ldots,d-1\}\Big\}.
 \end{aligned}
\end{equation}
Finally, an easy modification of Lemma \ref{lem:v-span} gives
\[
	\max\Big\{ \big|\langle [V, e_k\gamma_k^i],h\rangle_L\big|, \big|\big\langle e_k\gamma_k^i, h\big\rangle_L\big| : k\in K_L,\, i\in \{1,\ldots,d-1\}\Big\} \geqc |h|,
\]
so that taking $\delta$ small enough (depending on $\rho$) we obtain the desired lower bound.
\end{proof}

We are now equipped to prove Lemma \ref{lem:pMalNonDeg}.

\begin{proof}[\textbf{Proof of Lemma \ref{lem:pMalNonDeg}}]
Fix initial data $w\in \Hbf\times\cM$ and let $h\in \Hbf_L\times T_v\cM$ with $|h| =1$, fix $t\in (0,1)$. Denote for each $m\in \K_L$
\[
	X_s^m \equiv \big\langle V^L_s Q^m,h\big\rangle_L.
\]
 It is sufficient to show that
\begin{equation}\label{eq:suff-prop-est}
	\P\bigg(\bigcap_{m\in\K_L}\{\|X^m\|_{L^2([0,t])}^2 < \ep\}\bigg) \leqc_{p,\rho} t^{-ap}( 1 +  |z| )^{bp}\ep^p,
\end{equation}
where the constant does not depend on $h$ or the initial data. Using Proposition \ref{prop:commutator-Lemma}, as well as Lemmas \ref{lem:holder-ests}, \ref{lem:commutator-bounds} and \ref{lem:U,V-bounds} we find that we have the almost-sure bound
\begin{equation}\label{eq:Gk-holder-bound}
	[X^m]_{C^{1}([0,1])}^2 \leq C_\rho,
\end{equation}
where $C_\rho\geq 1$ is a determinisitic constant depending only on $\rho$. Applying Lemma \ref{lem:interp-eq} with $f = \int_0^\cdot X_s\ds$ and $\alpha = 1$, and then applying Cauchy-Schwarz we arrive at the inequality
\begin{equation}\label{eq:interp-application-1}
	\|X^m\|_{L^\infty([0,t])} 
	\leq 4 \,t^{-\frac{1}{2}}\|X^m\|_{L^2([0,t])}^{\frac{1}{2}}\cdot\max\Big\{\|X^m\|_{L^2([0,t])}^{\frac{1}{2}},[X^m]_{C^{1}([0,1])}^{\frac{1}{2}}\Big\}.
\end{equation}
Therefore, we can deduce
\[
	\P\bigg(\bigcap_{m\in\K_L}\{\|X^m\|_{L^2([0,t])}^2 <\ep\} \bigg) \leq \P\bigg(\bigcap_{m\in\K_L}\{|X^m\|_{L^\infty([0,t])} < 4 C_\rho\,t^{-\frac{1}{2}}\ep^{\frac{1}{4}}\}\bigg).
\]
Next, using Lemma \ref{prop:commutator-Lemma}, we write
\[
	X_s^m = X_0^m + \int_0^s B_r^m\,\dr
\]
where $B^m_s$ is the $\R$ valued predictable process defined by $B_s^m \equiv \big\langle V^L_s \Lambda_L Q^m(w_s),h\big\rangle_L$. This means that when $\|X^m\|_{L^\infty(0,t)} < 4C_{\rho}\,t^{-\frac{1}{2}}\ep^{\frac{1}{4}}$, then
\[
\left|\int_0^s B^m_r \dr \right| \leq 8C_{\rho}\,t^{-\frac{1}{2}}\ep^{\frac{1}{4}}
\]
Applying Lemma \ref{lem:interp-eq} again with $f = \int_0^\cdot B^m_s\ds$ and $\alpha = \frac{1}{3}$, we find
\begin{equation}\label{eq:interp-application-2}
	\|B^m\|_{L^\infty([0,t])} 
	\leq 4t^{-1}\left\|\int_0^\cdot B_s^m\ds\right\|_{L^\infty([0,t])}^{1/4}\times\max\left\{\left\|\int_0^\cdot B_s^m\ds\right\|_{L^\infty([0,t])}^{3/4},[B^m]_{C^{1/3}([0,1])}^{3/4}\right\},
\end{equation}
and an application of Proposition \ref{prop:commutator-Lemma}, along with Lemmas \ref{lem:holder-ests}, \ref{lem:commutator-bounds} and \ref{lem:U,V-bounds} gives the following H\"{o}lder estimate on $B^k$ for each $p \geq 1$
\begin{equation}\label{eq:regularity-B^k}
	\E[B^m]_{C^{1/3}([0,1])}^p \leqc_{p,\rho} 1.
\end{equation}
Since estimate \eqref{eq:Gk-holder-bound} implies that for each $p \in (1,\infty)$ and every $\ep \in (0,1)$
\[
	\P\Big([B^m]_{C^{1/3}([0,1])} \geq 8C_\rho\,t^{-\frac{1}{2}}\ep^{-\frac{1}{204}}\Big) \leqc_{p,\rho} \ep^{p}
\]
we can with overwhelming probability restrict ourselves to the event $\bigcap_{m\in \K_L}\{[X^m]_{C^{1/3}([0,1])} < 8C_\rho\,t^{-\frac{1}{2}}\ep^{-\frac{1}{204}}\}$.

 The choice of the exact power for $\ep^{-1/204}$ above is somewhat arbitrary and is chosen simply to give rise to the power of $\ep^{1/18}$ in inequality \eqref{eq:smallness-estimate2}. It is certainly possible to use other powers on $\ep$ without changing the essence of the proof.

Using inequality \eqref{eq:interp-application-2} we conclude that for every $p \geq 1$
\begin{equation}\label{eq:smallness-estimate1}
\begin{aligned}
	&\P\bigg(\bigcap_{m\in\K_L}\left\{\|X^m\|_{L^2([0,t])}^2 < \ep\right\}\bigg) \leqc_{p,\rho}\\
	&\hspace{.5in} \P\bigg(\bigcap_{m\in\K_L}\left\{\|X^m\|_{L^\infty([0,t])} < 4 C_\rho\,t^{-\frac{1}{2}}\ep^{\frac{1}{4}}\right\}\cap\left\{\|B^m\|_{L^\infty([0,t])} < 32 C_\rho\,t^{-\frac{3}{2}}\ep^{\frac{1}{17}}\right\}\bigg) + \ep^p.
	\end{aligned}
\end{equation}

By choosing $\ep \leqc t^{a}$ small enough for a large enough constant $a> 1$, we can remove the factor of $t^{-\frac{1}{2}}$ and $t^{-\frac{3}{2}}$ above at the expense of a slightly worse power on $\ep$. To remove this $t$-dependent restriction on $\ep$, we can treat the case $t^{a}\leqc \ep$ by simply using the fact that probabilities are bounded by 1 and that $1\leqc t^{-ap}\ep^p$ to deduce that for all $\ep \in (0,1)$ and $p \geq 1$
\begin{equation}\label{eq:smallness-estimate2}
\begin{aligned}
	&\P\bigg(\bigcap_{m\in\K_L}\left\{\|X^m\|_{L^2([0,t])}^2 < \ep\right\}\bigg) \leqc_{p,\rho}\\
	&\hspace{.5in} \P\bigg(\bigcap_{m\in\K_L}\left\{\|X^m\|_{L^\infty([0,t])} < \ep^{\frac{1}{5}}\right\}\cap\left\{\|B^m\|_{L^\infty([0,t])} < \ep^{\frac{1}{18}}\right\}\bigg) + t^{-ap}\ep^p.
	\end{aligned}
\end{equation}

Next, we show that for small enough $\ep$, and each initial data $w\in \Hbf\times\cM$ 
\begin{equation}\label{eq:emptyset-cond}
	\bigcap_{m\in\K_L}\left\{|X^m_0| \leq \ep \right\}\cap\left\{|B^m_0| < \ep^{r_*}\right\} = \emptyset,
\end{equation}
where $r^*$ is some number less than $1$. That is, at time $t=0$ for small enough $\ep$, it is not possible for all the $\{X^m\}$ and all the $\{B^m\}$ to be small. Indeed, since $X^m_0 = \langle Q^m,h\rangle_L$ and $B^m_0 = \langle \Lambda_L Q^m, h\rangle_L$ this follows from Lemma \ref{lem:brackets-lower-bound} since $|\langle Q^m,h\rangle_L| < \ep$ and $|\langle \Lambda_L Q^m, h\rangle_L| < \ep^{r^*}$ imply by \eqref{eq:spanning-cond} that
\[
	1\leqc_{\rho} (1 +  |z| )^3 \ep^{r_*} 
\]
Therefore choosing $\ep$ small enough so that $\ep \leqc_\rho (1 + |z|)^{-b}$ for a sufficiently large constant $b>0$ we deduce a contradiction and conclude that \eqref{eq:emptyset-cond} must hold. Again, to remove the $z$-dependent restriction on $\ep$ we can replace $\ep$ by $(1 + |z|)^b\ep$ on the right-hand side of estimate \eqref{eq:smallness-estimate2}, giving our desired estimate \eqref{eq:suff-prop-est}.
\end{proof}

\begin{remark} \label{rmk:NorrisMal} 
In order to treat noise as in Assumption \ref{a:Highs}, one needs to adjust the above proof in two ways. 
First, in the definition of the cutoff process \eqref{eq:cut-off-eq}, one needs to add additional Brownian motions to the modes $k$ in $(u_t)$ for which $k \not\in \cK$, in the same manner as was done for the Lagrangian flow, that is $\chi_\rho(\norm{u}_{\Hbf}) e_k\gamma_k^i z_{k,i} / (1 + \abs{z_{k,i}}^2)^{1/2}$ for $k \not\in \cK$. 
Then, in the proof of Lemma \ref{lem:pMalNonDeg}, for $\chi_{\rho}(\norm{u}_{\Hbf})  < \delta$, one needs to use Lie brackets of the Navier-Stokes nonlinearity to fill  the missing degrees of freedom in Navier-Stokes (these brackets are computed for 2D and 3D respectively in \cite{E2001-lg,Romito2004-rc}; see also Section \ref{sec:galerkin}). 
This requires taking one more time derivative in the proof of Lemma \ref{lem:pMalNonDeg} (allowing noise from the high frequencies to propagate to the lower modes), which in turn, requires the use of a version of Norris' Lemma \cite{Norris86} (in addition to Lemma \ref{lem:interp-eq}), as described in e.g. \cite{HairerNotes11}. Analogous to \cite{EH01,RomitoXu11}, one needs to slightly refine the statement found in e.g. \cite{HairerNotes11} to handle the singularity for short-times but this is a straightforward calculation.  
\end{remark}

\subsection{Basic estimates on Jacobians and Malliavin derivatives} \label{sec:FundEstsCutoff} 

The proofs of the following Lemmas are standard and are omitted for brevity (see \cite{DPZ96}).
 
\begin{lemma} \label{lem:CutBds}
The statements of Proposition \ref{prop:WPapp} hold for the $(w_t)$ process. 
We record the quantitative estimates here for the readers' convenience. 
For all $\gamma < \alpha-\frac{d}{2}$, $T \leq 1$, and $p \in [2,\infty)$ there holds 
\begin{align}
\EE \sup_{t \in [0,T]} \norm{w_t}_{H^{\gamma}}^p &\lesssim_{p,\gamma,\rho} 1 + \norm{w_0}^p_{H^\gamma} \\ 
\EE \int_0^T \norm{w_s}_{H^{\gamma + (d-1)}}^2 ds & \lesssim_{\gamma,\rho} 1 + \norm{w_0}^2_{H^\gamma}.  
\end{align}
\end{lemma}

We also need the following improved short-time regularization estimates. 
Specifically, for regularities all the way up to $\gamma < \sigma+(d-1)$. 
This is crucial for dealing with the high frequencies of the control.
\begin{lemma}\label{lem:aboveSigma}
For all $\gamma \in (\sigma, \sigma + (d-1))$,  $p \in [2,\infty)$, and $T \leq 1$ there holds for all $\delta > 0$,  
\begin{align}
\EE \left(\sup_{t \in [0,T]} t^{ \frac{\gamma - \sigma}{2(d-1)}} \norm{w_t}_{H^\gamma} \right)^p  & \lesssim_p 1 + \norm{w_0}_{H^{\sigma}}^p \label{ineq:LpHgCut}\\ 
\EE\int_t^T \norm{w_s}_{H^{\gamma+(d-1)}}^2 \ds & \lesssim_\delta 1 + t^{-2r}\norm{w_0}_{H^{\sigma}}^4, \label{ineq:L2tHgCut}
\end{align}
where
\begin{align}
r = \frac{\sigma - \left(\gamma + 2 - d + 2(d-1)\delta\right)}{2(d-1)} > 0. 
\end{align}
\end{lemma}
 
\begin{lemma} \label{lem:JstBds}
The following properties are satisfied for $J_{s,t}$ and $U^H_{s,t}$ for $0 < s < t < T \leq 1$, 
\begin{itemize}
\item[(i)] there holds for $\gamma \leq \sigma$, (almosts surely)
\begin{align}
\norm{J_{s,t} h}_{H^{\gamma}\times T_{v_t}\cM} +  \norm{U^H_{s,t} h^H}_{H^{\gamma}} &\lesssim_{\rho} \norm{h}_{H^\gamma\times T_{v_s}\cM},\\
\int_0^T \norm{J_{s,t}h}_{H^{\gamma + (d-1)}\times T_{v_t}\cM}^2 \dt + \int_0^T \norm{U^H_{s,t}h^H}_{H^{\gamma + (d-1)}}^2 \dt & \lesssim_{\rho} \norm{h}^2_{H^\gamma\times T_{v_s}\cM}.  
\end{align}
\item[(ii)] for all $\gamma \in(\sigma, \sigma+ (d-1))$ there holds (almost surely), 
\begin{align}
(t-s)^{\frac{\gamma - \sigma}{2(d-1)}} \norm{J_{s,t}h}_{H^{\gamma}\times T_{v_t}\cM} + (t-s)^{\frac{\gamma - \sigma}{2(d-1)}} \norm{U^H_{s,t}h^H}_{H^{\gamma}}&\lesssim_{\rho,T,\delta} \norm{h}_{H^\sigma\times T_{v_s}\cM}; 
\end{align}
\item[(iii)] for all$\gamma \in(\sigma, \sigma+ (d-1))$ and all $\delta$ sufficiently small
\begin{align}
&\EE \int_{s+s'}^T \norm{J_{s,t}h}_{H^{\gamma+2(d-1)}\times T_{v_t}\cM}^2 dt + \EE \int_{s+s'}^T \norm{U^H_{s,t}h^H}_{H^{\gamma+2(d-1)}}^2 dt\\
&\hspace{.5in} \lesssim_{\delta} (s')^{-2r}(1 + \norm{w_0}_{H^\sigma})^2\norm{h}_{H^{\sigma}\times T_{v_s}\cM}^2, \label{ineq:LtHgammaJst}
\end{align}
where 
\begin{align}
r = \frac{\sigma - \left(\gamma + 2 - d + 2(d-1)\delta\right)}{2(d-1)} > 0. 
\end{align}
\end{itemize}
\end{lemma}
\begin{remark} 
Note that the above estimates all hold almost almost surely and are independent of $w_0$ except for \eqref{ineq:LtHgammaJst}. This is because only \eqref{ineq:LtHgammaJst} requires regularities above $\sigma$ on the (linearization of) the nonlinear term. 
\end{remark}

\begin{lemma}\label{lem:U,V-bounds}
For each $p \geq 1$ an $T \leq 1$, the processes $U_t^L$ and $V_t^L$ satisfy the following bounds,
\[
	\sup_{t\in[0,T]} (|U_t^L| + |V^L_t|) \leqc_{\rho,p} 1
\]
and the constants {\em do not} depend on the initial data for $w_t$.
\end{lemma}

We also require the following estimates on the Jacobian, as in \cite{EH01}, which control the effect of low frequencies on high frequencies and vice-versa. 
\begin{lemma} \label{lem:CrossJacobian}
For each $T \leq 1$ and $h^L\in \Hbf_L\times T_{v}\cM$ and $h^H\in \Hbf_H$ we have the almost sure bounds
\begin{align}
\sup_{0<t<T} \|D_Hw^H_t h^L\|_{\Hbf_H}& \leqc_\rho T^{\frac{1}{2}} |h_L|\label{ineq:DLwHvan}\\ 
\sup_{0<t<T} |D_H w^L_t h^H| & \leqc_\rho T \|h_H\|_{\Hbf_H}, \label{ineq:DLwHvan}
\end{align}
(where the constants do not depend on the initial data $w$).
\end{lemma}
\begin{proof} 
Consider the case of $D_L w^H$. In this case we have 
\begin{align}
\partial_t \left(D_L w_t^H h^L\right) = D_HF_H(w_t) D_L w_t^H h^L + D_LF_H(w_t) D_L w_t^L h^L - A_H(D_L w^H_t h^L) 
\end{align}
and $D_L w_0^H h^L = 0$. Therefore
\begin{align}
D_L w_t^H h^L & = \int_0^t U^H_{s,t} D_LF_H(w_s) D_L w_s^L h^L ds. 
\end{align}
By Lemma \ref{lem:JstBds}, 
\begin{align}
\norm{D_L w_t^H h^L}_{\Hbf_L} & \lesssim \int_0^t \frac{1}{(t-s)^{1/2}} \norm{D_L F_H(w_s) D_L w_s^L h^L}_{H^{\sigma-1}} \ds \\ 
& \lesssim_\rho \int_0^t \frac{1}{(t-s)^{1/2}} \ds  \left(\sup_{ 0 < s < t} \norm{J_{0,s} h^L}_{\Hbf\times T_{v_s}\cM} \right) \\ 
& \lesssim \sqrt{t} |h^L|. 
\end{align}

The estimate on \eqref{ineq:DLwHvan} follows similarly (except no smoothing is necessary).  
\end{proof}

Next, we compute and estimate the Malliavin derivatives of the necessary quantities.  
First, we compute  
\[
	\MalD_s w_t f = J_{s,t}Qf 
\]
\[
	\MalD_s \left(U^L_{r,t}h\right) f = \int_r^t U^L_{l,t} \bar{D}^2 F_L(w_l)[U^L_{r,l} h, J_{s,l}Qf]\dee l
\]
\[
	\MalD_s \left(U^H_{r,t} h\right) f = \int_r^t U^H_{l,t}\bar{D}^2 F_H(w_l)[U^H_{r,l} h,J_{s,l}Qf]\dee l, 
\]
where $\bar{D}^2F$ denotes the full second variation of $F$ extended to the linear space $\Hbf_L\times \R^{4d}$. We further have
\begin{align}
\MalD_s Dw_t h f = \int_0^t J_{r,t} \bar{D}^2 F(w_r)[\MalD_s w_r f, J_{0,r}h] dr = \int_s^t J_{r,t} \bar{D}^2 F(w_r)[J_{s,r} Q f, J_{0,r}h] dr. 
\end{align}
Furthermore, one has the following for the derivatives of the inverse Malliavin matrix and $V^L_t$ 
\begin{align}
	\MalD_s (\cC_{T}^L)^{-1}f = - (\cC_{T}^L)^{-1}[\MalD_s\cC_{T}^Lf](\cC_{T}^L)^{-1}\quad \text{and}\quad \MalD_s V^L_tf = - V^L_t[\MalD_sU^L_tf]V^L_t.
\end{align}
\begin{lemma} \label{lem:MalliavinForAll}
The following estimates hold almost surely for $T \leq 1$, (and are independent of $\norm{w_0}_{\Hbf}$), 
\begin{align}
\sup_{0 < r < t < T} \abs{\MalD_s U^L_{r,t}h^L}_{\Wbf \to \Hbf_L\times T_{v_t}\cM} & \lesssim_{\rho} t\|h^L\|_{\Hbf_L\times T_{v_r}\cM} \\
\sup_{0 < r < t < T} \abs{\MalD_s V^L_{r,t}h^L}_{\Wbf\to \Hbf_L\times T_{v_t}\cM} & \lesssim_{\rho} t\|h^L\|_{\Hbf_L\times T_{v_r}\cM} \\
\sup_{0 < r < t < T} \norm{\MalD_s U^H_{r,t}h^H}_{\Wbf \to \Hbf_H} & \lesssim_{\rho} t^{\frac{1}{2}}\|h^H\|_{\Hbf_L\times T_{v_r}\cM} \\ 
\sup_{0 < r < t < T} \norm{\MalD_s J_{r,t}h}_{\Wbf \to \Hbf\times T_{v_t}\cM} & \lesssim_{\rho} t^{\frac{1}{2}}\|h\|_{\Hbf_L\times T_{v_r}\cM}. 
\end{align}
\end{lemma}
\begin{proof} 
Using the formula above, the case of $\MalD_s U^L_{r,t}$ follows immediately from Lemma \ref{lem:U,V-bounds}.  
The case of $U^H$ follows from the following, noting that $\sigma < \alpha -\frac{d}{2}$ and that $Q:\Wbf \to \Hbf\times \cM$ is bounded,   
\begin{align}
\norm{\MalD_s \left(U^H_{r,t} h\right) f}_{\Hbf_H} & \lesssim \int_r^t \frac{1}{(t-l)^{1/2}} \norm{U_{r,l}^H h}_{\Hbf_H} \norm{J_{s,l}Q f}_{\Hbf\times T_{v_l}\cM} \dee l \lesssim \sqrt{t} \norm{h}_{\Hbf\times T_{v}\cM} \norm{f}_{\Wbf}.
\end{align}
Consider next estimating $\MalD_s Dw_t h f$. 
For this we get (almosts surely due to the cutoff), 
\begin{align}
\norm{\MalD_s \left(Dw_t h\right) f}_{\Hbf_H} \lesssim_{\rho} \int_s^t \frac{1}{(t-r)^{1/2}} \norm{J_{s,r} Q f}_{\Hbf\times T_{v_r}\cM} \norm{J_{0,r}h}_{H^\sigma} dr \lesssim t^{1/2}. 
\end{align}
\end{proof}

\begin{lemma} \label{lem:MDbds}
The following holds for all $s < T$ and $1 \leq p < \infty$, (the constants $a,b$ are from Lemma \ref{lem:Malmatrix-bound}), 
\begin{align}
\EE \norm{\MalD_s (\cC_{T}^L)^{-1} }_{\Wbf \to \Hbf_L\times T_v\cM}^p \lesssim_p  \left(T^{-2a+1}(1+\abs{z})^{2b}\right)^{p}.
\end{align}
\end{lemma}
\begin{proof} 
Follows by Lemma \ref{lem:U,V-bounds} and Lemma \ref{lem:Malmatrix-bound}. 
\end{proof}

\section{Weak irreducibility and approximate control} \label{sec:Irr}
First, we prove Proposition \ref{prop:UniErg}, hence deducing the weak irreducibility of the stationary measures for the Markov processes $(u_t,x_t)$, $(u_t,x_t,v_t)$, $(u_t,x_t,\check{v}_t)$.
Combined with the strong Feller property, this yields unique stationary measures for these processes by the Doob-Khasminskii Theorem \cite{Doob1948-rb, khasminskii1960ergodic}. 

\begin{lemma} \label{lem:Ctrluxv}
Recall the control problem \eqref{eq:ctrl} for Systems \ref{sys:NSE}--\ref{sys:3DNSE}. Suppose that $\mathcal{K}$ is symmetric and $(1,0), (0,1) \in \mathcal{K}$ in 2D and $(1,0,0), (0,1,0),(0,0,1) \in \mathcal{K}$ in 3D.

Let $(x,v)$, $(x',v')$ be arbitrary points in $\T^{d} \times \S^{d-1}$. 
Then there exists a smooth control $Q g$ such that 
\begin{align}
(u_0,x_0,v_0) = (0,x,v), \quad (u_1,x_1,v_1) = (0,x',v').
\end{align}
Furthermore, $g$ can be chosen to depend smoothly on $x,x',v,v'$ and supported only in frequencies $\abs{k}_{\infty} \leq 1$. 
All of the above holds also for the $(u_t,x_t,\check{v}_t)$ process. 
\end{lemma} 
\begin{remark} 
By choosing arbitrary representatives on $\S^{d-1}$, it is clear that controlling the $(u_t,x_t,v_t)$ and $(u_t,x_t,\check{v}_t)$ processes, regarding $v_t,\check{v}_t$ as elements on  $\S^{d-1}$, implies controllability of the processes when considered on $P^{d-1}$. 
\end{remark}
\begin{proof}
First, let us consider the two dimensional case. 
Let $x = (a_0,b_0)$ and $x' = (a_1,b_1)$. 
For $t \in (0,1/4)$, suppose the velocity field is given by the shear flow
\begin{align}
u_t(y_1,y_2) = f_a(t) \begin{pmatrix} \cos(y_2 - b_0) \\ 0 \end{pmatrix}, 
\end{align}
such that $f_a \in C^\infty_c(0,1/4)$  and $\int_0^{1/4} f_a(t) dt = a_1 - a_0$. 
Similarly, for $t \in (1/4,1/2)$, suppose the velocity field was the shear flow 
\begin{align}
u_t(y_1,y_2) = f_b(t) \begin{pmatrix} 0 \\ \cos(y_1 - a_1) \end{pmatrix}, 
\end{align}
such that  $f_b \in C^\infty_c((1/4,1/2))$ and $\int_{1/4}^{1/2} f_b(t) dt = b_1 - b_0$. 
It follows that the solution to the ODE \eqref{eq:xctrl} satisfies $x_1 = (a_1,b_1)$. 

Next, we explain how to set $g$ in order to produce these flows. 
Notice that the shear flows $(\cos(y-b_0),0)$ and  $(0,\cos(x-a_1))$ are stationary solutions of 2D Euler: the nonlinearity vanishes on these flows. 
Hence, it suffices to control the Stokes flow, which gives the following control: 
\begin{align}
Q g(t) =  
\left(f_a'(t) + f_a(t) \right) \begin{pmatrix} \cos(y_2 - b_0) \\ 0 \end{pmatrix} + \left(f_b'(t) + f_b(t) \right) \begin{pmatrix} 0 \\ \cos(y_1 - a_1) \end{pmatrix}.
\end{align}
By the angle-difference formula and the assumptions on $\cK$, $g$ satisfies the requisite properties. 

Next, we augment the previous control also to deal with $v_t$; the treatment for $\check{v}_t$ is analogous and is omitted for brevity. 
During this time we have moved $v_t$ some amount, let $v_{1/2}$ be the new value. 
Suppose that the velocity field were given by the cellular flow
\begin{align}
u(t,y_1,y_2) = f_v(t) \begin{pmatrix} -\sin(y_2 - b_1) \\ \sin(y_1-a_1) \end{pmatrix}, \label{def:cellctrl}
\end{align}
such that $f_v \in C^\infty_c((1/2,1))$ with $\int_{1/2}^1 f_v(t) dt = \angle v' - \angle v_{1/2}$. 
This induces a rotation of $v_t$ via \eqref{eq:vctrl} into the desired final point without moving $x_t$. 
As above, the cellular flow is both a stationary solution of the 2D Euler equations and an eigenfunction of the Stokes operator. 
Therefore, it suffices to set $g$ on $t \in (1/2,1)$ to be such that 
\begin{align}
Q g(t) & = \left(f_v'(t) + f_v(t) \right) \begin{pmatrix} -\sin(y_2 - b_1) \\ \sin(y_1-a_1) \end{pmatrix}. 
\end{align}
 This completes the proof in 2D. 

Next, consider the 3D argument. It is clear that a similar proof applies to the $(u_t,x_t)$ process by utilizing 2D shear flows aligned with any of the three Cartesian directions. 
For the $(u_t,x_t,v_t)$ process, we consider the problem of controlling the $v_t$ process (as an element of $\S^2$) from one arbitrary position $v\in \S^{2}$ to another $v^\prime \in \S^{2}$ without moving $x_t$ using 2D cellular flows aligned with any of the three Cartesian directions. 
Each of these flows induces rotation along curves of constant `latitude' aligned with one of the three Cartesian directions. 
Note that no flow gives lines of constant longitude in any direction.
Arbitrarily, set the $x,y$ plane to be the equatorial plane relative to which we assign latitude and longitude.  
Using the cellular flow that is constant in $z$, adjust the longitude of $v_t$ so that $v_{1/3}$ lies in the $y,z$ plane. 
Then, using a cellular flow that is constant in $x$, adjust the latitude so that $v_{2/3}$ lies at the latitude of $v^\prime$. 
Finally, by re-applying the cellular flow that is constant in $z$, adjust the longitude so that $v^\prime = v_1$. 
\end{proof}

The controllability provided in Lemma \ref{lem:Ctrluxv} implies the following non-degeneracy of the Markov transition kernels. 
\begin{lemma} \label{lem:MarkovTrans}
For all $t > 0$ and $\eps> 0$, $\exists \eps' > 0$ such that for all $(x,v), (x',v') \in \T^d \times \S^{d-1}$ and all $u \in B_{\eps'}(0)$, 
\begin{align}
& \PP \left( (u_t,x_t) \in B_{\eps}(0) \times B_{\eps}(x') | (u_0,x_0) = (u,x) \right) > 0 \\ 
& \PP \left( (u_t,x_t,v_t) \in B_{\eps}(0) \times B_{\eps}(x') \times B_\eps(v') | (u_0,x_0,v_0) = (u,x,v)\right) > 0 \\ 
& \PP \left( (u_t,x_t,\check{v}_t) \in B_{\eps}(0) \times B_{\eps}(x') \times B_\eps(v') | (u_0,x_0,v_0) = (u,x,v) \right) > 0. 
\end{align} 
\end{lemma}
\begin{proof} 
Such non-degeneracy properties normally follow from standard perturbation arguments. 
However, one must be somewhat careful with the regularity, as we require $\sigma \in (\alpha-2(d-1),\alpha-\frac{d}{2})$ (i.e. close to the highest available regularity). 
Let us treat the $(u_t,x_t)$ process; the $(u_t,x_t,v_t)$ and $(u_t,x_t,\check{v}_t)$ processes are the same. 
Let $Qg$ be a control given as in Lemma \ref{lem:Ctrluxv} corresponding to the desired endpoints $x,x'$.
Let $u^c_t$ be the controlled solution from Lemma \ref{lem:Ctrluxv}. 
The first step is to prove that for all $\eps$, there holds
\begin{align}
\PP\left(\norm{u_t - u_t^c}_{L^\infty_t (0,1;\Hbf)} \lesssim \epsilon \right) > 0. \label{ineq:utuctctrl}
\end{align}   
Note that
the control is built from only $\Pi_{\leq 1} Q g$.
By the regularity of the stochastic convolution (Lemma \ref{lem:StochConv}) and positivity of the Wiener measure, $\forall \eps > 0$,
\begin{equation}\label{eq:AprxEvent}
\PP\left( \sup_{t \in (0,1)} \norm{\Gamma_t - \int_0^t e^{-(t-s)A} Qg_s \ds}_{L^\infty_t(0,1;\Hbf)} < \eps \right) > 0. 
\end{equation}    

Let $u_t$ be a solution to the stochastic Navier-Stokes with a sample path $\omega$ such that the event in \eqref{eq:AprxEvent} holds.
Then from the mild form 
\begin{align}
u_t -u^c_t = e^{-t A} u_0 + \int_0^t e^{-(t-s)A} \left( B(u_s,u_s) - B(u_s^c, u_s^c) \right) \ds + \Gamma_t - \int_0^t e^{-(t-s)A} Qg_s \ds
\end{align}
(actually by our choice of control $B(u_s^c,u_s^c) = 0$). 
By a generalized Gr{\"o}nwall's inequality [Lemma A.2, \cite{Kruse}] and parabolic smoothing, we have that 
\[
\norm{u_t - u_t^c}_{L^\infty(0,1;\Hbf)} \leq K'\epsilon,
\]
for a universal constant $K'$ depending only on $\sigma$, $\alpha$ (provided that $\norm{u_0}_{\Hbf} \leq \eps$). 
Therefore, we have \eqref{ineq:utuctctrl}. 
For the $x_t$ process, we similarly let $x_t$ and $x_t^c$ be the trajectories associated with the controlled system and that of the sample path $\omega$ (respectively). 
Then, (viewing $x_t,x^c_t$ as elements in $\Real^d$), 
\begin{align}
\frac{d}{dt}(x_t^c -x_t) = u_t^c(x^c_t) - u_t(x_t) = \left(u^c_t(x^c_t) - u^c_t(x_t)\right) + \left(u^c_t(x_t) - u_t(x_t)\right).  
\end{align}
We then obtain by the stability of the $(u_t)$ process (by potentially adjusting $K'$ and using $\sigma > \frac{d}{2}+1$ to apply Sobolev embedding to $\grad u$), 
\begin{align}
\PP\left( \set{\norm{u_1}_{\Hbf} \leq K'\epsilon} \bigcap \set{ d(x_1,x') < K'\eps } \right) > 0. 
\end{align}
The desired non-degeneracy for the Markov transition kernel then follows. 
\end{proof}

\begin{proof}[\textbf{Proof of Proposition \ref{prop:UniErg}}]
We prove this in the case of $(u_t,x_t)$; the processes including $P^{d-1}$ are the same. 
First, we verify irreducibility of stationary measures of the $(u_t)$ process in $H^\sigma$. 
In the case $L^2$ this is well-known; see e.g.  \cite{E2001-lg}. 
This can be proved by observing that there if there were no forcing we have, 
\begin{align}
\frac{d}{dt}\norm{u_t}_{L^2}^2 \leq -\norm{\grad u_t}_{L^2}^2 \lesssim -\norm{u_t}_{L^2}^2.
\end{align}
At the same time, in the absence of forcing, standard energy estimates give the uniform bound with $\delta > 0$, $\norm{u_t}_{H^{\sigma+\delta}} \lesssim_{\delta} \norm{u_0}_{H^{\sigma+\delta}}$ with an implicit constant that is \emph{independent of time}. 
Hence, Sobolev interpolation gives $\norm{u_t}_{H^\sigma} \lesssim \norm{u_0}_{H^{\sigma+\delta}} e^{-ct}$, for some constant $c$ depending only on $\sigma,\delta$.  

Let $\tilde\mu$ be an arbitrary stationary measure supported on $\Hbf \times \T^d$. 
By the parabolic smoothing (see e.g. \eqref{ineq:locRegu}) and stationarity, $\tilde\mu$ is also supported on $H^{\sigma+\delta}$ for $0<\delta < \alpha - \frac{d}{2} - \sigma$. 
Therefore, there exists a $C > 0$ such that 
\begin{align}
\tilde\mu ( \set{\norm{u}_{H^{\sigma+\delta}} \leq C }\times \T^d) > \frac{1}{2}. 
\end{align} 
Denote the set $\mathcal{B} = \set{u \in \Hbf: \norm{u}_{H^{\sigma+\delta}} \leq C }\times \T^d \subset \Hbf \times \T^d$. 
The stability argument applied in Lemma \ref{lem:MarkovTrans} (with $g \equiv 0$) 
gives the desired uniform decay:  for all $\gamma$, there exists a $T_\gamma$ such that for all $(u,x) \in \mathcal{B}$,  
\begin{align}
\PP\left( (u_{T_\gamma},x_{T_\gamma}) \in B_\gamma(0) \times B_\gamma(x^\prime) | (u_0,x_0) = (u,x) \right) > 0.    
\end{align}
Next, it follows from Lemma \ref{lem:MarkovTrans} that for $\gamma'$ sufficiently small, there exists a $\gamma$ (depending only $\gamma'$) such that for any $x' \in \T^d$, and all $(u,x) \in \mathcal{B}$
\begin{align}
\PP\left( (u_{T_\gamma+1},x_{T_\gamma+1}) \in B_{\gamma'}(0) \times B_{\gamma'}(x')  | (u_0,x_0) = (u,x) \right) > 0.  
\end{align}
Since this implies that
\[\tilde\mu(B_{\gamma'}(0) \times B_{\gamma'}(x')) \geq \int_{\mathcal{B}} P_{T_{\gamma}+1}( (u,x), B_{\gamma'}(0) \times B_{\gamma'}(x')) \tilde\mu(\dee u,\dx) > 0,
\] 
it follows that $(0,x')$ is in the support of the stationary measure. 
\end{proof}

Next, in order to complete the proof of Theorem \ref{thm:Lyap} in the case of Systems \ref{sys:NSE}--\ref{sys:3DNSE}, it suffices to prove the following, which shows that arbitrarily large gradient growth can be obtained on the unit time interval. 
\begin{proposition} \label{prop:Actrl}
For all $M > 0$ and $\eps > 0$, 
\begin{align}
\PP\left( (u_1,x_1,A_1) \in B_\eps(0) \times B_\eps(0) \times \set{A \in SL_{d}(\Real) : \abs{A} > M} | (u_0,x_0,A_0) = (0,0,\Id)  \right) > 0.  \label{ref:Q1ctrl}
\end{align}	
Together with Lemma \ref{lem:MarkovTrans}, this implies that Systems \ref{sys:NSE}--\ref{sys:3DNSE} satisfy Definition \ref{defn:condCprimeRDS} and hence Proposition \ref{prop:approxControlRuleOut} applies and the proof of Theorem \ref{thm:Lyap} is completed. 
\end{proposition}
\begin{proof} 
The control step is proved as in Lemma \ref{lem:Ctrluxv}, except now we apply the cellular flow translated so that the hyperbolic point is at the origin: 
\begin{align}
u(t) = f_{+}
\begin{pmatrix} 
\sin (y_2-b) \\ 
\sin (y_1-a)
\end{pmatrix}
\end{align}
with $\int_0^1 f_+(s) ds = \log M$. Then, set $g$ analogous to the choices in Lemma \ref{lem:Ctrluxv} (the size of $g$ now depends on $M$). 
The stability step proceeds as in Lemma \ref{lem:MarkovTrans}. 
\end{proof}

\begin{remark} \label{rmk:2Dstokes} 
All of the above controllability arguments also apply to the System \ref{sys:2DStokes} in $\T^2$ with only the condition: 
$\mathcal{K}$ symmetric and $(1,0),(0,1) \in \mathcal{K}$. 
This condition is not enough to guarantee that the $(u_t,x_t,A_t)$ process satisfies H\"ormander's condition. We can still verify Definition \ref{defn:condCprimeRDS} in this case, and hence it is sufficient to deduce Theorem \ref{thm:Lyap}. The claim in Remark \ref{rmk:StokesDegn} follows. 
Further, our arguments on Navier-Stokes similarly apply to the System \ref{sys:2DStokes} in $\T^d$ with infinitely many modes forced, under Assumption \ref{a:Highs}.  
\end{remark}

\begin{remark} \label{rmk:HighOnly}
For Systems \ref{sys:NSE}--\ref{sys:3DNSE}, using higher frequency shear flows and cellular flows, one can make all the same arguments in this section if we only take Assumption \ref{a:Highs}. Hence, by also Remark \ref{rmk:NorrisMal}, we can prove Theorem \ref{thm:Lyap} (and all our other results) for Systems \ref{sys:NSE}--\ref{sys:3DNSE} using \emph{only} Assumption \ref{a:Highs}. 
\end{remark}

%% Application to turbulence
\section{Applications to scalar turbulence}\label{sec:turbulence}
%!TEX root = master.tex

In this section we prove Theorem \ref{thm:turb}. 
First, we prove the weak anomalous dissipation property \eqref{eq:WAD}, Theorem \ref{thm:turb}, part (i). 
For this, we adapt the compactness-contradiction method of \cite{BCZGH}. 
Hence, it is easiest to begin by defining $f^\kappa = \sqrt{\kappa}g$ as in \eqref{eq:frenorms} and recall the re-scaled balance relation \eqref{eq:BalRenorm}.  
Next, we are interested in studying the limits of stationary measures $\bar{\mu}^\kappa$ to the problem \eqref{eq:frenorms} coupled with any of Systems \ref{sys:2DStokes}--\ref{sys:3DNSE}. 
It is standard that this (one-way) coupled system is well-posed in the sense of Proposition \ref{prop:WP} and defines an $\mathcal{F}_t$-adapted, Feller Markov process; see e.g. \cite{KS}. 
Similarly, the Krylov-Bogoliubov method implies the following: 
\begin{lemma} \label{lem:KryBog}
For all $\kappa > 0$, $\exists$ a stationary probability measure $\bar{\mu}^{\kappa}$ for the Markov process $(u_t,f_t^\kappa)$ supported on $\Hbf \times H^1$. 
Furthermore, the measure satisfies the following for all $p \geq 2$ (with implicit constant independent of $\kappa$),
\begin{align}\label{eq:BalRenorm-restate}
\int_{\Hbf \times H^1} \norm{\grad f}_{L^2}^2 d \bar{\mu}^\kappa(u,f) & = \bar{\eps} \\ 
\int_{\Hbf \times H^1} \norm{f}_{L^2}^p d \bar{\mu}^\kappa(u,f) & \lesssim_p \bar{\eps}^{p/2}. 
\end{align}
\end{lemma}
The following lemma is a straightforward adaptation of arguments in \cite{BCZGH,KS,K04}. Unlike in \cite{BCZGH}, the velocity field is not bounded a.s., however, the situation is not significantly different (using Proposition \ref{prop:WP}); indeed the original arguments of Kuksin \cite{K04} were specifically on the Navier-Stokes equations (see also \cite{KS,KS04}). 
\begin{lemma} \label{lem:compact}
Let $\{\bar{\mu}^\kappa\}_{\kappa > 0}$ be a family of stationary probability measure of the problem \eqref{eq:frenorms} as in Lemma \ref{lem:KryBog}, indexed by the diffusivity parameter $\kappa$, and $(u_t)$ given by one of Systems \ref{sys:2DStokes}--\ref{sys:3DNSE}. 
Then, the measures $\{\bar{\mu}^\kappa\}_{\kappa > 0}$ are tight on $\Hbf\times L^2$  as $\kappa \to 0$ and the subsequential weak limit $\bar{\mu}^0$ is a stationary measure of the inviscid problem \eqref{eq:SclNokap} with $\mu(A) = \bar{\mu}^0(A\times H^1)$ and $\bar{\mu}^0$ satisfies 
\begin{align}
\int_{\Hbf \times H^1} \norm{\grad f}^2_{L^2} d\bar{\mu}^0(u,f) & \leq \bar{\eps} \\
\int_{\Hbf \times H^1} \norm{f}_{L^2}^p d \bar{\mu}^0(u,f) & \lesssim_p \bar{\eps}^{p/2}. \label{ineq:Mompbd}
\end{align} 
\end{lemma} 
\begin{proof} 
Tightness follows from \eqref{eq:BalRenorm-restate} (and the corresponding balance on $u$) and Prokorov's theorem. 
The estimates follow from \eqref{eq:BalRenorm-restate} and lower semicontinuity. 
Finally, that $\bar{\mu}^0$ is a stationary measure of the inviscid problem \eqref{eq:SclNokap} follows as in the corresponding statements in \cite{BCZGH,K04} and is omitted for the sake of brevity. 
\end{proof}
Analogous to the arguments in \cite{BCZGH}, we deduce that necessarily $\bar{\mu}^0 = \mu \times \delta_0$ via Theorem \ref{thm:ExpGrwth}.
\begin{corollary} \label{cor:ufblow}
The only stationary measure for the process $(u_t,f_t^0)$ is the measure $\mu \times \delta_0$.
\end{corollary}
\begin{proof}
Let us use the notation $f_{t,u,f}$ to denote the scalar process $f_t^0$ associated with initial conditions $(u_0,f_0) = (u,f) \in \Hbf\times H^1$. Let $\bar \mu$ be any ergodic stationary measure for the process; by stationarity we have 
\[
\E \int_{\Hbf \times H^1} \bigg( \int_{\T^d} | \nabla f_{t,u,f}|^2 \, dx \bigg) d \bar \mu(u, f) = 
\int_{\Hbf \times H^1} \bigg( \int_{\T^d} | \nabla f|^2 dx \bigg) d \bar \mu(u, f)
\]
at all times $t \geq 0$. On the other hand, if $\bar \mu$ is not of the form $\mu \times \delta_0$
then by Theorem \ref{thm:ExpGrwth} there is a positive $\bar \mu$-measure set $\mathcal A \subset \Hbf \times H^1 \setminus \{ 0 \}$
with the property that for all $(u, f) \in \mathcal A$, we have 
$\E ( \int_{\T^d} | \nabla f_{t,u,f} |^2 \, dx ) \to \infty$ as $t \to \infty$.
This implies a contradiction. 
\end{proof} 

\begin{proof}[\textbf{Theorem \ref{thm:turb}, part (i)}]
 Follows from Lemma \ref{lem:compact} together with Corollary \ref{cor:ufblow} and \eqref{ineq:Mompbd} (with $p > 2$). 
\end{proof} 

Next, a variant of arguments in \cite{BCZPSW18} gives Yaglom's law \eqref{eq:Yaglom}. 

\begin{proof}[\textbf{Proof of Theorem \ref{thm:turb}, part (ii)}]
To adapt the arguments of \cite{BCZPSW18} the first step is to derive the analogue of the K\'arm\'an-Howarth-Monin relation \cite{deKarman1938,MoninYaglom,Frisch1995} for the passive scalar. In what follows $u$ and $g$ denote statistically stationary solutions to \eqref{eq:Sclkap}. Define the scalar two point correlation
\begin{align}
\mathfrak{G}(y) = \EE \fint_{\T^d} g(x) g(x+y) \dx 
\end{align}
and the vector
\begin{align}
\mathfrak{D}(y) = \EE \fint_{\T^d} \abs{\delta_{y} g(x)}^2  \delta_y u dx. 
\end{align}
Similarly, denote the two point covariance of the noise 
\begin{align}
\mathfrak{a}(y) = \frac{1}{2}\sum_{k \in \Z_0^d} \fint_{\T^d} \abs{\tilde{q}_k}^2 e_k(x) \otimes e_k(x+y) \dx,
\end{align}
Note that $\mathfrak{a}(0) = \bar{\eps}$. 
The KHM relation is the manifestation of the $L^2$ balance on the two point correlation $\mathfrak{G}$; it is significantly simpler for scalars than for the 3D Navier-Stokes equations. 
Hence, the proof is omitted for brevity; see \cite{BCZPSW18} for details. 
\begin{proposition}[Scalar KHM relation] \label{prop:KHM}
Let $(u_t,g_t)$ be a statistically stationary solution to \eqref{eq:Sclkap} coupled to one of Systems \ref{sys:2DStokes}--\ref{sys:3DNSE}. 
Then, for any $\eta = \eta(y)$ a smooth, compactly supported test function, there holds  
\begin{align}
\frac{1}{2} \int_{\Real^d} \grad\eta(y) \cdot \mathfrak{D}(y) \dy = 2\kappa \int_{\Real^d} \Delta \eta(y)\mathfrak{G}(y) \dy + 2 \int_{\Real^d} \eta(y) \mathfrak{a}(y) \dy. \label{eq:KHM}
\end{align}
\end{proposition}

Define (suppressing the time-dependence as anyway, the time-dependence vanishes after expectations due to stationarity), 
\begin{align}
\bar{\mathfrak{D}}(\ell) = \EE \fint_{\T^d} \fint_{\S^{d-1}} |\delta_{\ell n} g|^2 \delta_{\ell n} u\cdot n \, \dee S(n) \dee x. 
\end{align}
Equipped with Proposition \ref{prop:KHM}, we may proceed as in \cite{BCZPSW18}  by testing \eqref{eq:KHM} with a radially symmetric test function $\eta(h) = \phi(\abs{h})$. 
Hence, we obtain the following ODE for $S$ in the weak form
\begin{align}
\partial_{\ell}\left(\ell^2 \bar{\mathfrak{D}} \right) = -\ell^2 \left( 4\kappa \bar{\mathfrak{G}}'' + 4\kappa \frac{d-1}{\ell} \bar{\mathfrak{G}}' + 4\bar{\mathfrak{a}}\right), \label{eq:ODE}
\end{align}

where we denote the spherically averaged  quantities
\begin{align}
\bar{\mathfrak{G}}(\ell) & = \fint_{\S^{d-1}} \mathfrak{G}(\ell n) \dee S(n)\\ 
\bar{\mathfrak{a}}(\ell) & = \fint_{\S^{d-1}} \mathfrak{a}(\ell n) \dee S(n). 
\end{align}
From here, the proof proceeds as in the proof of the 4/3 law in \cite{BCZPSW18}. Specifically, one first integrates \eqref{eq:ODE}. Then, the weak anomalous dissiption \eqref{eq:WAD} is used to eliminate the contributions involving $\kappa$ as $\kappa \rightarrow 0$ over an appropriate range of scales $[\ell_D,\ell_I]$ with $\lim_{\kappa \rightarrow 0}\ell_D = 0$. 
Finally, regularity of $\bar{\mathfrak{a}}(\ell)$ near $\ell = 0$ is used to deduce that the resulting estimate for $\bar{\mathfrak{D}}(\ell)/\ell$ is asymptotically $-\frac{4}{3} \bar{\eps}$ as $\ell_I \rightarrow 0$. 
\end{proof}

%% Appendix
\appendix
\section{Appendix}\label{sec:Appendix}
%!TEX root = master.tex

\subsection{Well-posedness and the RDS framework}\label{subsec:wellPosedApp}

In this section we will confirm that the various processes considered in this paper, e.g., the Eulerian
process $(u_t)$ and the Lagrangian process $(u_t, x_t)$, arise as random dynamical systems
in the framework of Section \ref{sec:RDS}.

To start, without loss of generality, we may regard our probability space $\Omega$
as in Section \ref{sec:Intro} as a countable product of canonical spaces $\big( C([0,\infty), \R) \big)^{\otimes \N}$
 with the product topology; likewise, $\mathcal F$ is the corresponding Borel sigma algebra and $\P$
 the countable product of Weiner measures. 
 
 \medskip
 
For each of Systems \ref{sys:2DStokes}--\ref{sys:3DNSE}, we follow the standard procedure
of defining the $(u_t)$ process to be a solution of the corresponding equation in the mild sense \cite{KS,DPZ96}, i.e., 
\begin{align}
u_t = e^{-tA}u_0 + \Gamma_t + \int_0^t e^{-(t-s)A} B(u_s,u_s) ds \, ,  \label{eq:Mild} 
\end{align}
where $\Gamma_t = \int_0^t e^{-(t-s)A} Q dW(s)$ is the pertinent stochastic convolution for our additive noise.
in System \ref{sys:NSE}. For \eqref{eq:Mild} we have the following well-posedness theorem.
 \begin{proposition}[\cite{KS,DPZ96}] \label{prop:WPapp}
For each of Systems \ref{sys:2DStokes}--\ref{sys:3DNSE}, we have the following. For $\P$-almost every $\omega \in \Omega$;
 all $u_0 \in \hat{\Hbf} \cap H^\gamma$ with $\gamma < \alpha-\frac{d}{2}$; and all $T> 0, p \geq 1$, we have that
there exists a unique solution $(u_t)$ to \eqref{eq:Mild}. Moreover, the process $(u_t)$ is $\mathcal{F}_t$-adapted,
with $u \in L^p(\Omega;C([0,T];\Hbf \cap H^\gamma)) \cap L^2(\Omega;L^2(0,T;H^{\gamma+(d-1)}))$. 

Additionally, 
\begin{itemize}
\item[(i)] For all $p \geq 1$ and $\gamma < \gamma' < \alpha - \frac{d}{2}$, 
\begin{align}
\EE \sup_{t \in [0,T]} \norm{u_t}_{H^\gamma}^p & \lesssim_{T,p,\gamma} 1 + \norm{u_0}_{\Hbf \cap H^\gamma}^p \\
\EE \int_0^T \norm{u_s}_{H^{\gamma + (d-1)}}^2 ds & \lesssim_{T,\delta} 1 + \norm{u_0}^2_{H^\gamma} \\ 
\EE \sup_{t \in [0,T]} \left(t^{\frac{\gamma'-\gamma}{2(d-1)}} \norm{u_t}_{H^{\gamma'}}\right)^p &\lesssim_{p,T,\gamma,\gamma'} 1 + \norm{u_0}^p_{H^\gamma} \label{ineq:locRegu}
\end{align}
\item[(ii)] Suppose for $\gamma,\delta > 0$ arbitrary satisfying $\gamma + \delta < \alpha-\frac{d}{2}$, there holds
\begin{align*}
& \lim_{n \rightarrow \infty}\norm{Q W_{n} - QW}_{L^\infty(0,T;H^{\gamma+\delta})} = 0 \\ 
& \lim_{n \rightarrow \infty}\norm{u_0^{(n)} - u_0}_{H^{\gamma}} = 0. 
\end{align*}
Then, the corresponding solutions $u^{(n)}_t$ satisfy $\lim_{n \rightarrow \infty}\norm{u_t^{(n)} - u_t}_{L^\infty(0,T;H^{\gamma})} = 0$. 
Moreover, this  convergence is uniform over bounded sets, e.g. $\norm{u_0}_{H^\gamma} \leq C$  and $\norm{QW}_{L^\infty(0,T;H^{\gamma + \delta})} \leq C$ for $C < \infty$. 
\end{itemize} 
\end{proposition}
\begin{proof}
Item (i) is a consequence of standard arguments (see, e.g., \cite{KS}) combined with the following estimates on the stochastic convolution $\Gamma_t$:
\begin{lemma} \label{lem:StochConv}
Let $\Gamma_t = \int_0^t e^{-(t-s)A} Q dW(s)$. Then $\forall T > 0$, $p \in [1,\infty)$, and $\gamma < \alpha + \frac{d}{2} - 1$,
\begin{align}
\EE \sup_{t \in [0,T]} \norm{\Gamma_t}_{H^{\gamma}}^p & \lesssim_{p,T,\gamma} 1 \\ 
\EE \int_0^T \norm{\Gamma_t}_{H^{\gamma+(d-1)}}^2 & \lesssim_{p,T,\gamma} 1. 
\end{align}
\end{lemma}
Lemma \ref{lem:StochConv} follows from the Factorization Lemma, the Burkholder-Davis-Gundy Inequality, and the smoothing properties of the heat semigroup (see, e.g., \cite{DPZ96}). 

Proposition \ref{prop:WPapp}(ii) can be proved by essentially the same stability argument as that in
 the proof of Lemma \ref{lem:MarkovTrans}, to which we refer the reader for details.
 \end{proof}

Let $\mathcal U : [0,\infty) \times \Omega \times \hat \Hbf \to \hat \Hbf, (t, \omega, u) \mapsto \mathcal U^t_\omega(u)$
denote the mapping sending, for a given $t \geq 0$ and $\P$-generic $\omega \in \Omega$, a given $u \in \hat \Hbf$
to the time-$t$ vector field $u_t$ conditioned on $u_0 = u$. We conclude from Proposition \ref{prop:WPapp} that 
$\mathcal U$ is a continuous RDS in the sense of Section \ref{subsubsec:basicSetupRDS} on the space $Z = \hat \Hbf$
satisfying condition (H1). Similarly, the random ODE \eqref{eq:xtintro} defining the auxiliary process $x_t = \phi^t_{\omega, u_0} x_0$
is well-posed, and we conclude as before that the corresponding mapping $\Theta : [0,\infty) \times \Omega \times \hat \Hbf
\times \T^d \to \hat \Hbf \times \T^d$ for the Lagrangian flow process $(u_t, x_t)$ 
is a continuous RDS satisfying (H1) on the space $Z = \hat \Hbf \times \T^d$.
We leave it to the reader to confirm that the same is true for each of the processes $(u_t, x_t, v_t)$ and $(u_t, x_t, \check v_t)$
on $Z = \hat \Hbf \times \T^d \times P^{d-1}$ and $(u_t, x_t, A_t)$ on $Z = \hat \Hbf \times \T^d \times SL_d(\R)$, defined
by the random ODE in \eqref{eq:defineSDE}.

In addition, in this paper we consider the linear cocycles 
$\mathcal A, \check {\mathcal A} : [0, \infty) \times \Omega \times \hat \Hbf \times \T^d \to M_{d \times d}(\R)$ defined
by $\mathcal A^t_{\omega, u, x} = D_x \phi^t_{\omega, u}$ and 
$\check {\mathcal A}^t_{\omega, u, x} = (\mathcal A^t_{\omega, u, x})^{- \top}$.
The integrability condition (H2) in Section \ref{subsubsec:METRDS} for each of these processes follows 
from \eqref{ineq:locRegu} above, while the independent increments condition (H3) is equivalent to
condition (H1) for the $(u_t, x_t, A_t)$ process.

\subsection{H\"older Estimates and Interpolation Inequalities}

The following interpolation Lemma is very useful:
\begin{lemma}\label{lem:interp-eq}
Let $f$ be a $C^1$ function on $[0,1]$ and let $\alpha \in (0,1]$. Then the following inequality holds for all $t \in (0,1)$
\[
	\|\partial_t f\|_{L^\infty([0,t])} \leq \frac{4}{t} \|f\|_{L^\infty([0,t
	])}^{\frac{\alpha}{\alpha+1}}\max\Big\{\|f\|_{L^\infty([0,t])}^{\frac{1}{1+\alpha}}, [\partial_t f]_{C^\alpha([0,t])}^{\frac{1}{1+\alpha}}\Big\},
\]
where $[\,\cdot\,]_{C^\alpha([0,t])}$ denotes the $\alpha$-H\"{o}lder semi norm on $[0,t]$.
\end{lemma}

The following estimate on the H\"{o}lder norms of a process in a general Hilbert space is also useful for verifying the H\"{o}lder assumption used in the proof of the non-degeneracy of the Malliavin matrix
\begin{lemma}\label{lem:holder-ests}
Let $\cH$ and $\cW$ be separable Hilbert spaces and let $Y_t$, $t\in [0,1]$ be an $\cH$ valued process given by
\[
	Y_t = Y_0 + \int_0^t B_s\ds + \int_0^t Q_s\dee W_s,
\]
where $W_t$ is a cylindrical Wiener process on $\cW$, and $B_t$,$Q_t$ are predictable processes taking values in $\cH$ and $\cL^2(\cW,\cH)$, the space of bounded Hilbert-Schmidt operators from $\cW$ to $\cH$. Assume that $B_t$ and $Q_t$ satisfy, for every $p\geq 1$
\[
	\E\left(\|B\|_{L^\infty([0,1];\cH)}^p + \|Q\|_{L^\infty([0,1];\cL^2(\cW, \cH))}^p\right) < \infty,
\]
then for every $p > 12$, we have the estimate
\[
	\E \|Y\|_{C^{1/3}([0,1];\cH)}^p \leqc_p \E\left(\|B\|_{L^\infty([0,1];\cH)}^p + \|Q\|_{L^\infty([0,1];\cL^2(\cW, \cH))}^p\right).
\]
\end{lemma}

%% Bibliography
\bibliographystyle{abbrv}
\bibliography{bibliography}

\end{document}